%% file: BouligandLevenbergMarquardt.tex
\pdfoutput=1
\documentclass[english]{jnsao}

\usepackage[utf8]{inputenc}
\usepackage[nameinlink,capitalize]{cleveref}

\usepackage{enumitem} 
\usepackage{graphicx}
\usepackage{booktabs}
\usepackage[exponent-product=\cdot]{siunitx}

\renewcommand{\epsilon}{\varepsilon}
\newcommand{\N}{\mathbb{N}}
\newcommand{\R}{\mathbb{R}}
\newcommand{\1}{\mathbb{1}}
\newcommand{\Linop}{\mathbb{L}}

\DeclareMathOperator{\diag}{\mathrm{diag}}

\DeclareMathOperator{\argmin}{\mathrm{argmin}}

\newcommand{\norm}[1]{\|#1\|}
\renewcommand{\epsilon}{\varepsilon}

\newlist{assumption}{enumerate}{1}
\setlist[assumption]{label=(\textsc{a}\arabic*)}

\crefname{assumptioni}{Assumption}{Assumptions}

\usepackage{pgfplots}
\pgfplotsset{compat=newest}
\pgfplotsset{plot coordinates/math parser=false}
\pgfplotsset{every tick label/.append style={font=\footnotesize}}

\def\thetitle{Bouligand--Levenberg--Marquardt iteration for a non-smooth ill-posed inverse problem} 
\title{\thetitle}
\author{Christian Clason\thanks{Faculty of Mathematics, University of Duisburg-Essen, Thea-Leymann-Strasse 9, 45127 Essen, Germany (\email{christian.clason@uni-due.de}, \orcid{0000-0002-9948-842})}
    \and Vu Huu Nhu\thanks{Faculty of Mathematics, University of Duisburg-Essen, Thea-Leymann-Strasse 9, 45127 Essen, Germany (\email{huu.vu@uni-due.de}, \orcid{0000-0003-4279-3937})}
}
\acknowledgements{This work was supported by the DFG under the grants CL 487/2-1 and RO 2462/6-1, both within the priority programme SPP 1962 ``Non-smooth and Complementarity-based Distributed Parameter Systems: Simulation and Hierarchical Optimization''.}
\shortauthor{Clason, Nhu}
\date{2019-06-23}
\manuscriptlicense{}
\manuscriptcopyright{}
\manuscripteprinttype{arXiv}
\manuscripteprint{1902.10596}

\begin{document}
\maketitle

\begin{abstract}
    In this paper, we consider a modified Levenberg--Marquardt method for solving an ill-posed inverse problem where the forward mapping is not Gâteaux differentiable. By relaxing the standard assumptions for the classical smooth setting, we derive asymptotic stability estimates that are then used to prove the convergence of the proposed method. This method can be applied to an inverse source problem for a non-smooth semilinear elliptic PDE where a Bouligand subdifferential can be used in place of the non-existing Fréchet derivative, and we show that the corresponding \emph{Bouligand--Levenberg--Marquardt iteration} is an iterative regularization scheme. Numerical examples illustrate the advantage over the corresponding Bouligand--Landweber iteration.
\end{abstract}

\section{Introduction}

We consider the inverse problems of the form
\begin{equation}
    \label{eq:inverse-pro}
    F(u) = y^\delta, 
\end{equation}
where $F: D(F) \subset U \to Y$ is a \emph{non-smooth} (i.e., not necessarily Gâteaux differentiable) nonlinear operator between Hilbert spaces and the available data $y^\delta$ are some approximations of the corresponding true data $y^\dag := F(u^\dag)$. Furthermore, $D(F)$ denotes the domain of $F$ and $u^\dag$ is the unknown true solution that needs to be reconstructed. 

A typical example of problem \eqref{eq:inverse-pro} is the case where $U = Y := L^2(\Omega)$ and $F$ is the solution operator of the non-smooth semilinear elliptic equation
\begin{equation}
    \label{eq:maxpde-intro}
    \begin{aligned}
        -\Delta y + \max(0,y) &= u \quad\text{in }\Omega, \quad y \in H^1_0(\Omega)
    \end{aligned}
\end{equation}
with $u \in L^2(\Omega)$ and a bounded domain $\Omega\subset\R^d$, $d \in \{2,3\}$. In this case, $F$ is not G\^{a}teaux differentiable at $u^\dag$ if the set of $x \in \Omega$ such that $y^\dag (x) = 0$ has positive measure; see \cite[Prop.~3.4]{ClasonNhu2018}. Moreover, $F$ is completely continuous (see \cite[Lem.~3.2]{ClasonNhu2018}), and \eqref{eq:inverse-pro} is therefore \emph{ill-posed} in the sense that the solution to \eqref{eq:inverse-pro} does not depend continuously
on the data. A stable solution of \eqref{eq:inverse-pro} thus needs regularization techniques. Here we consider \emph{iterative} regularization techniques, which construct a sequence $\{u_n^\delta\}$ of approximations to $u^\dag$ and ensure stability by early stopping at an iteration index $N(\delta,y^\delta)$ chosen, e.g., according to Morozov's discrepancy principle; see, e.g., \cite{Engl1996,Kaltenbacher2008}. Iterative methods have the advantage over variational methods such as Tikhonov regularization that the selection of the regularization parameter (in this case, the stopping index) is part of the method and does not have to be performed by, e.g., checking a sequence of candidates or using additional information on the smoothness of the forward operator that is often not available. An iterative regularization method for \eqref{eq:maxpde-intro} of Landweber type (which can be interpreted as a generalized gradient descent) was proposed and analyzed in \cite{ClasonNhu2018}. However, like any first-order scheme, it usually requires a large number of iterations to satisfy the discrepancy principle, especially for small noise. This motivates considering iterative regularization methods of Newton type.

Recall that the Newton method for the smooth version of \eqref{eq:inverse-pro} with a continuously Fréchet differentiable operator $F$ reads as
\begin{equation*}
    F'(u_n^\delta)(u_{n+1}^\delta - u_n^\delta) = y^\delta - F(u_n^\delta),
\end{equation*}
where $F'(u):U\to Y$ denotes the Gâteaux derivative of $F$ at $u\in D(F)$.
However, if \eqref{eq:inverse-pro} is ill-posed, this equation is generally ill-posed as well and needs to be regularized.
Applying Tikhonov regularization leads to the \emph{Levenberg--Marquardt method}
\begin{equation}\label{eq:LM-method-var}
    u_{n+1}^\delta = \argmin\limits_{u \in D(F)} \norm{F'(u_n^\delta)(u - u_n^\delta) - y^\delta - F(u_n^\delta)}_Y^2 + \alpha_n \norm{u - u_n^\delta}_U^2
\end{equation}
or, equivalently,
\begin{equation}
    \label{eq:LM-method}
    u_{n+1}^\delta =u_n^\delta + \left(\alpha_n I + F'(u_n^\delta)^*F'(u_n^\delta) \right)^{-1}F'(u_n^\delta)^* \left(y^\delta - F(u_n^\delta)\right),
\end{equation}
where $\alpha_n >0$ is the Tikhonov parameter. For a linear operator, this method coincides with the non-stationary iterated Tikhonov method studied in, e.g., \cite{Brill87,HankeGroetsch1998}.
As noted above, for noisy data the iteration has to be terminated at a stopping index $N_\delta :=N(\delta,y^\delta)<\infty$ in order to be stable. Assuming that $\norm{y^\delta -y^\dag}_Y\leq \delta$ and that the Tikhonov parameters $\alpha_n$ are chosen via a Morozov discrepancy principle, \cite{Hanke1997} showed the \emph{regularization property} $u_{N_\delta}^\delta \to u^\dag$ as $\delta\to 0$ as well as the logarithmic estimate 
\begin{equation}
    \label{eq:log-esti}
    N_\delta = O\left(1 + |\log(\delta)| \right),
\end{equation}
provided that
\begin{equation}
    \label{eq:strong-gtcc}
    \norm{F(u_1) - F(u_2) - F'(u_2)(u_1-u_2)}_Y \leq c \norm{u_1-u_2}_U \norm{F(u_1)-F(u_2)}_Y 
\end{equation}
for all $u_1, u_2 \in \overline B_U(u^\dag ,\rho)$ and for some constants $c, \rho >0$. 
In \cite{Jin2010}, the regularization property as well as the logarithmic estimate \eqref{eq:log-esti} of the Levenberg--Marquardt method was shown under the a priori choice
\begin{equation}
    \label{eq:Tikhonov-choice}
    \alpha_n = \alpha_0 r^n, \quad n = 0,1,\ldots
\end{equation}
with $\alpha_0 >0$ and $r \in (0,1)$ and under the assumption that for any $u_1,u_2 \in \overline B_U(u^\dag ,\rho)$, there exists a bounded linear operator $Q(u_1,u_2): Y \to Y$ satisfying
\begin{equation}
    \label{eq:strong-gtcc2}
    F'(u_1) = Q(u_1,u_2)F'(u_2) \quad \text {and} \quad \norm{I - Q(u_1, u_2)}_{\Linop(Y)} \leq L \norm{u_1 - u_2}_U
\end{equation}
for some constant $L>0$. 
It is noted that the convergence analysis in \cite{Hanke1997,Jin2010} requires the \emph{stability} of the method, that is, there holds
\begin{equation*}
    u_{n}^\delta \to u_n \quad \text {as } \delta \to 0 \quad\text{for all } n \leq N(\delta, y^\delta) 
\end{equation*}
with $\delta$ small enough, where $u_n^\delta$ and $u_n$ are generated by the method corresponding to the noisy $(\delta >0)$ and the noise-free $(\delta=0)$ situations, respectively. The continuity of the derivative $F'$ (or more specifically, of the linear operator in the right-hand side of \eqref{eq:LM-method}) with respect to $u$ is therefore essential.

The purpose of this work is to present a \emph{modified Levenberg--Marquardt method} for solving \eqref{eq:inverse-pro} in the spirit of \cite{Scherzer1995,ClasonNhu2018}, where we replace the -- possibly nonexistent -- Fréchet derivative $F'(u)$ in \eqref{eq:LM-method} by another suitable bounded linear operator $G_u$. Our main aim is to show the regularization property of the proposed algorithm under the choice \eqref{eq:Tikhonov-choice} of Tikhonov parameters and conditions that relax \eqref{eq:strong-gtcc} and \eqref{eq:strong-gtcc2}. We also prove the logarithmic estimate \eqref{eq:log-esti} of the stopping index. 
However, unlike the situation in \cite{Hanke1997,Jin2010}, we lack the continuity of the mapping $ D(F) \ni u \mapsto G_u \in \Linop(U,Y)$. 
To overcome this essential difficulty, we shall combine a technique from \cite{Jin2010} with the approach in \cite{ClasonNhu2018} to prove \emph{asymptotic stability estimates} of iterates $u_n^\delta$; see \cref{sec:asym-stab} and \cref{prop:AS} in place of the missing stability of the method.
The proposed method is then applied to a non-smooth ill-posed inverse problem where the forward operator is the solution mapping of \eqref{eq:maxpde-intro}. In this case, the operator $G_u$ can be taken from the Bouligand subdifferential of the forward mapping and explicitly characterized by the solution of a suitable linearized PDE, see ~\cref{prop:Gu} below. We refer to this special case of the modified Levenberg--Marquardt method as \emph{Bouligand--Levenberg--Marquardt iteration}.

Let us briefly comment on related literature.
Newton-type methods, and in particular the Levenberg--Marquardt method, for approximately solving smooth nonlinear ill-posed problems have been extensively investigated in Hilbert spaces; see, e.g. \cite{Engl1996,Hanke1997,Kaltenbacher2008,Scherzer2011,LechleiterRieder2010,Rieder1999,Rieder2001,Rieder2005} and the references therein. 
More recently, inverse problems in Banach spaces have attracted increasing attention, and corresponding iterative regularization methods of Newton-type have been developed, e.g., in
\cite{Jin12,Jin13,Jin15,Jin16,Margotti14,MargottiRieder2015,Kaltenbacher09,Kaltenbacher10}. Considering \eqref{eq:LM-method-var} in Banach spaces (in particular, $L^1$ or the space of functions of bounded variation) or including additional constraints can lead to non-smooth optimization problems; however, none of the works so far has focused on inverse problems for non-smooth forward operators.

\paragraph*{Organization.} This paper is organized as follows. After briefly summarizing basic notation, 
we present the convergence analysis of the modified Levenberg--Marquardt method in \cref{sec:LMmethod}: \cref{sec:wellposed} is devoted to its well-posedness and the logarithmic estimate of the stopping index $N_\delta$; in \cref{sec:convergence-free}, we prove its convergence in the noise-free case; in \cref{sec:asym-stab} we verify its asymptotic stability estimates, which are crucial for the proof of the regularization property of the iterative method in \cref{sec:regularization}. \cref{sec:maxpde} introduces an application of the modified Levenberg--Marquardt method to the non-smooth ill-posed inverse source problem for \eqref{eq:maxpde-intro}. Finally, some numerical examples are provided in \cref{sec:num}.

\paragraph*{Notation.} For a Hilbert space $X$, we denote by $(\cdot,\cdot)_{X}$ and $\|\cdot\|_{X}$, respectively, the inner product and the norm on $X$. For a given $z$ belonging to a Banach space $Z$ and $\rho>0$, by $B_Z(z,\rho)$ and $\overline{B}_Z(z, \rho)$ we denote, respectively, the open and closed balls in $Z$ of radius $\rho$ centered at $z$. For each measurable function $u$ on $\Omega$ and a subset $T \subset \R$, the notation $\{u \in T \}$ stands for the sets of almost every $x \in \Omega$ at which $u(x)\in T$. Similarly, given measurable functions $u,v$ on $\Omega$ and subsets $T_1, T_2 \subset \R$, we denote the set of a.e. $x \in \Omega$ such that $u(x) \in T_1$ and $v(x) \in T_2$ by $\{ u \in T_1, v \in T_2 \}$. For a measurable set $S$ in $\R^d$, we write $|S|$ for the $d$-dimensional Lebesgue measure of $S$ and denote by $\1_S$ the characteristic function of the set $S$, i.e., $\1_S(s) = 1$ if $s \in S$ and $\1_S(s) = 0$ if $s \notin S$. The adjoint operator, the null space, and the range of a linear operator $G$ will be denoted by $G^*$, $\mathcal{N}(G)$, and $\mathcal{R}(G)$, respectively. Finally, we denote by $\Linop(X)$ and $\Linop(X,Y)$ the
set of all bounded linear operators from Hilbert space $X$ to itself and from $X$ to another Hilbert space $Y$, respectively.

\section{A modified Levenberg--Marquardt method} \label{sec:LMmethod} 

Let $U$ and $Y$ be real Hilbert spaces and $F$ a non-smooth mapping from $U$ to $Y$ with its domain $D(F) \subset U$. 
We consider the non-smooth ill-posed problem 
\begin{equation}
    \label{eq:inv-pro} 
    F(u) = y^\delta, 
\end{equation}
where the noisy data $y^\delta$ satisfy
\begin{equation}
    \label{eq:noise} 
    \| y^\delta - y^\dag \|_{Y} \leq \delta
\end{equation}
with $y^\dag \in \mathcal R(F)$. 
From now on, let $u^\dag$ be an arbitrary, but fixed, solution of \eqref{eq:inv-pro} corresponding to the exact data $y^\dag $. For a given number $\rho >0$, we denote by $S_\rho(u^\dag)$ the set of all solutions in $\overline B_U(u^\dag, \rho)$ of \eqref{eq:inv-pro} corresponding the exact data, that is,
\begin{equation*}
    S_\rho(u^\dag) := \left\{ u^* \in D(F): F(u^*) = y^\dag, \norm{u^* - u^\dag}_U \leq \rho \right \}.
\end{equation*}

Throughout this work, we make the following assumptions on $F$.
\begin{assumption}
\item\label{ass:gtcc} 
    There exists a constant $\rho_0 >0$ such that $\overline B_U(u^\dag , \rho_0) \subset D(F)$. Furthermore, there exists a family of bounded linear operators $\{G_u: u \in \overline B_U(u^\dag,\rho)\} \subset \Linop(U,Y)$  such that for all $\rho \in (0,\rho_0]$ and $u,\hat u$ in $\overline B_U(u^\dag,\rho)$, there holds the \emph{generalized tangential cone condition}
    \begin{equation}
        \label{eq:GTCC}
        \|F(\hat u) - F(u) - G_{u}(\hat u - u) \|_{Y} \leq \eta(\rho) \|F(\hat u) - F(u) \|_{Y}
        \tag{GTCC}
    \end{equation}
    for some non-decreasing function $\eta: (0,\rho_0] \to (0,\infty)$ satisfying
    \begin{equation}
        \label{eq:upper_bounded_GTCC}
        \eta_0 := \eta(\rho_0) < 1.
    \end{equation} 
    Moreover, for any pair $u_1, u_2 \in \overline B_U(u^\dag,\rho_0)$, there exists a bounded linear operator $Q(u_1,u_2) \in \Linop(Y)$ such that
    \begin{equation}
        \label{eq:Q-trans}
        G_{u_1} = Q(u_1,u_2)G_{u_2}
    \end{equation} 
    and that for all $\rho\in (0,\rho_0]$,
    \begin{equation}
        \label{eq:Q-identity}
        \norm{I -Q(u,\hat u)}_{\Linop(Y)} \leq \kappa(\rho)\quad \text{for all } u,\hat u \in \overline B_U(u^\dag,\rho)
    \end{equation}
    and for some non-negative and non-decreasing function $\kappa$ on $(0,\rho_0]$ with $\kappa_0:=\kappa(\rho_0)$.

\item\label{ass:adj_compact} The operator $G_{u^\dag}: U \to Y$ is compact. 
\end{assumption}

We will furthermore require that $\kappa(\rho)$ and $\eta(\rho)$ can be made sufficiently small by choosing $\rho$ small enough; we will make this more precise during the analysis in this section. In \cref{sec:ver-gtcc}, we will verify that this is possible for the Bouligand--Levenberg--Marquardt method applied to the non-smooth PDE \eqref{eq:maxpde-intro}.

\begin{remark}\label{rem:gtcc}
    If \eqref{eq:GTCC} is valid, then for all $\rho \in (0,\rho_0]$ and $u, \hat u \in \overline B_U(u^\dag,\rho)$, there hold
    \begin{align}
        \label{eq:GTCC2}
        \|F(\hat u) - F(u) - G_{u}(\hat u - u) \|_{Y} &\leq \frac{\eta(\rho)}{1 - \eta_0} \| G_{u}(\hat u - u) \|_{Y}
        \shortintertext{and}
        \label{eq:GTCC3}
        \|F(\hat u) - F(u) \|_{Y} &\leq \frac{1}{1 - \eta_0} \| G_{u}(\hat u - u) \|_{Y},
    \end{align}
    where the last inequality leads to the continuity of $F$ at $u$ and hence on $\overline B_U(u^\dag,\rho)$. Moreover, if \eqref{eq:Q-identity} holds, then it holds that
    \begin{equation}
        \label{eq:Q_upper_bound}
        \norm{Q(u_1,u_2)}_{\Linop(Y)} \leq \kappa_0 + 1
    \end{equation}
    for all $u_1, u_2 \in \overline B_U(u^\dag, \rho_0)$. 
    In addition, if \cref{ass:gtcc,ass:adj_compact} are satisfied, then all of \cref{ass:gtcc,ass:adj_compact} are fulfilled with $F$ and $G_u$ replaced, respectively, by $tF$ and $tG_u$ for some positive number $t$ small enough.
\end{remark}

We shall consider the \emph{modified Levenberg--Marquardt method} defined as
\begin{equation}
    \label{eq:LM-iteration}
    u_{n+1}^\delta = u_n^\delta + \left(\alpha_n I + G_{u_n^\delta}^*G_{u_n^\delta}\right)^{-1}G_{u_n^\delta}^*\left(y^\delta - F(u_n^\delta) \right), \quad n = 0,1, \ldots,
\end{equation}
where $u_0^\delta := u_0$ and $\{\alpha_n \}$ is given by
\begin{equation}
    \label{eq:Lag-para}
    \alpha_n = \alpha_0 r^n, \quad n = 0,1, \ldots
\end{equation}
for some constants $\alpha_0 >0$ and $r \in (0,1)$. We will assume that $\alpha_0$ is chosen such that
\begin{equation}
    \label{eq:scaled}
    \norm{G_{u^\dag}}_{\Linop(U,Y)} \leq \alpha_0^{1/2}.
\end{equation} 
Note that according to \cref{rem:gtcc}, this condition can be enforced for any $\alpha_0>0$ by scaling the problem \eqref{eq:inverse-pro} as well as $G_u$ accordingly.

The iteration is terminated via the \emph{discrepancy principle}
\begin{equation}
    \label{eq:disc2}
    \| y^\delta - F(u_{N_\delta}^\delta) \|_{Y} \leq \tau \delta < \| y^\delta - F(u_{n}^\delta) \|_{Y} \quad \text{for all} \ 0 \leq n < N_\delta,
\end{equation}
where $\tau >1$ is a given number.
Here $N_\delta :=N(\delta, y^\delta)$ stands for the stopping index of the iterative method.

By $\{u_n\}$, we denote the sequence of iterates defined by \eqref{eq:LM-iteration} corresponding to the noise free case $(\delta = 0)$, i.e.,
\begin{equation}
    \label{eq:LM-iteration-free}
    u_{n+1} = u_n + \left(\alpha_n I + G_{u_n}^*G_{u_n}\right)^{-1}G_{u_n}^*\left(y^\dag - F(u_n) \right), \quad n = 0,1, \ldots
\end{equation} 
For ease of exposition, from now on, we use the notations
\begin{align*}
    e_n^\delta &:= u_n^\delta - u^\dag, & \quad G_n^\delta &:= G_{u_n^\delta}, & \quad A_n^\delta &:= G_n^{\delta*}G_n^{\delta}, & \quad B_n^\delta &:= G_n^{\delta}G_n^{\delta*}, \\
    e_n &:= u_n - u^\dag,& \quad G_n &:= G_{u_n},& \quad A_n &:= G_n^*G_n, & \quad B_n &:= G_nG_n^*, \\
    & & \quad G_\dag &:= G_{u^\dag},& \quad A &:= G_{\dag}^*G_{\dag},& \quad B &:= G_{\dag}G_{\dag}^*.&&
\end{align*}

\begin{remark}
    If $F$ is smooth and if $G_u = F'(u)$, then \eqref{eq:LM-iteration} coincides with the classical Levenberg--Marquardt method; see, e.g., \cite{Kaltenbacher2008,Scherzer2011,Hanke1997,Jin2010}.
\end{remark}

\subsection{Well-posedness} \label{sec:wellposed}

We first show the well-posedness of the proposed iterative method as well as the logarithmic estimate of the stopping index $N_\delta$. 

The first lemma gives a useful tool to estimate the difference between iterates.
\begin{lemma}[{cf. \cite[Lem.~2]{Jin2010}}] \label{lem:transform}
    Assume that \eqref{eq:Q-trans} and \eqref{eq:Q-identity} are fulfilled. Let $\rho \in (0,\rho_0]$. For any $u, v \in \overline B_U(u^\dag , \rho)$, let $A_u := G_u^*G_u$ and $A_v := G_v^*G_v$. Then for any $\alpha >0$, the following identities hold
    \begin{align}
        \label{eq:trans1}
        \left(\alpha I + A_u \right)^{-1} - \left(\alpha I + A_v \right)^{-1} &= \left(\alpha I + A_v \right)^{-1}G_v^*R_\alpha(u,v)
        \shortintertext{and}
        \label{eq:trans2}
        \left(\alpha I + A_u \right)^{-1}G_u^* - \left(\alpha I + A_v \right)^{-1}G_v^* &= \left(\alpha I + A_v \right)^{-1}G_v^*S_\alpha(u,v),
    \end{align}
    where $R_\alpha(u,v): U \to Y$ and $S_\alpha(u,v): Y \to Y$ are bounded linear operators satisfying
    \begin{equation}
        \label{eq:RS-esti}
        \left \|R_\alpha(u,v) \right \|_{\Linop(U,Y)} \leq \alpha^{-1/2}\kappa(\rho) \quad \text {and} \quad \left \|S_\alpha(u,v) \right \|_{\Linop(Y)} \leq 3\kappa(\rho).
    \end{equation}
\end{lemma}
\begin{proof}
    The proof is analogous to that in \cite{Jin2010}. We see from \eqref{eq:Q-trans} that
    \begin{multline*}
        \left(\alpha I + A_u \right)^{-1} - \left(\alpha I + A_v \right)^{-1}\\
        \begin{aligned}
            & = \left(\alpha I + A_v \right)^{-1}(A_v -A_u) \left(\alpha I + A_u \right)^{-1}\\
            & = \left(\alpha I + A_v \right)^{-1}\left(G_v^*(G_v-G_u) + (G_v^* - G_u^*)G_u \right) \left(\alpha I + A_u \right)^{-1}\\
            & = \left(\alpha I + A_v \right)^{-1}\left(G_v^*(Q(v,u)-I)G_u + G_v^*(I - Q(u,v)^*)G_u \right) \left(\alpha I + A_u \right)^{-1}\\
            & = \left(\alpha I + A_v \right)^{-1}G_v^*\left((Q(v,u)-I) + (I - Q(u,v)^*) \right) G_u\left(\alpha I + A_u \right)^{-1}\\
            & = \left(\alpha I + A_v \right)^{-1}G_v^*R_\alpha(u,v)
        \end{aligned}
    \end{multline*}
    with
    \begin{equation*}
        R_\alpha(u,v) := \left((Q(v,u)-I) + (I - Q(u,v)^*) \right) G_u\left(\alpha I + A_u \right)^{-1}.
    \end{equation*}
    Easily, \eqref{eq:Q-identity} and \cref{lem:spectral-half} ensure that $R_\alpha(u,v)$ satisfies the first inequality in \eqref{eq:RS-esti}. On the other hand, using \eqref{eq:Q-trans} and the above representation yields 
    \begin{equation*}
        \begin{aligned}
            \left(\alpha I + A_u \right)^{-1}G_u^* - \left(\alpha I + A_v \right)^{-1}G_v^*
            & = \left(\alpha I + A_v \right)^{-1}(G_u^* - G_v^*) + \left[ \left(\alpha I + A_u \right)^{-1} - \left(\alpha I + A_v \right)^{-1}\right] G_u^*\\
            & = \left(\alpha I + A_v \right)^{-1} G_v^*\left(Q(u,v)^* - I \right) + \left(\alpha I + A_v \right)^{-1}G_v^*R_\alpha(u,v)G_u^* \\
            & = \left(\alpha I + A_v \right)^{-1}G_v^*S_\alpha(u,v)
        \end{aligned}
    \end{equation*}
    with 
    \begin{equation*}
        \begin{aligned}
            S_\alpha(u,v) & := Q(u,v)^* - I + R_\alpha(u,v)G_u^* \\
            & = Q(u,v)^* - I + \left((Q(v,u)-I) + (I - Q(u,v)^*) \right) \left(\alpha I + B_u \right)^{-1}B_u
        \end{aligned}
    \end{equation*}
    and $B_u := G_uG_u^*$. From the definition of $S_\alpha$, \eqref{eq:Q-identity} and \cref{lem:spectral-nu} lead to the second inequality in \eqref{eq:RS-esti}.
\end{proof}

To simplify the notation in the following proofs, we introduce the constants
\begin{equation}
    \label{eq:constants-c}
    c_0 := \frac{1}{\sqrt{r}}, \quad c_1 := \frac{1}{1-\sqrt{r}}, \quad c_2 := \frac{1}{\sqrt{r}(1- \sqrt{r})}, \quad c_3 := \frac{\sqrt{r}}{1-r}, \quad c_4 := \frac{1}{1-r},
\end{equation}
as well as
\begin{equation}
    \label{eq:constants-K} 
    K_0(r, \nu) := \frac{1}{\sqrt{r} \left(r^{\nu - 1/2} -1 \right)}, \quad K_1(r, \nu) := \frac{1}{r \left(r^{\nu - 1/2} -1 \right)}
\end{equation}
for $0 \leq \nu < \frac{1}{2}$.

Let now $\tilde{N}_\delta\in \N$ be such that
\begin{equation}
    \label{eq:stop-index2}
    \alpha_{\tilde{N}_\delta } \leq \left(\frac{\delta}{\gamma_0 \norm{e_0}_{U} }\right)^2 < \alpha_n \quad \text {for all } 0 \leq n < \tilde{N}_\delta,
\end{equation}
for a constant 
\begin{equation*}
    \gamma_0 > \frac{2c_0}{(1-\eta_0)(\tau - \tau_0)}
\end{equation*}
with $\tau > \tau_0 > 1$. We can now prove a logarithmic estimate for $\tilde N_\delta$, which will later be used to obtain the corresponding estimate for the actual stopping index $N_\delta$.
\begin{lemma} \label{lem:loga-esti}
    Let $\tilde{N}_\delta$ be defined by \eqref{eq:stop-index2}. Then there holds
    \begin{equation*}
        \tilde{N}_\delta = O(1 + |\log(\delta)|).
    \end{equation*}
\end{lemma}
\begin{proof}
    From \eqref{eq:stop-index2} and \eqref{eq:Lag-para}, we conclude that
    \begin{equation*}
        \left(\frac{\delta}{\gamma_0 \norm{e_0}_{U} }\right)^2 < \alpha_0 r^{\tilde{N}_\delta - 1},
    \end{equation*}
    which, together with the fact that $0 < r <1$, directly gives
    \begin{equation*}
        \tilde{N}_\delta < 1 + 2 \log_r \delta - 2 \log_r \left(\gamma_0 \| e_0 \|_U \right) - \log_r \alpha_0
    \end{equation*}
    and hence the desired estimate.
\end{proof}

We now show a uniform bound on the iterates and the error by, if necessary, further restricting the radius $\rho$ of the neighborhood of $u^\dag$.
\begin{lemma}[\protect{cf.~\cite[Lem.~4]{Jin2010}}] \label{lem:well-posedness}
    Let $\{\alpha_n\}$ be defined by \eqref{eq:Lag-para} and \eqref{eq:scaled}. Assume that \cref{ass:gtcc} holds. 
    Assume further that there exists a positive constant $\rho_1 \leq \rho_0$ such that
    \begin{equation} \label{eq:rho1_cond}
        \left\{
            \begin{aligned}
                2\kappa(\rho_1) + \frac{c_0 }{1 - \eta_0}(1 + 3 \kappa_0)\eta(\rho_1) &\leq \frac{c_0}{2c_2}  \\
                (c_1 +3)\kappa(\rho_1) + \frac{2c_2 }{1 - \eta_0}(1 + 3 \kappa_0)\eta(\rho_1) &\leq 1.
            \end{aligned}
        \right.
    \end{equation}
    Let $\rho \in (0,\rho_1]$ be arbitrary and let $u_0 \in U$ be such that $(2+ c_1\gamma_0) \norm{e_0}_U < \rho$. 
    Then there hold
    \begin{enumerate}[label=(\roman*)]
        \item\label{it:well-posed:1}  $u_n^\delta \in \overline B_U(u^\dag,\rho)$;
        \item \label{it:well-posed:2} $\norm{e_n^\delta}_U \leq \left(2+c_1\gamma_0\right) \norm{e_0}_U < \rho$;
        \item \label{it:well-posed:3} $\norm{G_\dag e_n^\delta}_{Y} \leq \left(c_0+2c_2\gamma_0\right) \norm{e_0}_U \alpha_n^{1/2}$
    \end{enumerate} 
    for all $0 \leq n \leq \tilde{N}_\delta$, where the constants $c_i$, $i=0,1,2$, are given by \eqref{eq:constants-c}.
\end{lemma}
\begin{proof}
    It is sufficient to show \ref{it:well-posed:2} and \ref{it:well-posed:3} by induction on $n$ with $0 \leq n \leq \tilde{N}_\delta$. Obviously, \ref{it:well-posed:2} and \ref{it:well-posed:3} are fulfilled with $n=0$. Now for any fixed $0 \leq l < \tilde{N}_\delta$, we assume that \ref{it:well-posed:2} and \ref{it:well-posed:3} hold true for all $0 \leq n \leq l$. We shall prove these assertions for $n = l+1$. To this end, we set for any $0 \leq m \leq l$
    \begin{equation}
        \label{eq:z-delta}
        z_m^\delta := F(u_m^\delta) - y^\dag - G_m^\delta e_m^\delta.
    \end{equation} 
    Moreover, we see from \eqref{eq:LM-iteration} and the identity $I = \alpha (\alpha I + T)^{-1} + (\alpha I + T)^{-1}T$ that
    \begin{equation*}
        \begin{aligned}
            e_{m+1}^\delta & = u_{m+1}^\delta - u^\dag \\
            & = e_m^\delta + \left(\alpha_n I + A_m^\delta \right)^{-1} G_m^{\delta*}\left(y^\delta - F(u_m^\delta) \right) \\
            & = \alpha_m \left(\alpha_m I + A_m^\delta \right)^{-1} e_m^\delta + \left(\alpha_m I + A_m^\delta \right)^{-1} A_m^\delta e_m^\delta + \left(\alpha_m I + A_m^\delta \right)^{-1} G_m^{\delta*}\left(y^\delta - F(u_m^\delta) \right) \\
            & = \alpha_m \left(\alpha_m I + A_m^\delta \right)^{-1} e_m^\delta + \left(\alpha_m I + A_m^\delta \right)^{-1} G_m^{\delta*}\left(y^\delta - F(u_m^\delta) + G_m^\delta e_m^\delta \right) \\
            & = \alpha_m \left(\alpha_m I + A_m^\delta \right)^{-1} e_m^\delta + \left(\alpha_m I + A_m^\delta \right)^{-1} G_m^{\delta*}\left(y^\delta - y^\dag -z_m^\delta \right),
        \end{aligned}
    \end{equation*}
    which together with \cref{lem:transform} gives
    \begin{equation*}
        \begin{aligned}
            e_{m+1}^\delta & = \alpha_m \left(\alpha_m I + A \right)^{-1} \left[ I + G_\dag^* R_{\alpha_m}(u_m^\delta, u^\dag) \right] e_m^\delta \\
            \MoveEqLeft[-1] + \left(\alpha_m I + A \right)^{-1} G_\dag^*\left[ I + S_{\alpha_m}(u_m^\delta, u^\dag) \right]\left(y^\delta - y^\dag -z_m^\delta \right).
        \end{aligned}
    \end{equation*}
    Consequently, it holds that
    \begin{equation}
        \label{eq:error-repr}
        e_{m+1}^\delta = \alpha_m \left(\alpha_m I + A \right)^{-1} e_m^\delta + \left(\alpha_m I + A \right)^{-1} G_\dag^* w_m^\delta 
    \end{equation}
    with 
    \begin{equation}
        \label{eq:w-define}
        w_m^\delta := \alpha_m R_{\alpha_m}(u_m^\delta, u^\dag) e_m^\delta + \left[ I + S_{\alpha_m}(u_m^\delta, u^\dag)\right] \left(y^\delta - y^\dag -z_m^\delta \right).
    \end{equation}
    The definition of $w_m^\delta$ and the estimates \eqref{eq:RS-esti} imply that
    \begin{equation*}
        \norm{w_m^\delta }_Y \leq \kappa(\rho) \alpha_m^{1/2} \norm{e_m^\delta}_U + \left(1 + 3\kappa(\rho)\right)\left(\delta + \norm{z_m^\delta }_Y\right). 
    \end{equation*}
    Furthermore, \eqref{eq:GTCC} and \eqref{eq:GTCC3} give 
    \begin{equation}
        \label{eq:z-esti}
        \norm{z_m^\delta }_Y \leq \frac{\eta(\rho)}{1 -\eta_0} \norm{G_\dag e_m^\delta}_Y.
    \end{equation}
    We thus have
    \begin{equation}
        \label{eq:w-esti}
        \norm{w_m^\delta }_Y \leq \kappa(\rho) \alpha_m^{1/2} \norm{e_m^\delta}_U + \left(1 + 3\kappa(\rho)\right)\left(\delta + \frac{\eta(\rho)}{1 -\eta_0 } \norm{G_\dag e_m^\delta}_Y\right). 
    \end{equation}
    By telescoping \eqref{eq:error-repr}, we obtain
    \begin{equation}
        \label{eq:error-plus1}
        e_{l+1}^\delta = \prod_{m=0}^{l} \alpha_m \left(\alpha_m I + A \right)^{-1} e_0 + \sum_{m=0}^{l} \alpha_m^{-1} \prod_{j=m}^{l} \alpha_j \left(\alpha_j I + A \right)^{-1} G_\dag^*w_m^\delta 
    \end{equation}
    and thus
    \begin{equation}
        \label{eq:resi-plus1}
        G_\dag e_{l+1}^\delta = \prod_{m=0}^{l} \alpha_m \left(\alpha_m I + B \right)^{-1} G_\dag e_0 + \sum_{m=0}^{l} \alpha_m^{-1} \prod_{j=m}^{l} \alpha_j \left(\alpha_j I + B\right)^{-1} B w_m^\delta,
    \end{equation}
    where we have used that identity $G_\dag \left(\alpha I + A \right)^{-1} = \left(\alpha I + B \right)^{-1} G_\dag$.
    Applying \cref{lem:spectral-nu,lem:spectral-half} to \eqref{eq:error-plus1} and using \eqref{eq:w-esti} yields
    \begin{equation*}
        \begin{aligned}
            \norm{e_{l+1}^\delta}_U & \leq \norm{e_0}_U + \frac{1}{2} \sum_{m=0}^{l} \alpha_m^{-1} \left(\sum_{j=m}^l \alpha_j^{-1} \right)^{-1/2} \norm{w_m^\delta }_Y \\ 
            & \leq \norm{e_0}_U + \frac{1}{2} \sum_{m=0}^{l} \alpha_m^{-1} \left(\sum_{j=m}^l \alpha_j^{-1} \right)^{-1/2} \left[\kappa(\rho) \alpha_m^{1/2} \norm{e_m^\delta}_U + \left(1 + 3\kappa(\rho)\right)\left(\delta + \frac{\eta(\rho)}{1 -\eta_0 } \norm{G_\dag e_m^\delta}_Y\right) \right]. 
        \end{aligned}
    \end{equation*}
    We now use the induction hypothesis to deduce that
    \begin{equation*}
        \begin{aligned}
            \norm{e_{l+1}^\delta}_U & \leq \norm{e_0}_U + \frac{1}{2} \sum_{m=0}^{l} \alpha_m^{-1} \left(\sum_{j=m}^l \alpha_j^{-1} \right)^{-1/2} \kappa(\rho) \alpha_m^{1/2} \left(2 + c_1\gamma_0 \right) \norm{e_0}_U \\
            \MoveEqLeft[-1] + \frac{1}{2} \sum_{m=0}^{l} \alpha_m^{-1} \left(\sum_{j=m}^l \alpha_j^{-1} \right)^{-1/2} \left[ \left(1 + 3\kappa(\rho)\right)\left(\delta + \frac{\eta(\rho)}{1 -\eta_0 } \left(c_0 + 2c_2\gamma_0 \right) \norm{e_0}_U \alpha_m^{1/2} \right) \right].
        \end{aligned}
    \end{equation*}
    From the choice of $\tilde{N}_\delta$, there holds that $\delta \leq \gamma_0 \norm{e_0}_U \alpha_m^{1/2}$ for all $0 \leq m \leq l < \tilde{N}_\delta$. The above estimates and \cref{lem:sum-esti} imply that
    \begin{equation*}
        \begin{aligned}
            \norm{e_{l+1}^\delta}_U & \leq \norm{e_0}_U + \frac{1}{2} \sum_{m=0}^{l} \alpha_m^{-1/2} \left(\sum_{j=m}^l \alpha_j^{-1} \right)^{-1/2} \kappa(\rho) \left(2 + c_1\gamma_0 \right) \norm{e_0}_U \\
            \MoveEqLeft[-1] + \frac{1}{2} \sum_{m=0}^{l} \alpha_m^{-1/2} \left(\sum_{j=m}^l \alpha_j^{-1} \right)^{-1/2} \left[ \left(1 + 3\kappa(\rho)\right)\left(\gamma_0 + \frac{\eta(\rho)}{1 -\eta_0 } \left(c_0 + 2c_2\gamma_0 \right) \right) \right] \norm{e_0}_U \\
            & \leq \norm{e_0}_U + \frac{1}{2} c_1 \norm{e_0}_U \left[ \kappa(\rho) \left(2 + c_1\gamma_0 \right) + \left(1 + 3\kappa(\rho)\right)\left(\gamma_0 + \frac{\eta(\rho)}{1 -\eta_0 } \left(c_0 + 2c_2\gamma_0 \right) \right) \right] \\
            & = \norm{e_0}_U \left[ 1 + \frac{c_1}{2} H_1(\rho) + \frac{c_1}{2} \gamma_0(1 + H_2(\rho)) \right]
        \end{aligned}
    \end{equation*}
    with
    \begin{equation*}
        H_1(\rho) := 2\kappa(\rho) + (1+3\kappa(\rho))\frac{c_0\eta(\rho)}{1-\eta_0}  , \quad H_2(\rho) := (c_1+3)\kappa(\rho) +\frac{2c_2\eta(\rho)}{1-\eta_0}(1 + 3\kappa(\rho)).
    \end{equation*}
    Combining this with the monotonic growth of $\kappa,\eta$ on $(0,\rho_0]$, the fact that $c_0/(2c_2) \leq 2/c_1$, and \eqref{eq:rho1_cond} yields
    \begin{equation}
        \label{eq:error-esti}
        \norm{e_{l+1}^\delta}_U \leq \norm{e_0}_U \left(2 + c_1\gamma_0 \right).
    \end{equation}
    On the other hand, \eqref{eq:resi-plus1} along with \cref{lem:spectral-half,lem:spectral-nu} gives
    \begin{equation*}
        \begin{aligned}
            \norm{G_\dag e_{l+1}^\delta}_Y & \leq \frac{1}{2} \left(\sum_{j=0}^{l} \alpha_j^{-1} \right)^{-1/2} \norm{e_0}_U + \sum_{m=0}^{l} \alpha_m^{-1} \left(\sum_{j=m}^l \alpha_j^{-1} \right)^{-1} \norm{w_m^\delta }_Y.
        \end{aligned}
    \end{equation*}
    From this, \cref{lem:sum-esti}, and \eqref{eq:w-esti}, the induction hypothesis and the choice of $\tilde{N}_\delta$ satisfying \eqref{eq:stop-index2} lead to
    \begin{equation*}
        \begin{aligned}
            \norm{G_\dag e_{l+1}^\delta}_Y 
            & \leq \frac{1}{2} c_0 \alpha_{l+1}^{1/2} \norm{e_0}_U + \norm{e_0}_U \sum_{m=0}^{l} \alpha_m^{-1/2} \left(\sum_{j=m}^l \alpha_j^{-1} \right)^{-1} \\
            \MoveEqLeft[-1]\cdot\left[ \kappa(\rho)(2+c_1\gamma_0) + \left(1 + 3\kappa(\rho)\right)\left(\gamma_0 + \frac{\eta(\rho)}{1 -\eta_0 } \left(c_0 + 2c_2\gamma_0 \right) \right) \right]\\
            & \leq \norm{e_0}_U \alpha_{l+1}^{1/2} \left\{ \frac{1}{2} c_0 + c_2 \left[ \kappa(\rho)(2+c_1\gamma_0) + \left(1 + 3\kappa(\rho)\right)\left(\gamma_0 + \frac{\eta(\rho)}{1 -\eta_0 } \left(c_0 + 2c_2\gamma_0 \right) \right) \right] \right\} \\
            & \leq \norm{e_0}_U \alpha_{l+1}^{1/2} \left[ \frac{1}{2} c_0 + c_2 H_1(\rho) + c_2\gamma_0 \left(1 + H_2(\rho) \right) \right].
        \end{aligned}
    \end{equation*}
    From this, the monotonic growth of $\kappa, \eta$ on $(0, \rho_0]$ and \eqref{eq:rho1_cond}, we have that
    \begin{equation*}
        \norm{G_\dag e_{l+1}^\delta}_Y \leq \left(c_0 + 2c_2 \gamma_0 \right)\norm{e_0}_U \alpha_{l+1}^{1/2},
    \end{equation*}
    which together with \eqref{eq:error-esti} implies that \ref{it:well-posed:2} and \ref{it:well-posed:3} are fulfilled with $n = l+1$. 
\end{proof}
By using a similar argument for $\{u_n\}$ defined by \eqref{eq:LM-iteration-free}, we obtain the following result.
\begin{lemma}[\protect{cf.~\cite[Lem.~5]{Jin2010}}] \label{lem:well-posedness-free}
    Let $\{\alpha_n\}$ be defined by \eqref{eq:Lag-para} and \eqref{eq:scaled}. Assume that \cref{ass:gtcc} and the first condition in \eqref{eq:rho1_cond} hold. Then, for any $u_0 \in U$ such that $2\norm{e_0}_U < \rho$ with $\rho \in (0,\rho_1]$, there hold
    \begin{enumerate}[label=(\roman*)]
        \item  $u_n \in \overline B_U(u^\dag,\rho)$;
        \item  $\norm{e_n}_U \leq 2 \norm{e_0}_U < \rho$;
        \item  $\norm{G_\dag e_n}_{Y} \leq c_0 \norm{e_0}_U \alpha_n^{1/2}$
    \end{enumerate} 
    for all $n \geq 0$. 
\end{lemma}

The next lemma is a crucial tool in our analysis to prove the well-posedness of the method as well as the asymptotic stability estimates.
\begin{lemma}
    \label{lem:asym-stab}
    Assume that all assumptions in \cref{lem:well-posedness} are satisfied. Then
    \begin{equation}
        \label{eq:asym-stab-repr}
        u_{n+1}^\delta - u_{n+1} = \alpha_n \left(\alpha_n I + A \right)^{-1}(u_n^\delta - u_n) + \sum_{i=1}^5 s_n^{(i)} \quad \text{for all } 0 \leq n < \tilde{N}_\delta,
    \end{equation}
    where
    \begin{equation}
        \label{eq:s-repr}
        s_n^{(i)} = \left(\alpha_k I + A \right)^{-1} G_\dag^{*}\xi_n^{(i)}, \quad i= 1,\ldots,5
    \end{equation}
    with some $\xi_n^{(i)} \in Y$ satisfying
    \begin{equation}
        \label{eq:xi14-esti}
        \sum_{i=1}^4 \norm{\xi_n^{(i)}}_Y \leq L_1(\rho) \left(\alpha_{n}^{ 1/2}\norm{e_n}_U + \alpha_{n}^{ 1/2}\norm{u_n^\delta - u_n}_U + \norm{G_\dag e_n}_Y + \norm{G_\dag (u_n^\delta - u_n)}_Y \right)
    \end{equation}
    and
    \begin{equation}
        \label{eq:xi5-esti}
        \xi_n^{(5)} = (y^\delta -y^\dag) + \tilde{\xi}_n^{(5)} \quad \text {with} \quad \norm{\tilde \xi_n^{(5)}}_Y \leq L_1(\rho)\delta
    \end{equation}
    for 
    \begin{equation}
        \label{eq:L1_func}
        L_1(\rho) := \max \left\{3 \kappa(\rho), (1+3\kappa_0)\left(2 + \frac{3 \eta_0}{1- \eta_0} \right) \kappa(\rho), \frac{(1+3\kappa_0)(1+\kappa_0)}{1- \eta_0} \eta(\rho)  \right\}. 
    \end{equation}
\end{lemma}
\begin{proof}
    For any $m \in \N$, we define
    \begin{equation}
        \label{eq:z-free}
        z_m := F(u_m) - y^\dag - G_m e_m.
    \end{equation} 
    We thus obtain from \eqref{eq:LM-iteration}, \eqref{eq:LM-iteration-free}, \eqref{eq:z-delta}, and \eqref{eq:z-free} that
    \begin{equation*}
        \begin{aligned}
            u_{n+1}^\delta - u_{n+1} &= u_{n}^\delta - u_{n} + \left(\alpha_n I + A_n^\delta \right)^{-1} G_n^{\delta*}\left(y^\delta - F(u_n^\delta) \right) - \left(\alpha_n I + A_n \right)^{-1} G_n^{*}\left(y^\dag - F(u_n) \right) \\
            &= u_{n}^\delta - u_{n} + \left(\alpha_n I + A_n^\delta \right)^{-1} G_n^{\delta*}\left(y^\delta - y^\dag - z_n^\delta - G_n^\delta e_n^\delta \right) \\
            \MoveEqLeft[-1] + \left(\alpha_n I + A_n \right)^{-1} G_n^{*}\left(z_n + G_n e_n \right) \\
            &= u_{n}^\delta - u_{n} + \left(\alpha_n I + A_n^\delta \right)^{-1} G_n^{\delta*}\left(y^\delta - y^\dag \right) - \left(\alpha_n I + A_n^\delta \right)^{-1} G_n^{\delta*}z_n^\delta \\
            \MoveEqLeft[-1] + \left(\alpha_n I + A_n \right)^{-1} G_n^{*} z_n + \left[ \left(\alpha_n I + A_n \right)^{-1} A_n e_n - \left(\alpha_n I + A_n^\delta \right)^{-1} A_n^\delta e_n^\delta\right].
        \end{aligned}
    \end{equation*}
    Furthermore, using the identity $(\alpha I + T)^{-1}T = I - \alpha (\alpha I + T)^{-1}$ gives
    \begin{multline*}
        \left(\alpha_n I + A_n \right)^{-1} A_n e_n -\left(\alpha_n I + A_n^\delta \right)^{-1} A_n^\delta e_n^\delta \\
        \begin{aligned}
            & =\left(\alpha_n I + A_n \right)^{-1} A_n (u_n - u_n^\delta) + \left[ \left(\alpha_n I + A_n \right)^{-1} A_n - \left(\alpha_n I + A_n^\delta \right)^{-1} A_n^\delta \right] e_n^\delta \\
            & = \left(\alpha_n I + A \right)^{-1} A (u_n - u_n^\delta) + \alpha_n \left[ \left(\alpha_n I + A_n \right)^{-1} - \left(\alpha_n I + A \right)^{-1} \right](u_n^\delta - u_n) \\
            \MoveEqLeft[-1] + \alpha_n \left[ \left(\alpha_n I + A_n^\delta \right)^{-1}- \left(\alpha_n I + A_n \right)^{-1} \right] e_n^\delta 
        \end{aligned}
    \end{multline*}
    and thus
    \begin{multline*}
        \left(\alpha_n I + A_n \right)^{-1} A_n e_n -\left(\alpha_n I + A_n^\delta \right)^{-1} A_n^\delta e_n^\delta = \left(\alpha_n I + A \right)^{-1} A (u_n - u_n^\delta) \\
        \begin{aligned}
            & + \alpha_n \left[ \left(\alpha_n I + A_n^\delta \right)^{-1} - \left(\alpha_n I + A \right)^{-1} \right](u_n^\delta - u_n) + \alpha_n \left[ \left(\alpha_n I + A_n^\delta \right)^{-1} - \left(\alpha_n I + A_n \right)^{-1} \right]e_n.
        \end{aligned}
    \end{multline*}
    Defining 
    \begin{align*}
        &s_n^{(1)} := \alpha_n \left[ \left(\alpha_n I + A_n^\delta \right)^{-1} - \left(\alpha_n I + A \right)^{-1} \right](u_n^\delta - u_n),\\
        & s_n^{(2)} := \alpha_n \left[ \left(\alpha_n I + A_n^\delta \right)^{-1} - \left(\alpha_n I + A_n \right)^{-1} \right]e_n,\\
        & s_n^{(3)} := \left[ \left(\alpha_n I + A_n \right)^{-1} G_n^{*} - \left(\alpha_n I + A_n^\delta \right)^{-1} G_n^{\delta*}\right] z_n,\\
        & s_n^{(4)} := \left(\alpha_n I + A_n^\delta \right)^{-1} G_n^{\delta*} (z_n - z_n^\delta), \\
        & s_n^{(5)} := \left(\alpha_n I + A_n^\delta \right)^{-1} G_n^{\delta*} (y^\delta - y^\dag) 
    \end{align*} 
    yields \eqref{eq:asym-stab-repr}. 
    We now verify \eqref{eq:s-repr}, \eqref{eq:xi14-esti}, and \eqref{eq:xi5-esti}. To this end, we use \cref{lem:transform} and obtain that
    \begin{equation*}
        \begin{aligned}
            s_n^{(1)} & = \alpha_n \left(\alpha_n I + A \right)^{-1} G_\dag^* R_{\alpha_n}(u_n^\delta, u^\dag) (u_n^\delta - u_n),
        \end{aligned}
    \end{equation*}
    and so \eqref{eq:s-repr} holds for $i=1$ with
    \begin{equation*}
        \xi_n^{(1)} := \alpha_n R_{\alpha_n}(u_n^\delta, u^\dag)(u_n^\delta - u_n) .
    \end{equation*}
    Note that $u_n^\delta, u_n \in \overline{B}(u^\dag, \rho)$, according to \cref{lem:well-posedness,lem:well-posedness-free}. We thus deduce from \eqref{eq:RS-esti} that
    \begin{equation}
        \label{eq:xi1-esti}
        \begin{aligned}
            \norm{\xi_n^{(1)}}_Y &\leq \kappa(\rho) \alpha_n^{1/2} \norm{ u_n^\delta - u_n}_U.
        \end{aligned}
    \end{equation}
    Similarly, we obtain
    \begin{equation}
        \label{eq:xi25-define}
        \begin{cases}
            & \xi_n^{(2)} := \alpha_n \left[I + S_{\alpha_n}(u_n, u^\dag)\right] R_{\alpha_n}(u_n^\delta, u_n) e_n, \\
            & \xi_n^{(3)} :=- \left[I + S_{\alpha_n}(u_n, u^\dag)\right] S_{\alpha_n}(u_n^\delta, u_n) z_n, \\
            & \xi_n^{(4)} := \left[I + S_{\alpha_n}(u_n^\delta, u^\dag)\right] (z_n - z_n^\delta), \\
            & \xi_n^{(5)} := \left[I + S_{\alpha_n}(u_n^\delta, u^\dag)\right] (y^\delta - y^\dag)
        \end{cases}
    \end{equation}
    and \eqref{eq:s-repr} then follows. 
    Obviously, \eqref{eq:xi5-esti} is verified with $\tilde{\xi}_n^{(5)} := S_{\alpha_n}(u_n^\delta, u^\dag)(y^\delta - y^\dag)$. It remains to prove the estimate \eqref{eq:xi14-esti}. First, it is easy to see from \eqref{eq:RS-esti} and the definition of $\xi_n^{(2)}$ in \eqref{eq:xi25-define} that
    \begin{equation}
        \label{eq:xi2-esti}
        \norm{\xi_n^{(2)}}_Y \leq (1+ 3\kappa(\rho)) \kappa(\rho) \alpha_n^{1/2} \norm{ e_n}_U.
    \end{equation}
    As a result of \eqref{eq:GTCC}, we have
    \begin{equation*}
        \begin{aligned}
            \norm{z_n}_Y & \leq \frac{\eta(\rho)}{1-\eta_0} \norm{G_\dag e_n}_Y,
        \end{aligned}
    \end{equation*}
    which together with \eqref{eq:RS-esti} and the definition of $\xi_n^{(3)}$ yields
    \begin{equation}
        \label{eq:xi3-esti}
        \norm{\xi_n^{(3)}}_Y \leq \frac{3\kappa(\rho)\eta(\rho)(1+3\kappa(\rho))}{1-\eta_0} \norm{G_\dag e_n}_Y.
    \end{equation}
    Furthermore, we can conclude from the definitions of $z_n^\delta$ and $z_n$, \eqref{eq:GTCC}, \eqref{eq:Q-trans}, and \eqref{eq:Q-identity} that
    \begin{equation*}
        \begin{aligned}
            \norm{z_n^\delta -z_n}_Y& \leq \norm{F(u_n)-F(u_n^\delta) - G_n^\delta(u_n - u_n^\delta)}_Y + \norm{(G_n - G_n^\delta)e_n}Y \\
            & \leq \frac{\eta(\rho)}{1-\eta_0} \norm{G_{n}(u_n - u_n^\delta) }_Y + \norm{Q(u_n, u^\dag) - Q(u_n^\delta, u^\dag)}_{\Linop(Y)} \norm{G_\dag e_n}_Y \\
            & \leq \frac{\eta(\rho)}{1-\eta_0} \norm{Q(u_n,u^\dag)G_\dag(u_n - u_n^\delta) }_Y + \norm{Q(u_n, u^\dag) - Q(u_n^\delta, u^\dag)}_{\Linop(Y)} \norm{G_\dag e_n}_Y \\
            & \leq \frac{\eta(\rho)}{1-\eta_0}(1+ \kappa(\rho)) \norm{G_\dag(u_n - u_n^\delta) }_Y + 2\kappa(\rho) \norm{G_\dag e_n}_Y.
        \end{aligned}
    \end{equation*}
    This, the definition of $\xi_n^{(4)}$, and \eqref{eq:RS-esti} therefore imply that
    \begin{equation*}
        \norm{\xi_n^{(4)}}_Y \leq  (1 +3\kappa(\rho)) \left[\frac{1+ \kappa(\rho)}{1-\eta_0} \eta(\rho) \norm{G_\dag(u_n - u_n^\delta) }_Y + 2\kappa(\rho) \norm{G_\dag e_n}_Y\right]. 
    \end{equation*}
    From this, \eqref{eq:xi1-esti}, \eqref{eq:xi2-esti}, \eqref{eq:xi3-esti}, and the monotonic growth of $\kappa$ on $(0,\rho_0]$, we obtain \eqref{eq:xi14-esti}. 
\end{proof}

\begin{lemma} 
    \label{lem:asym-auxi}
    Let all assumptions of \cref{lem:asym-stab} be satisfied. Then there hold
    \begin{equation} \label{eq:asymp-esti-auxi}
        \norm{u^\delta_{l+1} - u_{l+1}}_U \leq \frac{1}{2} c_3 \left(1+ L_1(\rho)\right) \frac{\delta}{\sqrt{\alpha_{l+1}}} + \frac{L_1(\rho)}{2} \sum_{m=0}^{l} \alpha_m^{-1} \left(\sum_{j=m}^l \alpha_j^{-1} \right)^{-1/2} \sigma_m
    \end{equation}
    and
    \begin{equation} \label{eq:G-esti-auxi}
        \norm{G_\dag(u^\delta_{l+1} - u_{l+1})- y^\delta + y^\dag}_Y \leq \delta \left(1+ c_4L_1(\rho)\right) + L_1(\rho)\sum_{m=0}^{l} \alpha_m^{-1} \left(\sum_{j=m}^l \alpha_j^{-1} \right)^{-1} \sigma_m
    \end{equation}
    for all $0 \leq l < \tilde{N}_\delta$ with $L_1(\rho)$ defined as \eqref{eq:L1_func} and
    \begin{equation*}
        \sigma_m := \alpha_{m}^{ 1/2}\norm{e_m}_U +\alpha_{m}^{ 1/2}\norm{u_m^\delta - u_m}_U + \norm{G_\dag e_m}_Y + \norm{G_\dag (u_m^\delta - u_m)}_Y. 
    \end{equation*}
\end{lemma}
\begin{proof}
    Telescoping \eqref{eq:asym-stab-repr} and \eqref{eq:s-repr} gives
    \begin{equation}
        \label{eq:asym-stab-plus1-0}
        u^\delta_{l+1} - u_{l+1} = \sum_{m=0}^{l} \alpha_m^{-1} \prod_{j=m}^{l} \alpha_j \left(\alpha_j I + A \right)^{-1} G_\dag^* \sum_{i=1}^5 \xi_m^{(i)} 
    \end{equation}
    and thus
    \begin{multline}
        \label{eq:resi-asym-stab-plus1-0}
        G_\dag (u^\delta_{l+1} - u_{l+1}) - (y^\delta - y^\dag) \\
        \begin{aligned}[t]
            &= \sum_{m=0}^{l} \alpha_m^{-1} \prod_{j=m}^{l} \alpha_j \left(\alpha_j I + B\right)^{-1} B \left[ \sum_{i=1}^4 \xi_m^{(i)} + \tilde \xi_m^{(5)} \right] + \left[ I -\sum_{m=0}^{l} \alpha_m^{-1} \prod_{j=m}^{l} \alpha_j \left(\alpha_j I + B\right)^{-1} B \right](y^\dag - y^\delta) \\
            &= \sum_{m=0}^{l} \alpha_m^{-1} \prod_{j=m}^{l} \alpha_j \left(\alpha_j I + B\right)^{-1} B \left[ \sum_{i=1}^4 \xi_m^{(i)} + \tilde \xi_m^{(5)} \right] + \prod_{j=0}^{l} \alpha_j \left(\alpha_j I + B\right)^{-1} (y^\dag - y^\delta),
        \end{aligned}
    \end{multline}
    where we have used the identity
    \begin{equation*}
        I - \sum_{m=0}^{l} \alpha_m^{-1} \prod_{j=m}^{l} \alpha_j \left(\alpha_j I + B\right)^{-1} B = \prod_{j=0}^{l} \alpha_j \left(\alpha_j I + B\right)^{-1}
    \end{equation*}
    to obtain the last equality.
    Applying \cref{lem:spectral-half} to \eqref{eq:asym-stab-plus1-0} and exploiting the estimates \eqref{eq:xi14-esti} as well as \eqref{eq:xi5-esti} yields
    \begin{multline*}
        \norm{u^\delta_{l+1} - u_{l+1}}_U\\
        \begin{aligned}
            & \leq \frac{1}{2} \sum_{m=0}^{l} \alpha_m^{-1} \left(\sum_{j=m}^l \alpha_j^{-1} \right)^{-1/2} \sum_{i=1}^5 \norm{\xi_m^{(i)} }_Y \\ 
            & \leq \frac{L_1(\rho) }{2} \sum_{m=0}^{l} \alpha_m^{-1} \left(\sum_{j=m}^l \alpha_j^{-1} \right)^{-1/2} \left(\alpha_{m}^{ 1/2}\norm{e_m}_U +\alpha_{m}^{ 1/2}\norm{u_m^\delta - u_m}_U + \norm{G_\dag e_m}_Y + \norm{G_\dag (u_m^\delta - u_m)}_Y \right) \\
            \MoveEqLeft[-1] + \frac{1}{2} \delta \left(1+ L_1(\rho)\right) \sum_{m=0}^{l} \alpha_m^{-1} \left(\sum_{j=m}^l \alpha_j^{-1} \right)^{-1/2}. 
        \end{aligned}
    \end{multline*}
    The estimate \eqref{eq:asymp-esti-auxi} then follows from the above estimate and 
    \cref{lem:sum-esti}.
    Similarly, applying \cref{lem:spectral-nu} to \eqref{eq:resi-asym-stab-plus1-0}, using \cref{lem:sum-esti}, and exploiting the estimates \eqref{eq:xi14-esti} as well as \eqref{eq:xi5-esti} yield
    \begin{multline*}
        \norm{G_\dag(u^\delta_{l+1} - u_{l+1})- y^\delta + y^\dag}_Y\\
        \begin{aligned}
            &\leq \delta + \sum_{m=0}^{l} \alpha_m^{-1} \left(\sum_{j=m}^l \alpha_j^{-1} \right)^{-1} \left[ \sum_{i=1}^4 \norm{\xi_m^{(i)} }_Y +\norm{\tilde \xi_m^{(5)} }_Y \right]\\ 
            & \leq L_1(\rho) \sum_{m=0}^{l} \alpha_m^{-1} \left(\sum_{j=m}^l \alpha_j^{-1} \right)^{-1} \left(\alpha_{m}^{ 1/2}\norm{e_m}_U +\alpha_{m}^{ 1/2}\norm{u_m^\delta - u_m}_U + \norm{G_\dag e_m}_Y + \norm{G_\dag (u_m^\delta - u_m)}_Y \right) \\
            \MoveEqLeft[-1] + \delta \left(1+  c_4L_1(\rho) \right), 
        \end{aligned}
    \end{multline*}
    which gives \eqref{eq:G-esti-auxi}.
\end{proof}
\begin{corollary}
    Under the assumptions in \cref{lem:asym-auxi}, there hold that
    \begin{equation} \label{eq:G-esti2} 
        \norm{G_\dag(u^\delta_{l} - u_{l})- y^\delta + y^\dag}_Y \leq \delta \left(1+ c_4L_1(\rho) \right) + L_1(\rho) c_2 \left[ (6 + 3c_0) + \gamma_0(c_1 + 2c_2) \right] \norm{e_0}_U \alpha_l^{1/2}
    \end{equation}
    and
    \begin{equation} \label{eq:F-esti2} 
        \norm{F(u^\delta_{l}) - F(u_{l})- y^\delta + y^\dag}_Y \leq \delta \left(1+ c_4L_1(\rho) \right) +  \left(L_2(\rho) + \gamma_0 L_3(\rho) \right) \norm{e_0}_U \alpha_l^{1/2}
    \end{equation}
    for all $0 \leq l \leq \tilde{N}_\delta$ with 
    \begin{equation}
        \label{eq:L23_func}
        \left\{
            \begin{aligned}
                & L_2(\rho) := c_2(6 + 3c_0)L_1(\rho) + 2c_0\left[\frac{(1+\kappa_0)}{1-\eta_0} \eta(\rho) + \kappa(\rho) \right],\\
                & L_3(\rho):= c_2(c_1 + 2c_2)L_1(\rho) + 2c_2\left[\frac{(1+\kappa_0)}{1-\eta_0} \eta(\rho) + \kappa(\rho) \right].
            \end{aligned}
        \right.
    \end{equation}    
\end{corollary}
\begin{proof}
    It suffices to prove \eqref{eq:G-esti2} and \eqref{eq:F-esti2} for all $1 \leq l \leq \tilde{N}_\delta$. 
    According to \eqref{eq:G-esti-auxi}, \cref{lem:well-posedness,lem:well-posedness-free}, we have
    \begin{multline*}
        \norm{G_\dag(u^\delta_{l} - u_{l})- y^\delta + y^\dag}_Y \leq \delta \left(1+ c_4L_1(\rho) \right) \\
        \begin{aligned} 
            & + L_1(\rho) \left[ (6 + 3c_0) + \gamma_0 (c_1 + 2c_2) \right] \norm{e_0}_U \sum_{m=0}^{l-1} \alpha_m^{-1/2} \left(\sum_{j=m}^{l-1} \alpha_j^{-1} \right)^{-1},
        \end{aligned}
    \end{multline*} 
    which along with \cref{lem:sum-esti} gives \eqref{eq:G-esti2}. On the other hand, we can deduce from \cref{ass:gtcc} and \eqref{eq:Q_upper_bound} that
    \begin{equation*}
        \begin{aligned}
            \norm{F(u^\delta_{l}) -F(u_l) - y^\delta + y^\dag}_Y & \leq \norm{G_\dag(u^\delta_{l} - u_{l})- y^\delta + y^\dag}_Y \\
            \MoveEqLeft[-1] + \norm{F(u^\delta_{l}) -F(u_l) - G_l(u^\delta_{l} - u_l) }_Y + \norm{(G_l - G_\dag)(u^\delta_{l} - u_l)}_Y \\
            & \leq \norm{G_\dag(u^\delta_{l} - u_{l})- y^\delta + y^\dag}_Y \\
            \MoveEqLeft[-1] + \frac{\eta(\rho)}{1 - \eta_0}\norm{G_l (u^\delta_{l} - u_l) }_Y + \norm{(Q(u_l, u^\dag) - I) G_\dag (u^\delta_{l} - u_l)}_Y \\
            & \leq \norm{G_\dag(u^\delta_{l} - u_{l})- y^\delta + y^\dag}_Y \\
            \MoveEqLeft[-1] + \frac{\eta(\rho)}{1 - \eta_0}\norm{Q(u_l,u^\dag)G_\dag (u^\delta_{l} - u_l) }_Y + \kappa(\rho) \norm{G_\dag (u^\delta_{l} - u_l)}_Y \\
            & \leq \norm{G_\dag(u^\delta_{l} - u_{l})- y^\delta + y^\dag}_Y \\
            \MoveEqLeft[-1] + \left[ \frac{(1+\kappa_0)}{1 - \eta_0}\eta(\rho) +\kappa(\rho)  \right] \norm{G_\dag (u^\delta_{l} - u_l)}_Y.
        \end{aligned}
    \end{equation*}
    From this, \eqref{eq:G-esti2}, \cref{lem:well-posedness,lem:well-posedness-free}, and the monotonic growth of $\kappa$, a simple computation verifies \eqref{eq:F-esti2}.
\end{proof}

We finish this subsection by providing the logarithmic estimate of the stopping index ${N}_\delta$, where we again may have to further restrict the radius $\rho$ of the neighborhood of $u^\dag$.
\begin{lemma} \label{lem:stopping-log-esti}
    Let \cref{ass:gtcc} and \eqref{eq:rho1_cond} be fulfilled and let $\{\alpha_n\}$ be defined by \eqref{eq:Lag-para} and \eqref{eq:scaled}. Assume that $\tau > \tau_0 > 1$, $\gamma_0 \geq \frac{2c_0}{(1-\eta_0)(\tau - \tau_0)}$. Assume further that there exists a positive constant $\rho_2 \leq \rho_1$, with $\rho_1$ given as in \cref{lem:well-posedness}, such that
    \begin{equation}
        \label{eq:rho2_cond}
        c_4 L_1(\rho_2) + L_3(\rho_2) \leq \tau_0 -1 \quad \text{and} \quad	L_2(\rho_2) \leq \frac{c_0}{1-\eta_0}
    \end{equation}
    with $L_i$, $1 \leq i \leq 3$, defined as in \eqref{eq:L1_func} and \eqref{eq:L23_func}.
    Let $\rho \in (0,\rho_2]$ and $u_0 \in U$ be arbitrary such that $(2+c_1\gamma_0)\norm{e_0}_U < \rho$. Then the modified Levenberg--Marquardt iteration \eqref{eq:LM-iteration}--\eqref{eq:disc2} terminates after $N_\delta$ steps with
    \begin{equation*}
        N_\delta = O(1 + |\log(\delta)|). 
    \end{equation*}
\end{lemma}
\begin{proof}
    As a result of \cref{lem:loga-esti}, it suffices to prove $N_\delta \leq \tilde{N}_\delta$. If $\tilde{N}_\delta = 0$, then by definition we have $\alpha_0^{1/2} \norm{e_0}_U \leq \frac{\delta}{\gamma_0} $. The estimate \eqref{eq:GTCC3} thus gives
    \begin{align*}
        \norm{F(u_0) - y^\delta}_Y & \leq \norm{y^\dag - y^\delta}_Y + \norm{F(u_0) - y^\dag}_Y \\
        & \leq \delta + \frac{1}{1 -\eta_0} \norm{G_\dag e_0}_Y \\
        & \leq \delta + \frac{1}{1- \eta_0} \norm{e_0}_U \alpha_0^{1/2} \\
        & \leq \delta \left(1 + \frac{1}{(1- \eta_0)\gamma_0} \right) \\
        & < \tau \delta.
    \end{align*}
    In the following we shall assume $\tilde{N}_\delta >0$. 
    We deduce from \eqref{eq:F-esti2} for $l = \tilde{N}_\delta$ that
    \begin{equation*} 
        \begin{aligned}
            \norm{F(u^\delta_{\tilde{N}_\delta}) - F(u_{\tilde{N}_\delta })- y^\delta + y^\dag}_Y \leq \delta \left(1+ c_4L_1(\rho) \right) + \left(L_2(\rho) + \gamma_0 L_3(\rho) \right) \norm{e_0}_U \alpha_{\tilde{N}_\delta}^{1/2}.
        \end{aligned}
    \end{equation*}
    Using \eqref{eq:GTCC3}, \cref{lem:well-posedness-free}, and noting that $\alpha_{\tilde{N}_\delta }^{1/2} \norm{e_0}_U \leq \frac{\delta}{\gamma_0}$, we derive
    \begin{equation*}
        \begin{aligned}
            \norm{y^\delta -F(u^\delta_{\tilde{N}_\delta})}_Y & \leq \norm{F(u^\delta_{\tilde{N}_\delta}) -F(u_{\tilde{N}_\delta}) - y^\delta + y^\dag }_Y + \norm{F(u_{\tilde{N}_\delta})- y^\dag}_Y \\
            & \leq \norm{F(u^\delta_{\tilde{N}_\delta}) -F(u_{\tilde{N}_\delta}) - y^\delta + y^\dag }_Y + \frac{1}{1-\eta_0} \norm{G_\dag e_{\tilde{N}_\delta}}_Y\\
            & \leq \delta \left(1+ c_4L_1(\rho) \right) +  \left(L_2(\rho) + \gamma_0 L_3(\rho) \right) \norm{e_0}_U \alpha_{\tilde{N}_\delta}^{1/2} +\frac{c_0}{1-\eta_0} \norm{e_0}_U \alpha_{\tilde{N}_\delta}^{1/2} \\
            & \leq \delta 
            \left[ 1 + (c_4 L_1(\rho)+ L_3(\rho)) + \frac{1}{\gamma_0} \left(\frac{c_0}{1-\eta_0} + L_2(\rho) \right) \right].
        \end{aligned}
    \end{equation*}
    Combining this with \eqref{eq:rho2_cond}, the definitions of $L_i$, $ 1 \leq i \leq 3$, and the monotonic growth of $\kappa, \eta$, we obtain
    \begin{equation*}
        \norm{y^\delta -F(u^\delta_{\tilde{N}_\delta})}_Y \leq \delta \left( \tau_0 + \frac{1}{\gamma_0}\frac{2c_0}{1 - \eta_0} \right) \leq \delta \tau.
    \end{equation*}
    From this and the definition of $N_\delta$, we have $N_\delta \leq \tilde{N}_\delta$.
\end{proof}

\subsection{Convergence in the noise free setting} \label{sec:convergence-free}

In this subsection we will show the convergence of the sequence $\{u_n\}$ defined via \eqref{eq:LM-iteration-free}, provided that $e_0 \in \mathcal{N}(G_\dag)^{\bot}$ and that the parameter $\eta(\rho)$ and $\kappa(\rho)$ are small enough if the radius $\rho$ can be chosen accordingly.

We first derive some estimates on $e_n$ and $G_\dag e_n$ under the \emph{generalized source condition}
\begin{equation}
    \label{eq:source-cond}
    e_0 = A^\nu w
\end{equation}
for some $\nu \in (0, \frac{1}{2})$ and some $w \in U$, where $A=G_\dag^*G_\dag$. Again, we may have to restrict $\rho$ further.
\begin{lemma} \label{lem:error-source-free}
    Let all assumptions in \cref{lem:well-posedness-free} hold. Assume further that there exists a constant $\bar \rho_1 \in (0,\rho_1]$, with $\rho_1$ given as in \cref{lem:well-posedness}, satisfying
    \begin{equation}
        \label{eq:rho_1bar_cond}
        \kappa(\bar\rho_1) + c_0 \frac{1+3\kappa_0}{1-\eta_0} \eta(\bar\rho_1) \leq \frac{c_0}{2K_1(r,\nu)}
    \end{equation}
    for some constant $\nu \in (0, \frac{1}{2})$. Let $\rho \in (0,\bar\rho_1]$ be arbitrary and let $u_0 \in U$ satisfy $2\norm{e_0}_U < \rho$ and $e_0 = A^\nu w$ for  some $w \in U$. Then there hold
    \begin{equation}
        \label{eq:error-free-source}
        \norm{e_n}_U \leq 2c_0 \alpha_n^\nu \norm{w}_U \quad \text {and} \quad \norm{G_\dag e_n}_Y \leq 2c_0^2 \alpha_n^{\nu + 1/2} \norm{w}_U
    \end{equation}
    for all $n \geq 0$.
\end{lemma}
\begin{proof}
    We shall prove the lemma by induction on $n$. Obviously, \eqref{eq:error-free-source} is valid for $n=0$. We now assume that \eqref{eq:error-free-source} holds for all $0 \leq n \leq l$ and prove it is also true for $n = l+1$. An argument similar to the one used to obtain \eqref{eq:error-plus1} and \eqref{eq:resi-plus1} yields
    \begin{equation}
        \label{eq:error-plus1-free}
        e_{l+1} = \prod_{m=0}^{l} \alpha_m \left(\alpha_m I + A \right)^{-1} e_0 + \sum_{m=0}^{l} \alpha_m^{-1} \prod_{j=m}^{l} \alpha_j \left(\alpha_j I + A \right)^{-1} G_\dag^*w_m 
    \end{equation}
    and
    \begin{equation}
        \label{eq:resi-plus1-free}
        G_\dag e_{l+1} = \prod_{m=0}^{l} \alpha_m \left(\alpha_m I + B \right)^{-1} G_\dag e_0 + \sum_{m=0}^{l} \alpha_m^{-1} \prod_{j=m}^{l} \alpha_j \left(\alpha_j I + B\right)^{-1} B w_m,
    \end{equation}
    where, analogous to \eqref{eq:w-define},
    \begin{equation*}
        w_m := \alpha_m R_{\alpha_m}(u_m, u^\dag) e_m - \left[ I + S_{\alpha_m}(u_m, u^\dag)\right] z_m
    \end{equation*}
    for all $0 \leq m \leq l$ with $z_m$ defined via \eqref{eq:z-free}. Similarly to \eqref{eq:w-esti}, we have
    \begin{equation*}
        \norm{w_m }_Y \leq \kappa(\rho) \alpha_m^{1/2} \norm{e_m}_U + \left(1 + 3\kappa(\rho)\right)\frac{\eta(\rho)}{1 -\eta_0} \norm{G_\dag e_m}_Y. 
    \end{equation*}
    This and the induction hypothesis yield
    \begin{equation}
        \label{eq:w-esti-free}
        \norm{w_m }_Y \leq Q_1(\rho) \alpha_{m}^{\nu + 1/2}\norm{w}_U
    \end{equation}
    for all $0 \leq m \leq l$
    with
    \begin{equation*}
        Q_1(\rho) := 2c_0\kappa(\rho)+ 2c_0^2 \frac{1+3\kappa(\rho)}{1 - \eta_0}\eta(\rho).
    \end{equation*}
    Inserting $e_0 = A^\nu w$ into \eqref{eq:error-plus1-free} and then applying \cref{lem:spectral-nu,lem:spectral-half}, we deduce
    \begin{equation*}
        \begin{aligned}
            \norm{e_{l+1}}_U & \leq \norm{w}_U \left(\sum_{j=m}^l \alpha_j^{-1} \right)^{-\nu} + \frac{1}{2} \sum_{m=0}^{l} \alpha_m^{-1} \left(\sum_{j=m}^l \alpha_j^{-1} \right)^{-1/2} \norm{w_m}_Y \\
            & \leq c_0^{2\nu} \alpha_{l+1}^\nu \norm{w}_U + \frac{1}{2} Q_1(\rho) \norm{w}_U \sum_{m=0}^{l} \alpha_m^{\nu-1/2} \left(\sum_{j=m}^l \alpha_j^{-1} \right)^{-1/2} \\
            & \leq c_0 \alpha_{l+1}^\nu \norm{w}_U + \frac{1}{2} Q_1(\rho) K_0(r,\nu) \alpha_{l+1}^\nu \norm{w}_U,
        \end{aligned}
    \end{equation*}
    where we used \eqref{eq:w-esti-free} and \cref{lem:sum-esti} to obtain the second inequality and exploited \cref{lem:sum-nu} to obtain the last inequality. 
    By virtue of \eqref{eq:rho_1bar_cond}, the fact that $c_0 /(2K_1(r,\nu)) \leq 1/K_0(r,\nu)$, and the monotonic growth of $\kappa$ and $\eta$, it holds that
    \begin{equation}
        \label{eq:error-esti-free}
        \norm{e_{l+1}}_U \leq 2 c_0 \alpha_{l+1}^\nu \norm{w}_U.
    \end{equation}
    Moreover, by inserting $e_0 = A^\nu w$ into \eqref{eq:resi-plus1-free}, \cref{lem:spectral-halfnu,lem:spectral-nu} and \eqref{eq:w-esti-free} reveal that 
    \begin{equation*}
        \begin{aligned}
            \norm{G_\dag e_{l+1}}_Y & \leq \norm{w}_U \left(\sum_{j=0}^l \alpha_j^{-1} \right)^{-\nu-1/2} + \sum_{m=0}^{l} \alpha_m^{-1} \left(\sum_{j=m}^l \alpha_j^{-1} \right)^{-1} \norm{w_m}_Y \\
            & \leq c_0^{2\nu+1} \alpha_{l+1}^{\nu +1/2} \norm{w}_U + Q_1(\rho) \norm{w}_U \sum_{m=0}^{l} \alpha_m^{\nu-1/2} \left(\sum_{j=m}^l \alpha_j^{-1} \right)^{-1} \\
            & \leq c_0^2 \alpha_{l+1}^{\nu +1/2} \norm{w}_U + Q_1(\rho) K_1(r,\nu) \alpha_{l+1}^{\nu +1/2} \norm{w}_U.
        \end{aligned}
    \end{equation*}
    Here the second estimate is derived using \cref{lem:sum-esti} while the last estimate is obtained using \cref{lem:sum-nu}. Then there holds
    \begin{equation}
        \label{eq:resi-esti-free}
        \norm{G_\dag e_{l+1}}_Y \leq 2c_0^2 \alpha_{l+1}^{\nu +1/2} \norm{w}_U.
    \end{equation}
    From \eqref{eq:error-esti-free} and \eqref{eq:resi-esti-free}, we conclude that \eqref{eq:error-free-source} is fulfilled with $n = l+1$. 
\end{proof}

We now take $\hat u_0 \in U$ to be a perturbation of $u_0 \in U$ and denote by $\{\hat u_n \}$ the iterates given by \eqref{eq:LM-iteration-free} with $u_0$ replaced by $\hat u_0$, that is,
\begin{equation}
    \label{eq:LM-iteration-free-app}
    \hat u_{n+1} = \hat u_n + \left(\alpha_n I + G_{\hat u_n}^*G_{\hat u_n}\right)^{-1}G_{\hat u_n}^*\left(y^\dag - F(\hat u_n) \right), \quad n = 0,1,\ldots
\end{equation} 

For ease of exposition, from now on, we use the notations
\begin{align*}
    &\hat e_n := \hat u_n - u^\dag,&& \quad \hat G_n := G_{\hat u_n},&& \quad \hat A_n :=\hat G_n^*\hat G_n, && \quad \hat B_n := \hat G_n\hat G_n^*, && \quad \hat z_n := F(\hat u_n) - y^\dag - \hat G_n \hat e_n.
\end{align*}

The next lemma is analogous to \cref{lem:asym-stab}.
\begin{lemma} \label{lem:stab-repr}
    Let all assumptions in \cref{lem:well-posedness-free} be fulfilled and let $\bar\rho_1 \in (0, \rho_1]$ and $\nu \in (0,\frac{1}{2})$ satisfy \eqref{eq:rho_1bar_cond}. Assume that $\rho \in (0,\bar\rho_1]$ and that $u_0, \hat u_0 \in U$ satisfy $\min\{2\norm{e_0}_ U,2\norm{\hat e_0}_ U \}  < \rho$ and $\hat e_0 = A^\nu w$ for some $w \in U$. Then for all $k \geq 0$, there holds
    \begin{equation}
        \label{eq:stab-repr}
        u_{k+1} - \hat u_{k+1} = \alpha_k \left(\alpha_k I + A \right)^{-1}(u_k - \hat u_k) + \sum_{i=1}^4 t_k^{(i)},
    \end{equation}
    where
    \begin{equation}
        \label{eq:t-repr}
        t_k^{(i)} = \left(\alpha_k I + A \right)^{-1} G_\dag^{*}h_k^{(i)}, \quad i= 1,2,3,4
    \end{equation}
    with some $h_k^{(i)} \in Y$ satisfying
    \begin{equation}
        \label{eq:h-esti}
        \sum_{i=1}^4 \norm{h_k^{(i)}}_Y \leq C(\rho) \left[ \norm{w}_U \alpha_{k}^{\nu + 1/2} + \norm{u_k - \hat u_k}_U \alpha_{k}^{ 1/2} +\norm{G_\dag (u_k - \hat u_k)}_Y \right]
    \end{equation}
    for 
    \begin{equation}
        \label{eq:C_func}
        C(\rho) := (1+3\kappa_0) \max \left\{ 2c_0\left( \frac{3c_0\eta_0}{1-\eta_0} + 1 + 2c_0(1+\kappa_0) \right) \kappa(\rho),  \frac{(1 +\kappa_0)}{1-\eta_0} \eta(\rho) \right \}.
    \end{equation}
\end{lemma}
\begin{proof}
    Analogous to \eqref{eq:asym-stab-repr}, we see from \eqref{eq:LM-iteration-free}, \eqref{eq:LM-iteration-free-app}, \eqref{eq:z-free}, and the definition of $\hat z_k$ that
    \eqref{eq:stab-repr} is satisfied with 
    \begin{align*}
        & t_k^{(1)} := \alpha_k \left[ \left(\alpha_k I + A_k \right)^{-1} - \left(\alpha_k I + A \right)^{-1} \right](u_k - \hat u_k),\\
        &t_k^{(2)} := \alpha_k \left[ \left(\alpha_k I + A_k \right)^{-1}- \left(\alpha_k I + \hat A_k \right)^{-1} \right] \hat e_k,\\
        & t_k^{(3)} := \left[ \left(\alpha_k I + \hat A_k \right)^{-1} \hat G_k^{*} - \left(\alpha_k I + A_k \right)^{-1} G_k^*\right] \hat z_k,\\
        & t_k^{(4)} := \left(\alpha_k I + A_k \right)^{-1} G_k^{*} (\hat z_k - z_k). 
    \end{align*} 
    We now prove \eqref{eq:t-repr} and \eqref{eq:h-esti}. To verify these relations, we use \cref{lem:transform} to obtain
    \begin{equation*}
        t_k^{(1)} = \alpha_k \left(\alpha_k I + A \right)^{-1} G_\dag^* R_{\alpha_k}(u_k, u^\dag) (u_k - \hat u_k) 
    \end{equation*}
    and thus $h_k^{(1)} := \alpha_k R_{\alpha_k}(u_k, u^\dag) (u_k - \hat u_k)$ verifies \eqref{eq:t-repr} for $i=1$. The estimate \eqref{eq:RS-esti} then implies that
    \begin{equation}
        \label{eq:h1_esti}
        \norm{h_k^{(1)}}_Y \leq \kappa(\rho)\norm{u_k - \hat u_k}_U \alpha_{k}^{ 1/2}.
    \end{equation}
    Furthermore, we have
    \begin{equation*}
        \begin{aligned}
            t_k^{(2)} & = \alpha_k \left(\alpha_k I + \hat A_k \right)^{-1}\hat G_k^* R_{\alpha_k}(u_k, \hat u_k) \hat e_k \\
            &= \alpha_k \left(\alpha_k I + A \right)^{-1} G_\dag^* \left[I + S_{\alpha_k}(\hat u_k, u^\dag)\right] R_{\alpha_k}(u_k, \hat u_k) \hat e_k,
        \end{aligned}
    \end{equation*}
    and so \eqref{eq:t-repr} is valid for $i=2$ with
    \begin{equation*}
        h_k^{(2)} := \alpha_k \left[I + S_{\alpha_k}(\hat u_k, u^\dag)\right] R_{\alpha_k}(u_k, \hat u_k) \hat e_k.
    \end{equation*} 
    This and \eqref{eq:RS-esti} yield
    \begin{equation} \label{eq:h2_esti}
        \begin{aligned}
            \norm{h_k^{(2)}}_Y &\leq (1+ 3\kappa(\rho)) \kappa(\rho) \alpha_k^{1/2} \norm{\hat e_k}_U \\
            & \leq 2c_0 (1+ 3\kappa_0) \kappa(\rho) \norm{w}_U \alpha_{k}^{\nu + 1/2},
        \end{aligned}
    \end{equation}
    where we have used \cref{lem:error-source-free} and the monotonic growth of $\kappa$ to obtain the last estimate. Noting that $u_k, \hat u_k \in \overline{B}(u^\dag, \rho)$, according to \cref{lem:well-posedness-free}, we have 
    \begin{equation*}
        t_k^{(3)} = \left(\alpha_k I + A \right)^{-1} G_\dag^* \left[I + S_{\alpha_k}(u_k, u^\dag)\right] S_{\alpha_k}(\hat u_k, u_k) \hat z_k
    \end{equation*} 
    and therefore $h_k^{(3)} := \left[I + S_{\alpha_k}(u_k, u^\dag)\right] S_{\alpha_k}(\hat u_k, u_k) \hat z_k$. The estimate \eqref{eq:RS-esti} then yields
    \begin{equation*}
        \norm{h_k^{(3)}}_Y \leq 3 (1+ 3\kappa_0) \kappa(\rho) \norm{\hat z_k}_Y.
    \end{equation*} 
    On the other hand, as a result of \eqref{eq:GTCC} and \cref{lem:error-source-free}, we have
    \begin{equation*}
        \begin{aligned}
            \norm{\hat z_k}_Y & \leq \frac{\eta(\rho)}{1-\eta_0} \norm{G_\dag \hat e_k}_Y \leq 2c_0^2 \frac{\eta(\rho)}{1-\eta_0} \norm{w}_U \alpha_{k}^{\nu + 1/2}.
        \end{aligned}
    \end{equation*}
    The two estimates above show that $h_k^{(3)}$ satisfies 
    \begin{equation}
        \label{eq:h3_esti}
        \norm{h_k^{(3)}}_Y  \leq 6c_0^2 (1+ 3\kappa_0) \kappa(\rho)\frac{\eta(\rho)}{1-\eta_0} \norm{w}_U \alpha_{k}^{\nu + 1/2}.
    \end{equation}    
    Finally,
    \begin{equation*}
        \begin{aligned}
            t_k^{(4)} &= \left(\alpha_k I + A \right)^{-1} G_\dag^* \left[I + S_{\alpha_k}(u_k, u^\dag) \right] (\hat z_k - z_k) \\
            & = \left(\alpha_k I + A \right)^{-1} G_\dag^*h_k^{(4)} 
        \end{aligned}
    \end{equation*} 
    with $h_k^{(4)} := \left[I + S_{\alpha_k}(u_k, u^\dag) \right](\hat z_k - z_k)$. From this and \eqref{eq:RS-esti}, we obtain
    \begin{equation}
        \label{eq:h4-1}
        \norm{h_k^{(4)}}_Y \leq \left(1 + 3 \kappa_0\right)\norm{\hat z_k - z_k}_Y.
    \end{equation}
    From the definitions of $z_k$ and $\hat z_k$, it follows that 
    \begin{equation*}
        \begin{aligned}
            \norm{\hat z_k - z_k}_Y & = \norm{F(\hat u_k) - F(u_k) - G_k(\hat u_k - u_k) + (G_k - \hat G_k)\hat e_k }_Y \\
            & \leq \norm{F(\hat u_k) - F(u_k) - G_k(\hat u_k - u_k) }_Y + \norm{ (G_k - \hat G_k)\hat e_k }_Y \\
            & \leq \frac{\eta(\rho)}{1-\eta_0} \norm{G_k(\hat u_k - u_k)}_Y + \norm{ (G_k - \hat G_k)\hat e_k }_Y.
        \end{aligned}
    \end{equation*}
    Here we used \eqref{eq:GTCC2}. Combining this with \eqref{eq:Q-trans}, \eqref{eq:Q-identity}, and \eqref{eq:Q_upper_bound}, we obtain
    \begin{equation*}
        \begin{aligned}
            \norm{\hat z_k - z_k}_Y & \leq \frac{\eta(\rho)}{1-\eta_0} (1+ \kappa_0) \norm{G_\dag(\hat u_k - u_k)}_Y + 2\kappa(\rho)(1+ \kappa_0) \norm{G_\dag \hat e_k}_Y \\
            & \leq (1+ \kappa_0) \left[ \frac{\eta(\rho)}{1-\eta_0}  \norm{G_\dag(\hat u_k - u_k)}_Y + 4c_0^2 \kappa(\rho)\norm{w}_U \alpha_{k}^{\nu + 1/2} \right],
        \end{aligned}
    \end{equation*}
    where we have used \eqref{eq:error-free-source} to get the last inequality. This together with \eqref{eq:h4-1} shows that
    \begin{equation}
        \label{eq:h4_esti}
        \norm{h_k^{(4)}}_Y \leq \left(1 + 3 \kappa_0\right)(1+ \kappa_0) \left[ \frac{\eta(\rho)}{1-\eta_0}  \norm{G_\dag(\hat u_k - u_k)}_Y + 4c_0^2 \kappa(\rho)\norm{w}_U \alpha_{k}^{\nu + 1/2} \right].
    \end{equation}
    Summing up from \eqref{eq:h1_esti}--\eqref{eq:h3_esti} to \eqref{eq:h4_esti} yields     
    \eqref{eq:h-esti}.
\end{proof}

\begin{lemma} \label{lem:stab-free}
    Let all assumptions in \cref{lem:well-posedness-free} be fulfilled and let $\bar\rho_1, \nu$ be defined as in \cref{lem:error-source-free}. Assume that there exists a constant $\bar\rho_2 \in (0,\bar\rho_1]$ satisfying
    \begin{equation}
        \label{eq:rho2_bar_cond}
        C(\bar\rho_2) \leq \min \left\{\frac{c_0}{2c_2(2+c_0)}, \frac{1}{K_0(r,\nu) + 2 K_1(r,\nu)}  \right\},
    \end{equation}
    where $C(\rho)$ is defined by \eqref{eq:C_func}. Let $\rho \in (0, \bar\rho_2]$ be arbitrary.
    Assume in addition that $u_0, \hat u_0 \in U$ are such that $\min\{2 \norm{\hat e}_ U,2 \norm{\hat e_0}_ U \} < \rho$ and $\hat e_0 = A^\nu w$ for some $w \in U$. Then there hold
    \begin{align}
        \label{eq:stab-free}
        \norm{u_n - \hat u_n}_U &\leq 2 \norm{u_0 - \hat u_0}_U +  \pi_1(\rho) \norm{w}_U \alpha_n^\nu
        \shortintertext{and}
        \label{eq:resi-stab-free}
        \norm{G_\dag(u_n - \hat u_n)}_U &\leq c_0 \alpha_n^{1/2} \norm{u_0 - \hat u_0}_U +  \pi_2(\rho) \norm{w}_U \alpha_n^{\nu +1/2} 
    \end{align}
    for all $n \geq 0$. Here
    \begin{equation*}
        \pi_1(\rho) := C(\rho)K_0(r,\nu), \quad \pi_2(\rho) := 2 C(\rho)K_1(r,\nu).
    \end{equation*}
\end{lemma}
\begin{proof}
    We show \eqref{eq:stab-free} and \eqref{eq:resi-stab-free} by induction on $n$. Easily, these estimates hold for $n=0$. Assume that \eqref{eq:stab-free} and \eqref{eq:resi-stab-free} are satisfied for all $ 0 \leq n \leq l$. We shall prove these estimates also hold for $n=l+1$. To that purpose, we apply \cref{lem:stab-repr} to obtain
    \begin{equation}
        \label{eq:stab-free-plus1}
        u_{l+1} - \hat u_{l+1} = \prod_{m=0}^{l} \alpha_m \left(\alpha_m I + A \right)^{-1} (u_0 - \hat u_0) + \sum_{m=0}^{l} \alpha_m^{-1} \prod_{j=m}^{l} \alpha_j \left(\alpha_j I + A \right)^{-1} G_\dag^* \sum_{i=1}^4 h_m^{(i)} 
    \end{equation}
    and
    \begin{equation}
        \label{eq:resi-free-plus1}
        G_\dag \left(u_{l+1} - \hat u_{l+1} \right) = \prod_{m=0}^{l} \alpha_m \left(\alpha_m I + B \right)^{-1} G_\dag (u_0 - \hat u_0) + \sum_{m=0}^{l} \alpha_m^{-1} \prod_{j=m}^{l} \alpha_j \left(\alpha_j I + B\right)^{-1} B \sum_{i=1}^4 h_m^{(i)}.
    \end{equation}
    Applying \cref{lem:spectral-nu} and \cref{lem:spectral-half} to \eqref{eq:stab-free-plus1} and using \cref{lem:stab-repr}, we obtain 
    \begin{equation*}
        \begin{aligned}
            \norm{u_{l+1} - \hat u_{l+1}}_U 
            & \leq \norm{u_0 - \hat u_0}_U + \frac{1}{2} \sum_{m=0}^{l} \alpha_m^{-1} \left(\sum_{j=m}^l \alpha_j^{-1} \right)^{-1/2} \sum_{i=1}^{4} \norm{h_m^{(i)} }_Y \\ 
            & \leq \norm{u_0 - \hat u_0}_U + \frac{1}{2} C(\rho)  \sum_{m=0}^{l} \alpha_m^{-1} \left(\sum_{j=m}^l \alpha_j^{-1} \right)^{-1/2} \\
            \MoveEqLeft[-1] \cdot\left[ \alpha_m^{\nu +1/2} \norm{w}_U + \alpha_m^{1/2} \norm{u_{m} - \hat u_{m}}_U + \norm{G_\dag (u_m - \hat u_m)}_Y \right],
        \end{aligned} 
    \end{equation*}
    which together with the induction hypothesis as well as \cref{lem:sum-esti,lem:sum-nu} shows that
    \begin{equation*}
        \begin{aligned}
            \norm{u_{l+1} - \hat u_{l+1}}_U & \leq \norm{u_0 - \hat u_0}_U + \frac{1}{2} C(\rho) (2+c_0) \norm{u_0 - \hat u_0}_U \sum_{m=0}^{l} \alpha_m^{-1/2}\left(\sum_{j=m}^l \alpha_j^{-1} \right)^{-1/2} \\
            \MoveEqLeft[-1] + \frac{1}{2} C(\rho) (1 + \pi_1(\rho) + \pi_2(\rho) ) \norm{w}_U \sum_{m=0}^{l} \alpha_m^{\nu-1/2} \left(\sum_{j=m}^l \alpha_j^{-1} \right)^{-1/2} \\
            & \leq \norm{u_0 - \hat u_0}_U \left[ 1 + \frac{1}{2} C(\rho) (2+c_0)c_1 \right] + \frac{1}{2} C(\rho) K_0(r,\nu) (1 + \pi_1(\rho) + \pi_2(\rho) ) \norm{w}_U \alpha_{l+1}^\nu. 
        \end{aligned} 
    \end{equation*}
    Thanks to \eqref{eq:rho2_bar_cond}, the definition of $C(\rho)$, the fact that $c_0/(2c_2) \leq 2/c_1$, and the monotonic growth of $\kappa, \eta$, we obtain
    \begin{equation*}
        \norm{u_{l+1} - \hat u_{l+1}}_U \leq 2 \norm{u_0 - \hat u_0}_U + C(\rho) K_0(r,\nu) \norm{w}_U \alpha_{l+1}^\nu.
    \end{equation*}
    This verifies \eqref{eq:stab-free} for $n=l+1$. It remains to prove \eqref{eq:resi-stab-free} for $n =l+1$. To this end, using similar argument as above, we obtain from \eqref{eq:resi-free-plus1}, \eqref{eq:h-esti}, \cref{lem:spectral-half,lem:spectral-nu} that
    \begin{equation*}
        \begin{aligned}
            \norm{G_\dag(u_{l+1} - \hat u_{l+1})}_Y 
            & \leq \frac{1}{2} \left(\sum_{j=0}^l \alpha_j^{-1} \right)^{-1/2} \norm{u_0 - \hat u_0}_U + \sum_{m=0}^{l} \alpha_m^{-1} \left(\sum_{j=m}^l \alpha_j^{-1} \right)^{-1} \sum_{i=1}^{4} \norm{h_m^{(i)} }_Y \\ 
            & \leq \frac{1}{2} c_0 \alpha_{l+1}^{1/2} \norm{u_0 - \hat u_0}_U + C(\rho) \sum_{m=0}^{l} \alpha_m^{-1} \left(\sum_{j=m}^l \alpha_j^{-1} \right)^{-1} \\
            \MoveEqLeft[-1] \cdot \left[ \alpha_m^{\nu +1/2} \norm{w}_U 
            + \alpha_m^{1/2} \norm{u_{m} - \hat u_{m}}_U + \norm{G_\dag (u_m - \hat u_m)}_Y \right].
        \end{aligned}
    \end{equation*}
    The induction hypothesis as well as \cref{lem:sum-esti,lem:sum-nu} then imply that
    \begin{equation*}
        \begin{aligned}
            \norm{G_\dag(u_{l+1} - \hat u_{l+1})}_Y 
            & \leq \frac{1}{2} c_0 \alpha_{l+1}^{1/2}\norm{u_0 - \hat u_0}_U + C(\rho) (2+c_0) \norm{u_0 - \hat u_0}_U \sum_{m=0}^{l} \alpha_m^{-1/2}\left(\sum_{j=m}^l \alpha_j^{-1} \right)^{-1} \\
            \MoveEqLeft[-1] + C(\rho) (1 + \pi_1(\rho) + \pi_2(\rho) ) \norm{w}_U \sum_{m=0}^{l} \alpha_m^{\nu-1/2} \left(\sum_{j=m}^l \alpha_j^{-1} \right)^{-1} \\
            & \leq \norm{u_0 - \hat u_0}_U\alpha_{l+1}^{1/2} \left[ \frac{1}{2}c_0 + C(\rho) (2+c_0)c_2 \right] \\
            \MoveEqLeft[-1] + C(\rho) K_1(r,\nu) (1 + \pi_1(\rho) + \pi_2(\rho) ) \norm{w}_U \alpha_{l+1}^{\nu +1/2} \\
            & \leq c_0\norm{u_0 - \hat u_0}_U\alpha_{l+1}^{1/2} + 2C(\rho)K_1(r,\nu)\norm{w}_U \alpha_{l+1}^{\nu +1/2}, 
        \end{aligned}
    \end{equation*}
    where the last inequality follows from \eqref{eq:rho2_bar_cond}, the definition of $C(\rho)$, and the monotonic growth of $\kappa, \eta$. We thus obtain the desired conclusion. 
\end{proof}

The following corollary is a direct consequence of \cref{lem:error-source-free,lem:stab-free}.
\begin{corollary} \label{cor:stab-free} 
    Under the assumptions of \cref{lem:stab-free}, there hold
    \begin{align}
        \label{eq:error-free}
        \norm{e_n}_U &\leq 2 \norm{u_0 - \hat u_0}_U + \left(2c_0 + \pi_1(\rho) \right) \alpha_n^\nu \norm{w}_U 
        \shortintertext{and}
        \label{eq:resi-free}
        \norm{G_\dag e_n}_Y &\leq c_0 \alpha_n^{1/2} \norm{u_0 - \hat u_0}_U + \left(2c_0^2 + \pi_2(\rho)\right) \alpha_n^{\nu + 1/2} \norm{w}_U
    \end{align}
    for all $n \geq 0$.
\end{corollary}

In the remainder of this subsection, we show the convergence to $u^\dag$ of the sequence $\{u_n\}$. 
\begin{theorem} \label{thm:convergence-free}
    Let $\{\alpha_n\}$ be defined by \eqref{eq:Lag-para} and \eqref{eq:scaled}. Assume that \cref{ass:gtcc} holds and that there exists a constant $\bar\rho_2$ satisfying \eqref{eq:rho2_bar_cond} corresponding to $\nu = \frac{1}{4}$. 
    Let $\rho \in(0,\bar\rho_2]$ and $u_0 \in U$ satisfy $4 \norm{e_0}_U < \rho$ and $e_0 \in \mathcal{N}(G_\dag)^\bot$. Then,
    there holds
    \begin{equation}
        \label{eq:convergence-free}
        \norm{e_n}_U \to 0 \quad \text {and} \quad \frac{\norm{G_\dag e_n}_Y}{\sqrt{ \alpha_n }} \to 0\quad\text{as } n \to \infty.
    \end{equation} 
\end{theorem}
\begin{proof}
    Let $\epsilon>0$ be such that $4 \epsilon < \rho$. Since $e_0 \in \mathcal{N}(G_\dag)^\bot$ and 
    \begin{equation*}
        \mathcal{N}(G_\dag)^\bot = \overline{\mathcal{R}(G_\dag^*)} = \overline{\mathcal{R}(A^{1/2})} \subset \overline{\mathcal{R}(A^{1/4})},
    \end{equation*} 
    there exists an element $\hat u \in U$ such that 
    $\norm{\hat u_0 - u_0} < \epsilon$ and
    $\hat u_0 - u^\dag = A^{1/4}w$ for some $w \in U$. Obviously, $2 \norm{\hat e_0}_U < \rho$ with $\hat e_0 := \hat u_0 - u^\dag$. Applying \cref{cor:stab-free} to the case $\nu = 1/4$ leads to  the estimates
    \begin{align}
        \label{eq:error-free-fourth}
        \norm{e_n}_U &\leq 2 \norm{u_0 - \hat u_0}_U + \left(2c_0 + \pi_1(\rho)\right) \alpha_n^{1/4} \norm{w}_U 
        \shortintertext{and}
        \label{eq:resi-free-fourth}
        \norm{G_\dag e_n}_Y &\leq c_0 \alpha_n^{1/2} \norm{u_0 - \hat u_0}_U + \left(2c_0^2 + \pi_2(\rho)\right) \alpha_n^{3/4} \norm{w}_U
    \end{align}
    are satisfied.
    Since $\alpha_n \to 0$ as $n \to \infty$, there exists a number $n_0 := n_0(\epsilon, \norm{w}_U)$ such that
    \begin{equation*}
        \left(2c_0 + \pi_1(\rho) \right) \alpha_n^{1/4} \norm{w}_U \leq \epsilon \quad\text{for all } n\geq n_0.
    \end{equation*}
    This and \eqref{eq:error-free-fourth} give 
    $\norm{e_n}_U \leq 3 \epsilon$ for all  $n\geq n_0$.
    The first limit in \eqref{eq:convergence-free} then follows. The second limit in \eqref{eq:convergence-free} is similarly obtained from \eqref{eq:resi-free-fourth}. 
\end{proof}

\subsection{Asymptotic stability estimates} \label{sec:asym-stab}

This subsection provides some estimates on $\norm{u_n^\delta - u_n}_U$ and $\norm{G_\dag(u_n^\delta - u_n)}_Y$ with $0 \leq n \leq \tilde{N}_\delta$ that are crucial to prove the regularization property of the modified Levenberg--Marquardt method.
\begin{proposition} \label{prop:asym-stab}
    Let all assumptions in \cref{lem:stab-free} hold true. Assume furthermore that a positive constant $\rho_3 \leq \min\{ \rho_2, \bar\rho_2\}$ exists such that
    \begin{equation}
        \label{eq:rho3_cond}
        \left\{
            \begin{aligned}
                & T_1(\rho_3) + T_2(\rho_3) \leq 2+c_0,\\
                & \pi_1(\rho_3) + \pi_2(\rho_3) + T_1(\rho_3) + T_2(\rho_3) \leq 2c_0(1+c_0),\\
                & L_1(\rho_3)(3+c_3+T_3(\rho_3)) \leq 1,\\
                & T_3(\rho_3) \leq 3 + c_3
            \end{aligned}
        \right.
    \end{equation}
    with 
    \begin{equation*}
        T_1(\rho) := (2c_0 + 2c_0^2)K_0(r,\nu)L_1(\rho), \quad 
        T_2(\rho) := 4 c_0(1+c_0) K_1(r,\nu)L_1(\rho), \quad    T_3(\rho) := 2c_4(3+ c_3)L_1(\rho),
    \end{equation*}
    $L_1(\rho)$ defined as in \cref{lem:asym-stab}, and $\pi_1(\rho)$ and $\pi_2(\rho)$ given in \cref{lem:stab-free}. Let $\rho \leq \rho_3$ and $u_0, \hat u_0 \in U$ be such that $\min\{2 \norm{e_0}_ U,2 \norm{\hat e_0}_ U \} < \rho$ and $\hat e_0 = A^\nu w$ for some $w \in U$. Then there hold
    \begin{align} \label{eq:asym-stab-esti}
        \norm{u_n^\delta - u_n}_U &\leq T_1(\rho) \left(\norm{u_0 - \hat u_0}_U+ \alpha_n^\nu \norm{w}_U \right) + c_3 \frac{\delta}{\sqrt{\alpha_n}} 
        \shortintertext{and}
        \label{eq:resi-asym-stab-esti}
        \norm{G_\dag(u_n^\delta - u_n)}_Y &\leq T_2(\rho) \left(\norm{u_0 - \hat u_0}_U \alpha_n^{1/2} + \alpha_n^{\nu +1/2} \norm{w}_U \right) + (2 + T_3(\rho)) \delta.
    \end{align}   
\end{proposition}
\begin{proof}
    We show \eqref{eq:asym-stab-esti} and \eqref{eq:resi-asym-stab-esti} by induction on $0 \leq n \leq \tilde{N}_\delta$. It is easy to see that these estimates are valid for $n=0$. Now for any fixed $0 \leq l < \tilde{N}_\delta$ we assume that \eqref{eq:asym-stab-esti} and \eqref{eq:resi-asym-stab-esti} are fulfilled for all $0 \leq n \leq l$ and show that these estimates also hold true for $n=l+1$. To this end, using \eqref{eq:asymp-esti-auxi}, the induction hypothesis, and \cref{cor:stab-free}, we estimate 
    \begin{multline*}
        \norm{u^\delta_{l+1} - u_{l+1}}_U \\
        \begin{aligned}
            & \leq \frac{L_1(\rho) }{2} \sum_{m=0}^{l} \alpha_m^{-1} \left(\sum_{j=m}^l \alpha_j^{-1} \right)^{-1/2} \norm{u_0 - \hat u_0}_U \alpha_m^{1/2} \left(2 + c_0 + T_1(\rho) + T_2(\rho) \right) \\
            \MoveEqLeft[-1] +\frac{L_1(\rho) }{2} \sum_{m=0}^{l} \alpha_m^{-1} \left(\sum_{j=m}^l \alpha_j^{-1} \right)^{-1/2} \norm{w}_U \alpha_m^{\nu+1/2} \left(2c_0 + 2c_0^2 + \pi_1(\rho) + \pi_2(\rho) +T_1(\rho) + T_2(\rho) \right) \\
            \MoveEqLeft[-1] +\frac{L_1(\rho) }{2} \sum_{m=0}^{l} \alpha_m^{-1} \left(\sum_{j=m}^l \alpha_j^{-1} \right)^{-1/2} \delta (2 + c_3 + T_3(\rho)) + \frac{1}{2} c_3 \left(1+ L_1(\rho)\right) \frac{\delta}{\sqrt{\alpha_{l+1}}}, 
        \end{aligned}
    \end{multline*}
    which together with \cref{lem:sum-esti,lem:sum-nu} leads to
    \begin{equation*}
        \begin{aligned}
            \norm{u^\delta_{l+1} - u_{l+1}}_U & \leq \frac{1}{2} L_1(\rho) \norm{u_0 - \hat u_0}_U\left(2 + c_0 + T_1(\rho) + T_2(\rho) \right)c_1 \\
            \MoveEqLeft[-1] + \frac{1}{2} L_1(\rho) \norm{w}_U \left(2c_0 + 2c_0^2 + \pi_1(\rho) + \pi_2(\rho) +T_1(\rho) + T_2(\rho) \right) K_0(r,\nu) \alpha_{l+1}^\nu \\
            \MoveEqLeft[-1] + \frac{1}{2} L_1(\rho) (2 + c_3 + T_3(\rho)) c_3 \delta \alpha_{l+1}^{-1/2} + \frac{1}{2} c_3 \left(1+ L_1(\rho)\right) \frac{\delta}{\sqrt{\alpha_{l+1}}},
        \end{aligned}
    \end{equation*} 
    or, equivalently,
    \begin{equation*}
        \begin{aligned}
            \norm{u^\delta_{l+1} - u_{l+1}}_U 
            & \leq \frac{1}{2} c_3 \left[1+ L_1(\rho) (3 + c_3 + T_3(\rho)) \right] \frac{\delta}{\sqrt{\alpha_{l+1}}} \\
            \MoveEqLeft[-1] + \frac{1}{2} L_1(\rho) \norm{u_0 - \hat u_0}_U\left(2 + c_0 + T_1(\rho) + T_2(\rho) \right)c_1 \\
            \MoveEqLeft[-1] + \frac{1}{2} L_1(\rho) \norm{w}_U \left(2c_0 + 2c_0^2 + \pi_1(\rho) + \pi_2(\rho) +T_1(\rho) + T_2(\rho) \right) K_0(r,\nu) \alpha_{l+1}^\nu.
        \end{aligned}
    \end{equation*}
    From this and \eqref{eq:rho3_cond}, the monotonic growth of $\kappa$ and $\eta$ gives 
    \begin{equation*}
        \norm{u^\delta_{l+1} - u_{l+1}}_U \leq T_1(\rho) \left(\norm{u_0 - \hat u_0}_U +\alpha_{l+1}^\nu \norm{w}_U \right) + c_3 \frac{\delta}{\sqrt{\alpha_{l+1}}}.
    \end{equation*}
    The estimate \eqref{eq:asym-stab-esti} is thus verified for $n=l+1$. Similarly, from \eqref{eq:G-esti-auxi}, the induction hypothesis, and \cref{cor:stab-free}, we obtain
    \begin{multline*}
        \norm{G_\dag(u^\delta_{l+1} - u_{l+1})- y^\delta + y^\dag}_Y\\
        \begin{aligned}
            & \leq L_1(\rho) \sum_{m=0}^{l} \alpha_m^{-1} \left(\sum_{j=m}^l \alpha_j^{-1} \right)^{-1} \norm{u_0 - \hat u_0}_U \alpha_m^{1/2} \left(2 + c_0 + T_1(\rho) + T_2(\rho) \right) \\
            \MoveEqLeft[-1] +L_1(\rho) \sum_{m=0}^{l} \alpha_m^{-1} \left(\sum_{j=m}^l \alpha_j^{-1} \right)^{-1} \norm{w}_U \alpha_m^{\nu+1/2} \left(2c_0 + 2c_0^2 + \pi_1(\rho) + \pi_2(\rho) +T_1(\rho) + T_2(\rho) \right) \\
            \MoveEqLeft[-1] +L_1(\rho) \sum_{m=0}^{l} \alpha_m^{-1} \left(\sum_{j=m}^l \alpha_j^{-1} \right)^{-1} \delta (2 + c_3 + T_3(\rho)) + \delta \left(1+ c_4L_1(\rho) \right), 
        \end{aligned}
    \end{multline*}
    which together with \cref{lem:sum-esti,lem:sum-nu} leads to
    \begin{multline*}
        \norm{G_\dag(u^\delta_{l+1} - u_{l+1})- y^\delta + y^\dag}_Y \\
        \begin{aligned}
            &\leq \delta \left[ 1 + L_1(\rho) c_4(3 + c_3 + T_3(\rho)) \right] + L_1(\rho) \norm{u_0 - \hat u_0}_U\left(2 + c_0 + T_1(\rho) + T_3(\rho) \right)c_2 \alpha_{l+1}^{1/2} \\
            \MoveEqLeft[-1] + L_1(\rho) \norm{w}_U \left(2c_0 + 2c_0^2 + \pi_1(\rho) + \pi_2(\rho) +T_1(\rho) + T_2(\rho) \right) K_1(r,\nu) \alpha_{l+1}^{\nu +1/2}.
        \end{aligned}
    \end{multline*} 
    By virtue of \eqref{eq:rho3_cond} and the monotonicity of $\kappa$ and $\eta$, we have that
    \begin{equation}
        \label{eq:resi-asym-esti}
        \begin{aligned}[t]
            \norm{G_\dag(u^\delta_{l+1} - u_{l+1})- y^\delta + y^\dag}_Y& \leq L_1(\rho) \norm{u_0 - \hat u_0}_U \left(4 + 2c_0 \right)c_2 \alpha_{l+1}^{1/2} \\
            \MoveEqLeft[-1] + L_1(\rho) \norm{w}_U \left(4c_0 + 4c_0^2 \right) K_1(r,\nu) \alpha_{l+1}^{\nu +1/2} \\
            \MoveEqLeft[-1] + \delta \left[ 1 + 2L_1(\rho) c_4(3 + c_3) \right] \\
            & \leq T_2(\rho) \left(\norm{u_0 - \hat u_0}_U \alpha_{l+1}^{1/2} + \norm{w}_U \alpha_{l+1}^{\nu +1/2} \right) + \delta (1 + T_3(\rho)).
        \end{aligned}
    \end{equation} 
    Noting that $L_1(\rho)(4 + 2c_0)c_2 \leq L_1(\rho)(4c_0 + 4c_0^2) K_1(r,\nu) =: T_2(\rho)$, the estimate \eqref{eq:resi-asym-stab-esti} is therefore satisfied for $n= l+1$. 
\end{proof}

As a result of \eqref{eq:resi-asym-esti} and \cref{prop:asym-stab}, we have the following corollary, whose proof is similar to that of \eqref{eq:F-esti2}.
\begin{corollary} \label{cor:resi-asym}
    Let all assumptions of \cref{prop:asym-stab} be satisfied. Then there holds
    \begin{equation*}
        \norm{F(u^\delta_{n}) -F(u_n) - y^\delta + y^\dag}_Y \leq T_4(\rho) \left(\norm{u_0 - \hat u_0}_U \alpha_{n}^{1/2} + \norm{w}_U \alpha_{n}^{\nu +1/2} \right) + \delta (1 + T_5(\rho))
    \end{equation*}
    for all $0 \leq n \leq \tilde{N}_\delta$, where
    \begin{equation*}
        T_4(\rho):= T_2(\rho) \left(1 + \kappa(\rho) + \frac{1+\kappa_0}{1-\eta_0} \eta(\rho)\right) , \quad T_5(\rho) := T_3(\rho) + \left( \kappa(\rho) + \frac{1+\kappa_0}{1-\eta_0} \eta(\rho)\right)(2 + T_3(\rho)) 
    \end{equation*}
    and  $T_2(\rho)$ and $T_3(\rho)$ are defined as in \cref{prop:asym-stab}.
\end{corollary}

\subsection{Regularization property} \label{sec:regularization}

This subsection is concerned with the convergence of the sequence $\{u_{N_\delta}^\delta\}$ as $\delta \to 0$, provided that $e_0 \in \mathcal{N}(G_\dag)^\bot$ and that $u_0$ is sufficiently close to $u^\dag$. Let $\{\delta_k\}$ be a positive zero sequence. To simplify the notation, from now on we write $N_k := N_{\delta_k}$. The next lemma will be used to show the convergence of subsequences of $\{u_{N_k}^{\delta_k} \}$ for the case where $\{ N_k\}$ is bounded.
\begin{lemma} \label{lem:converge-finite}
    Assume that all assumptions of \cref{lem:well-posedness} are satisfied. Let $\overline N \in \N$ be arbitrary but fixed and let $\{\delta_k \}$ be a positive zero sequence such that $\tilde{N}_{\delta_k} \geq \overline N$ for all $k \geq 1$.  Assume in addition that \cref{ass:adj_compact} holds.
    Then for any subsequence of $\{\delta_k \}$ there exist a subsequence $\{ \delta_{k_i} \}$ and elements $\tilde{u}_j \in \overline B_U(u^\dag, \rho)$ for $0 \leq j \leq \overline N$ such that
    \begin{equation}
        \label{eq:converge-finite}
        u_{j}^{\delta_{k_i}} \to \tilde{u}_j \quad \text {as} \quad i \to \infty
    \end{equation} 
    for all $0 \leq j \leq \overline N$.
\end{lemma}
\begin{proof}
    We shall show by induction on $j$ the existence of a subsequence $\{ \delta_{k_i} \}$ and elements $\tilde{u}_j \in \overline B_U(u^\dag, \rho)$ for $0 \leq j \leq \overline N$ that satisfy \eqref{eq:converge-finite}. 

    First, \eqref{eq:converge-finite} holds for $j=0$ with $\tilde{u}_0 := u_0$. By a slight abuse of notation, we assume $\{\delta_{k_i}\}$ itself is a subsequence satisfying $u_{j}^{\delta_{k_i}} \to \tilde{u}_j$ as $i \to \infty$ for some $\tilde{u}_j \in \overline{B}_U(u^\dag , \rho)$ and some $0 \leq j < \overline N$. 
    To simplify the notation, we write 
    \begin{equation*}
        u_j^{(i)} := u_{j}^{\delta_{k_i}},\qquad u_{j+1}^{(i)} := u_{j+1}^{\delta_{k_i}},\qquad A_j^{(i)} := A_{j}^{\delta_{k_i}},\qquad \text{and} \quad  G_j^{(i)} := G_{j}^{\delta_{k_i}}.
    \end{equation*}
    It follows from \eqref{eq:LM-iteration} and \cref{lem:transform} that 
    \begin{equation*}
        \begin{aligned}
            u_{j+1}^{(i)} &= u_{j}^{(i)} + \left(\alpha_{j} I + A_{j}^{(i)} \right)^{-1} G_j^{(i)*}\left(y^{\delta_{k_i}} - F(u_j^{(i)}) \right) \\
            & =u_{j}^{(i)} + \left(\alpha_{j} I + A \right)^{-1} G_\dag^{*}\left[ I + S_{\alpha_j}\left(u_{j}^{(i)}, u^\dag \right) \right] \left(y^{\delta_{k_i}} - F(u_j^{(i)}) \right) \\
            & =u_{j}^{(i)} + \left(\alpha_{j} I + A \right)^{-1} G_\dag^{*}\left[ I + S_{\alpha_j}\left(u_{j}^{(i)}, u^\dag \right) \right] \left(y^{\dag} - F(\tilde{u}_j) \right) \\
            \MoveEqLeft[-1] + \left(\alpha_{j} I + A \right)^{-1} G_\dag^{*}\left[ I + S_{\alpha_j}\left(u_{j}^{(i)}, u^\dag \right) \right] \left(y^{\delta_{k_i}} - F(u_j^{(i)}) - y^\dag +F(\tilde{u}_j) \right)
        \end{aligned}
    \end{equation*}
    and thus
    \begin{equation}
        \label{eq:u-plus-finite}
        u_{j+1}^{(i)} = u_{j}^{(i)} + a_i + b_i,
    \end{equation}
    where
    \begin{align*}
        & a_i := \left(\alpha_{j} I + A \right)^{-1} G_\dag^{*}h_i,\\
        & b_i := \left(\alpha_{j} I + A \right)^{-1} G_\dag^{*}\left[ I + S_{\alpha_j}\left(u_{j}^{(i)}, u^\dag \right) \right] \left(y^{\delta_{k_i}} - F(u_j^{(i)}) - y^\dag +F(\tilde{u}_j) \right)
    \end{align*}
    with
    \begin{equation*}
        h_i:= \left[ I + S_{\alpha_j}\left(u_{j}^{(i)}, u^\dag \right) \right] \left(y^{\dag} - F(\tilde{u}_j) \right).
    \end{equation*}
    Applying \cref{lem:spectral-half} with $m=l$ and using \eqref{eq:RS-esti} gives
    \begin{equation*}
        \begin{aligned}
            \norm{b_i}_U & \leq \frac{1}{2 \sqrt{\alpha_j}} \left(1 + 3 \kappa(\rho) \right) \norm{y^{\delta_{k_i}} - F(u_j^{(i)}) - y^\dag +F(\tilde{u}_j) }_Y.
        \end{aligned}
    \end{equation*}
    Letting $i \to \infty$ and employing the continuity of $F$ yields 
    \begin{equation}
        \label{eq:bk-lim}
        b_i \to 0 \quad \text {as} \quad i \to \infty.
    \end{equation}
    Furthermore, \eqref{eq:RS-esti} ensures the boundedness of sequence $\{h_i\}$ in $Y$. Moreover, as a result of \cref{ass:adj_compact}, the operator $\left(\alpha_{j} I + A \right)^{-1} G_\dag^{*}$ is compact. This implies that $\{a_i\}$ is compact in $U$. There thus exist a subsequence of $\{ a_i\}$, denoted by the same symbol, and an element $a \in U$ such that
    \begin{equation}
        \label{eq:a-lim}
        a_i \to a \quad \text {as} \quad i \to \infty.
    \end{equation}
    From \eqref{eq:u-plus-finite}, \eqref{eq:bk-lim}, \eqref{eq:a-lim}, and the induction hypothesis, we deduce
    $u_{j+1}^{(i)} \to \tilde{u}_j + a =: \tilde{u}_{j+1}$.
    Consequently, \eqref{eq:converge-finite} holds for $j+1$.
    From \cref{lem:well-posedness}, we have $ u_{j+1}^{(i)} \in \overline{B}(u^\dag,\rho)$ for all $i \geq 0$ and so $\tilde u_{j+1} \in \overline{B}(u^\dag,\rho)$. The proof is complete.
\end{proof}

Before representing our main theorem, we give a result on the \emph{asymptotic stability} of the modified Levenberg--Marquardt method. The definition of this notion in the following proposition generalizes that in \cite[Def.~2.1]{ClasonNhu2018}.
\begin{proposition} \label{prop:AS}
    Let \cref{ass:gtcc,ass:adj_compact} be fulfilled. Assume that there exists a constant $\rho_3$ satisfying \eqref{eq:rho3_cond} in \cref{prop:asym-stab} corresponding to $\nu = \frac{1}{4}$. Let $\rho \in (0, \rho_3]$ and $u_0 \in U$ satisfy $e_0 \in \mathcal{N}(G_\dag)^\bot$ and  $2(2+c_1\gamma_0)\norm{e_0}_U < \rho$. Then the modified Levenberg--Marquardt method \eqref{eq:LM-iteration}--\eqref{eq:disc2} is \emph{asymptotically stable} in the following sense: For any subsequence of a positive zero sequence $\{\delta_k\}$, there exist a subsequence $\{\delta_{k_i}\}$ and elements $\tilde u_n \in \overline B_U(u^\dag , \rho)$ for all $0 \leq n \leq \overline N:= \lim\limits_{i \to \infty} N_{k_i}$ (where the last inequality is strict if $\overline N = \infty$) such that 
    \begin{equation} \label{eq:AS}
        \lim\limits_{n \to \overline N} \left(\limsup\limits_{i \to \infty} \| u_n^{\delta_{k_i}} - \tilde{u}_n \|_U \right) = 0
    \end{equation}
    and 
    \begin{equation}
        \label{eq:AS-additional}
        \tilde{u}_n \to u^* \quad \text {as} \quad n \to \overline N
    \end{equation}
    for some $u^* \in S_\rho(u^\dag)$.
\end{proposition}
\begin{proof}
    Let $\{\delta_k\}$ itself be a subsequence. Since $\{N_k\}$ is a sequence of integers, there exists a subsequence $\{N_{k_i}\}$ such that either it is a constant sequence or it tends to infinity. For the first case where $N_{k_i} = \overline N$ for some integer $\overline N$ and for all $i \geq 0$, \cref{lem:converge-finite} and the discrepancy principle \eqref{eq:disc2} give the conclusion of the proposition. For the second case where $N_{k_i} \to \infty$, we shall show that the elements $\tilde{u}_n := u_n$, $n \geq 0$, any subsequence $\{\delta_{k_i}\}$, and $u^* := u^\dag$ satisfy \eqref{eq:AS} and \eqref{eq:AS-additional}. To this end, we first see that \cref{thm:convergence-free} implies \eqref{eq:AS-additional}. Let $\epsilon >0$ be arbitrary small but fixed such that $2(2+c_1\gamma_0)\epsilon \leq \rho$. Since $e_0 \in \mathcal{N}(G_\dag)^\bot$ and $\mathcal{N}(G_\dag)^\bot = \overline{\mathcal{R}(G_\dag^*)} = \overline{\mathcal{R}(A^{1/2})} \subset \overline{\mathcal{R}(A^{1/4})}$, there is an element $\hat u \in U$ such that $\norm{\hat u_0 - u_0} < \epsilon$ and $\hat e_0 := \hat u_0 - u^\dag = A^{1/4}w$ for some $w \in U$. 
    Obviously, we have $(2+c_1\gamma_0)\norm{\hat e_0}_U < \rho$.
    From this and the choice of $\rho$, we thus can apply \cref{prop:asym-stab} to obtain the estimate
    \begin{equation*} 
        \norm{u_n^{\delta_{k_i}} - u_n}_U \leq T_1(\rho) \left(\norm{u_0 - \hat u_0}_U+ \alpha_n^{1/4} \norm{w}_U \right) + c_3 \frac{\delta_{k_i}}{\sqrt{\alpha_n}} 
    \end{equation*}
    for all $0 \leq n \leq N_{k_i}$ and for all $i \geq 0$. By letting $i \to \infty$ and then $n \to \infty$, we therefore have
    \begin{align*} 
        \limsup\limits_{n \to \infty} \left(\limsup\limits_{i \to \infty} \| u_n^{\delta_{k_i}} - {u}_n \|_U \right) & \leq T_1(\rho) \epsilon.
    \end{align*}
    The limit \eqref{eq:AS} then follows.
\end{proof}

\bigskip

We are now well prepared to derive the main result of the paper, where some lines in the proof follow the ones in \cite{Jin2016}.
\begin{theorem}[regularization property] 
    \label{thm:REG}
    Let $\{\alpha_n\}$ be defined by \eqref{eq:Lag-para} and \eqref{eq:scaled} and let $\{\delta_k \}$ be a positive zero sequence. Assume that \cref{ass:gtcc,ass:adj_compact} hold and that $\tau > \tau_0 > 1$, $\gamma_0 > \frac{2c_0}{(1-\eta_0)(\tau - \tau_0)}$. 
    Assume further that a constant $\rho_3 \leq \rho_0$ exists and satisfies \eqref{eq:rho3_cond} corresponding to $\nu = \frac{1}{4}$ as well as 
    \begin{equation}
        \label{eq:rho_last_cond}
        T_5(\rho_3) < \tau -1
    \end{equation}
    with $T_5(\rho)$ defined as in \cref{cor:resi-asym}.

    Let $\rho \in (0,\rho_3]$ and $u_0 \in U$ satisfy $2(2+c_1\gamma_0) \norm{u_0 - u^\dag}_U < \rho$. Then the method \eqref{eq:LM-iteration}--\eqref{eq:disc2} is well-defined and the integer $N_{\delta_k}$ defined by the discrepancy principle \eqref{eq:disc2} satisfies
    \begin{equation}
        \label{eq:loga-stopping}
        N_{\delta_k} = O\left(1 + |\log(\delta_k)| \right).
    \end{equation}
    Moreover, if $u_0 - u^\dag \in \mathcal{N}(G_\dag)^{\bot}$, then
    \begin{equation} \label{eq:converge_pro}
        u^{\delta_k}_{N_{\delta_k}} \to u^\dag \quad \text {as} \quad k \to \infty.
    \end{equation}
\end{theorem}
\begin{proof}
    Under the assumptions, the well-posedness of the method follows from \cref{lem:well-posedness}, while the logarithmic estimate \eqref{eq:loga-stopping} is shown in \cref{lem:stopping-log-esti}. It is therefore sufficient to prove \eqref{eq:converge_pro}. 

    To this end, we first assume that there exists a subsequence $\delta_{k_i}$ such that 
    $N_{k_i} = \overline N$ for all $i \geq 0$.  
    By virtue of \cref{lem:converge-finite}, there exist a subsequence $\{ k_m \}$ of $\{k_i\}$ and elements $\tilde{u}_j \in \overline B_U(u^\dag,\rho)$ with $j =0,1,\ldots, \overline N$ such that
    \begin{equation}
        \label{eq:conver-finite}
        u_j^{\delta_{k_m}} \to \tilde{u}_j \quad \text {as} \quad m \to \infty
    \end{equation} 
    for all $0 \leq j \leq \overline N$. Moreover, from the discrepancy principle \eqref{eq:disc2}, we obtain
    \begin{equation*}
        \norm{ F(u_{\overline N}^{\delta_{k_m}}) - y^{\delta_{k_m}} }_Y \leq \tau \delta_{k_m}.
    \end{equation*}
    Letting $m \to \infty$ and using the continuity of $F$ yields
    \begin{equation*}
        F(\tilde{u}_{\overline N}) = y^{\dag},
    \end{equation*}
    which together with \eqref{eq:conver-finite} yields that 
    \begin{equation}
        \label{eq:converge_finite1}
        u_{\overline N}^{\delta_{k_m}} \to \tilde{u}_{\overline N} \quad \text{as } m \to \infty
    \end{equation}
    with $\tilde{u}_{\overline N} \in S_\rho(u^\dag)$. We now show that $ \tilde{u}_{\overline N} = u^\dag$. According to \eqref{eq:LM-iteration}, it holds for all $0 \leq n \leq \overline N-1$ and $m \geq 0$ that
    \begin{equation*}
        u_{n+1}^{\delta_{k_m}} - u_n^{\delta_{k_m}} = G_n^{\delta_{k_m}*} \left( \alpha_n I + G_n^{\delta_{k_m}}G_n^{\delta_{k_m}*} \right)^{-1}(y^{\delta_{k_m}}-F(u_n^{\delta_{k_m}})) \subset \mathcal{R}\left(G_n^{\delta_{k_m}*}\right)
    \end{equation*}
    Combining this with \eqref{eq:Q-trans}, we have $u_{n+1}^{\delta_{k_m}} - u_n^{\delta_{k_m}} \subset \mathcal{R}(G_\dag^*) \subset \mathcal{N}(G_\dag)^{\bot}$ for all $m \geq 0$ and $0 \leq n \leq \overline N-1$. Consequently, using $u_0 - u^\dag \in \mathcal{N}(G_\dag)^{\bot}$, there holds
    $u_{\overline N}^{\delta_{k_m}} - u^\dag \subset \mathcal{R}(G_\dag^*) \subset \mathcal{N}(G_\dag)^{\bot}$. From this and the limit \eqref{eq:converge_finite1}, we have $\tilde{u}_{\overline N} - u^\dag \in \mathcal{N}(G_\dag)^{\bot}$. On the other hand, as a result of \eqref{eq:GTCC} and the fact that $u^\dag, \tilde{u}_{\overline N} \in S_\rho(u^\dag)$, it holds that
    $\tilde u_{\overline N} - u^\dag \in \mathcal{N}(G_\dag)$. We thus have $\tilde u_{\overline N} - u^\dag \in \mathcal{N}(G_\dag) \cap \mathcal{N}(G_\dag)^{\bot} = \{0\}$. Therefore,  a subsequence-subsequence argument can conclude that
    \begin{equation}
        \label{eq:converge_finite2}
        u_{\overline N}^{\delta_{k_i}} \to u^\dag \quad \text{as } i \to \infty.
    \end{equation}

    We next assume that there exists a subsequence $\delta_{k_i}$ such that $N_{k_i} \to \infty$ as $i \to \infty$. 
    In this case, let $\epsilon >0$ be arbitrary but fixed such that $ 0 < 2(2+c_1 \gamma_0)\epsilon < \rho$. Since $e_0 \in \mathcal{N}(G_\dag)^\bot$ and $\mathcal{N}(G_\dag)^\bot = \overline{\mathcal{R}(G_\dag^*)} = \overline{\mathcal{R}(A^{1/2})} \subset \overline{\mathcal{R}(A^{1/4})}$, there exists an element $\hat u \in U$ such that $\norm{\hat u_0 - u_0} < \epsilon$ and
    $\hat u_0 - u^\dag = A^{1/4}w$ for some $w \in U$. 
    Easily, $2\norm{\hat e_0}_U \leq (2 + c_1 \gamma_0) \norm{\hat e_0}_U < \rho$ with $\hat e_0 := \hat u_0 - u^\dag$. From \cref{prop:asym-stab} and \cref{cor:resi-asym}, we have
    \begin{align} \label{eq:asym-stab-fourth}
        \norm{u_j^{\delta_{k_i}} - u_j}_U &\leq T_1(\rho) \left(\norm{u_0 - \hat u_0}_U+ \alpha_j^{1/4} \norm{w}_U \right) + c_3 \frac{\delta_{k_i}}{\sqrt{\alpha_j}} 
        \intertext{and} 
        \label{eq:resi-asym-stab-fourth}
        \norm{F(u_j^{\delta_{k_i}}) - F(u_j) - y^{\delta_{k_i} } + y^\dag }_Y &\leq T_4(\rho) \left(\norm{u_0 - \hat u_0}_U \alpha_j^{1/2} + \alpha_j^{3/4} \norm{w}_U \right) + (1 + T_5(\rho)) \delta_{k_i}
    \end{align}
    for all $0 \leq j \leq N_{k_i}$ and for all $i \geq 0$. 
    On the other hand, we can conclude from the discrepancy principle \eqref{eq:disc2} and the estimate \eqref{eq:GTCC3} for all $0 \leq j < N_{k_i}$ that
    \begin{align*}
        \tau \delta_{k_i} &< \norm{F(u_j^{\delta_{k_i} }) - y^{\delta_{k_i}} }_Y \\
        & \leq \norm{F(u_j^{\delta_{k_i}}) - F(u_j) - y^{\delta_{k_i} } + y^\dag }_Y + \norm{F(u_j)- y^\dag}_Y \\
        & \leq \norm{ F(u_j^{\delta_{k_i}}) - F(u_j) - y^{\delta_{k_i} } + y^\dag }_Y + \frac{1}{1 - \eta_0} \norm{G_\dag e_j}_Y.
    \end{align*}
    Combining this with \eqref{eq:resi-asym-stab-fourth} yields for all $0 \leq j < N_{k_i}$ that
    \begin{equation*}
        \delta_{k_i} \left(\tau - (1 + T_5(\rho))\right) \leq T_4(\rho) \left(\norm{u_0 - \hat u_0}_U \alpha_j^{1/2} + \alpha_j^{3/4} \norm{w}_U \right) + \frac{1}{1 - \eta_0} \norm{G_\dag e_j}_Y
    \end{equation*}
    and thus
    \begin{equation*}
        \left(\tau - (1 + T_5(\rho)\right) \frac{\delta_{k_i}}{\sqrt{{\alpha_{N_{k_i}}}}} \leq  T_4(\rho) \left(\frac{1}{r^{1/2}} \norm{u_0 - \hat u_0}_U + \frac{1}{r^{3/4}} \alpha_{N_{k_i}}^{1/4} \norm{w}_U \right) + \frac{1}{\sqrt{r}(1 - \eta_0)} \frac{\norm{G_\dag e_{N_{k_i}-1} }_Y}{\sqrt{\alpha_{N_{k_i}-1} } }. 
    \end{equation*}
    Letting $i \to \infty$, employing \eqref{eq:rho_last_cond}, and using the second limit in \eqref{eq:convergence-free} gives
    \begin{equation*}
        \left(\tau - (1 + T_5(\rho))\right)\limsup \limits_{i \to \infty} \frac{\delta_{k_i}}{\sqrt{{\alpha_{N_{k_i}}}}} \leq \frac{T_4(\rho)}{r^{1/2}} \norm{u_0 - \hat u_0}_Y \leq \frac{T_4(\rho)}{r^{1/2}} \epsilon.
    \end{equation*}
    Noting that $N_{k_i} \to \infty$ as $i \to \infty$, this implies that
    \begin{equation*}
        \limsup \limits_{i \to \infty} \frac{\delta_{k_i}}{\sqrt{{\alpha_{N_{k_i}}}}} = 0.
    \end{equation*}
    This and \eqref{eq:asym-stab-fourth} yield
    \begin{align*}
        \limsup \limits_{i \to \infty} \norm{u_{N_{k_i}}^{\delta_{k_i}} - u_{N_{k_i }} }_U &\leq T_1(\rho) \norm{u_0 - \hat u_0}_U \leq T_1(\rho) \epsilon
        \intertext{and hence, since $\epsilon>0$ was arbitrary,}
        \limsup \limits_{i \to \infty} \norm{u_{N_{k_i}}^{\delta_{k_i}} - u_{N_{k_i}} }_U &=0.
    \end{align*}
    Together with \eqref{eq:convergence-free}, this implies that
    \begin{equation*}
        u_{N_{k_i}}^{\delta_{k_i}} \to u^\dag \quad \text{as } i \to \infty.
    \end{equation*}
    From this, \eqref{eq:converge_finite2}, and a subsequence-subsequence argument, we obtain 
    \eqref{eq:converge_pro}.
\end{proof}

\section{Iterative regularization for a non-smooth forward operator} \label{sec:maxpde}

In this section, we study the solution operator to \eqref{eq:maxpde-intro} based on previous results from \cite{Christof2018,ClasonNhu2018}. In particular, we show that this operator together with one of its Bouligand subderivatives satisfies the assumptions in \cref{sec:LMmethod}.

\subsection{Well-posedness and directional differentiability}

Let $\Omega\subset\R^d$, $2 \leq d \leq 3$, be a bounded domain with Lipschitz boundary $\partial\Omega$. For $u\in L^2(\Omega)$, we consider the equation
\begin{equation}\label{eq:maxpde}
    \left\{
        \begin{aligned}
            -\Delta y + y^+ &= u \quad\text{in }\Omega,\\
            y &=0 \quad\text{on }\partial\Omega,
        \end{aligned}
    \right.
\end{equation}
where $y^+(x) := \max(y(x),0)$ for all $x \in \Omega$.
From \cite[Thm.~4.7]{Troltzsch}, we obtain for each $u \in L^2(\Omega)$ a unique weak solution $y_u$ belonging to $H^1_0(\Omega) \cap C(\overline{\Omega})$ and satisfying the a priori estimate
\begin{equation*}
    \norm{y_u}_{H^1_0(\Omega)} + \norm{y_u}_{C(\overline{\Omega})} \leq c_\infty \norm{u}_{L^2(\Omega)}
\end{equation*} 
for some constant $c_\infty>0$ independent of $u$.

Let us denote by $F: L^2(\Omega) \to H^1_0(\Omega) \cap C(\overline{\Omega}) \hookrightarrow L^2(\Omega)$ the solution operator of \eqref{eq:maxpde}. 
As shown in \cite[Prop.~3.1]{ClasonNhu2018} (see also \cite[Prop.~2.1]{Christof2018}), $F$ is Lipschitz continuous as a function from $L^2(\Omega)$ to $H^1_0(\Omega) \cap C(\overline{\Omega})$, that is,
\begin{align}
    \| F(u)- F(v)\|_{H^1_0(\Omega)} + \| F(u)- F(v)\|_{C(\overline\Omega)} &\leq C_F \| u -v \|_{L^2(\Omega)} \label{eq:F-Lip}
\end{align} 
for all $u,v \in L^2(\Omega)$ and for some constant $C_F$. Moreover, $F$ is completely continuous
as a function from $L^2(\Omega)$ to $H^1_0(\Omega)$ and from $L^2(\Omega)$ to itself. However, $F$ is in general not Gâteaux differentiable, but it is Gâteaux differentiable at $u$ if and only if $|\{F(u)=0\}|=0$. 

Similarly to \cite{ClasonNhu2018}, we shall use as a replacement for the Fréchet derivative a Bouligand subderivative of $F$ as the operator $G_u$ in \cref{sec:LMmethod}. We first define the set of \emph{Gâteaux points} of $F$ as
\begin{equation*}
    D := \{ v\in L^2(\Omega) : \text{$F: L^2(\Omega) \to H^1_0(\Omega)$ is G\^ateaux differentiable in $v$}\}.
\end{equation*} 
Denoting the Gâteaux derivative of $F$ at $u\in D$ by $F'(u)\in \Linop(L^2(\Omega),H^1_0(\Omega))$ by $F'(u)$, the (strong-strong) \emph{Bouligand subdifferential} at $u\in L^2(\Omega)$ is then defined as
\begin{align*} 
    \partial_{B} F(u) 
    := \{ & G_u \in \Linop(L^2(\Omega), H^1_0(\Omega)) : 
        \text{there exists } \{u_n\}_{n\in\N} \subset D \text{ such that}\\
        & u_n \to u \text{ in } L^2(\Omega)
    \text{ and } F'(u_n)h \to G_u h \text{ in } H^1_0(\Omega) \text{ for all } h \in L^2(\Omega)\}.
\end{align*}

We have the following convenient characterization of a specific Bouligand subderivative of $F$.
\begin{proposition}[{\cite[Prop.~3.16]{Christof2018}}] \label{prop:Gu}
    Given $u \in L^2(\Omega)$, let $G_u: L^2(\Omega) \to H^1_0(\Omega) \hookrightarrow L^2(\Omega)$ be the solution operator mapping $h\in L^2(\Omega)$ to the unique solution $\zeta\in H^1_0(\Omega)$ of
    \begin{equation} \label{eq:bouligand_pde}
        \left\{
            \begin{aligned}
                - \Delta \zeta + \1_{\{y_u>0\}}\zeta &= h \quad \text {in } \Omega,\\
                \zeta &= 0 \quad \text {on } \partial\Omega,
            \end{aligned}
        \right.
    \end{equation}
    where $y_u:=F(u)$. Then $G_u \in \partial_{B} F(u)$.
\end{proposition} 
In general, for a given $h \in L^2(\Omega)$, the mapping $L^2(\Omega) \ni u \mapsto G_uh \in L^2(\Omega)$ is not continuous (see, e.g., \cite[Exam.~3.8]{ClasonNhu2018}), and the mapping $L^2(\Omega) \ni u \mapsto G_u \in \Linop(L^2(\Omega))$ is thus not continuous.

\subsection{Verification of assumptions}
\label{sec:ver-gtcc}

We now verify that the solution mapping for our example together with the mapping $G_u$ defined as in \cref{prop:Gu} satisfies \cref{ass:gtcc} as well as allowing $\rho$ to be taken sufficiently small to satisfy the conditions of \cref{thm:convergence-free,thm:REG}. We begin with the verification of the generalized tangential cone condition \eqref{eq:GTCC}.
\begin{proposition}
    \label{prop:gtcc}
    Let $\bar u \in L^2(\Omega)$, $\bar y := F(\bar u)$, and $\rho >0$. Then there holds
    \begin{equation*}
        \|F(\hat u) - F(u) - G_u(\hat u - u) \|_{L^2(\Omega)} \leq \eta(\rho) \|F(\hat u) - F(u) \|_{L^2(\Omega)}
    \end{equation*} 
    for all $u, \hat u \in \overline B_{L^2(\Omega)}(\bar u, \rho)$
    with 
    \begin{equation}
        \label{eq:eta_func}
        \eta(\rho) := C_\Omega \left| \left\{ |\bar y| \leq C_F \rho \right\} \right|^{1/14}
    \end{equation}
    for some constant $C_\Omega > 0$.
\end{proposition}
\begin{proof}
    Applying to \cite[Lem.~3.9]{ClasonNhu2018} for $p= \frac{7}{4}$ yields 
    \begin{equation}
        \label{eq:eta_esti1}
        \|F(\hat u) - F(u) - G_u(\hat u - u) \|_{L^2(\Omega)} \leq C_\Omega M(u, \hat u)^{1/14} \|F(\hat u) - F(u) \|_{L^2(\Omega)}
    \end{equation} 
    for some constant $C_\Omega$ and $M(u,\hat u) := \left| \{ y_u \leq 0, y_{\hat u} >0 \} \cup \{ y_u > 0, y_{\hat u} \leq 0 \} \right|$.
    According to \eqref{eq:F-Lip}, we thus have that
    \begin{equation*}
        \| \bar y - y_u \|_{C(\overline{\Omega})} \leq C_F \|\bar u - u \|_{L^2(\Omega)} \leq C_F\rho =:\epsilon
    \end{equation*}
    for all $u \in \overline B_{L^2(\Omega)}(\bar u, \rho)$ and $y_u := F(u)$. This implies, for any $u \in \overline B_{L^2(\Omega)}(\bar u, \rho)$, that
    \begin{equation*}
        -\epsilon + y_u(x) \leq \bar y \leq \epsilon + y_u(x) 
    \end{equation*} for all $x \in \overline\Omega$ with $y_u := F(u)$.
    We then have for any $u, \hat u \in \overline B_{L^2(\Omega)}(\bar u, \rho)$ that 
    \begin{align*}
        & \{y_{u}>0, y_{\hat u} \leq 0\} \subset \{-\epsilon \leq \bar y \leq \epsilon\},\\
        & \{y_{u} \leq 0, y_{\hat u} > 0\} \subset \{-\epsilon \leq \bar y \leq \epsilon\}
    \end{align*}
    with $y_{u} := F(u)$ and $y_{\hat u} := F(\hat u)$. It therefore holds that
    \begin{equation}
        \label{eq:different_state}
        \left| \1_{ \left\{ y_{u} >0 \right\} } - \1_{ \left\{ y_{\hat u} >0 \right\} } \right| = \left | \1_{ \left\{ y_{u} >0, y_{\hat u} \leq 0 \right\} } - \1_{ \left\{ y_{\hat u} >0, y_{u} \leq 0 \right\} } \right| \leq  \1_{ \left\{ -\epsilon \leq \bar y \leq \epsilon \right\} }.
    \end{equation}
    From this, we have
    \begin{equation*}
        M(u, \hat u) \leq \left| \left\{ |\bar y| \leq \epsilon \right\} \right| = \left| \left\{ |\bar y| \leq C_F \rho \right\} \right|,
    \end{equation*}
    which together with \eqref{eq:eta_esti1} deduces the desired result.
\end{proof}

We next construct, for any $u_1, u_2 \in L^2(\Omega)$, a bounded linear operator $Q(u_1, u_2): L^2(\Omega) \to L^2(\Omega)$ that satisfies \eqref{eq:Q-trans} and \eqref{eq:Q-identity}. 
\begin{lemma} \label{lem:gtcc-extend}
    Let $u_1, u_2 \in L^2(\Omega)$ be arbitrary and let $G_{u_i}$, $i =1,2$, be defined as in \cref{prop:Gu}. Then there exists a bounded linear operator $Q(u_1,u_2): L^2(\Omega) \to L^2(\Omega)$ such that
    \begin{equation}
        \label{eq:Q-trans-pde}
        G_{u_1} = Q(u_1, u_2)G_{u_2}
    \end{equation} 
    and
    \begin{equation}
        \label{eq:Q-identity-pde}
        \norm{I - Q(u_1, u_2)}_{\Linop(L^2(\Omega))} \leq C(u_1, u_2), 
    \end{equation}
    where 
    \begin{equation*}
        C(u_1, u_2) := C_* \norm{ \1_{ \left\{ F(u_1) >0 \right\} } - \1_{ \left\{ F(u_2) >0 \right\} } }_{L^{3}(\Omega)}
    \end{equation*}
    with some constant $C_* >0$ independent of $u_1$ and $u_2$.
\end{lemma}
\begin{proof}
    To prove the existence of the bounded linear operator $Q(u_1, u_2)$, we first construct this operator on $H^1_0(\Omega)$ and then extend it to $L^2(\Omega)$ by density.
    To this end, we set $y_i := F(u_i)$ with $ i= 1,2$. We now define the linear operator $Q(u_1, u_2): H^1_0(\Omega) \to H^1_0(\Omega) \hookrightarrow L^2(\Omega)$ as follows: for any $v \in H^1_0(\Omega)$, we set $w := Q(u_1, u_2)v$ defined as the unique solution in $H^1_0(\Omega)$ to
    \begin{equation*}
        -\Delta w + \1_{ \left \{ y_1 >0 \right\} }w = -\Delta v + \1_{ \left \{ y_2 >0 \right\} }v \quad \text {in } \Omega.
    \end{equation*}
    We now show that
    \begin{equation}
        \label{eq:L2-bounded}
        \norm{Q(u_1,u_2)v}_{L^2(\Omega)} \leq C \norm{v}_{L^2(\Omega)} \quad\text{for all } v \in H^1_0(\Omega)
    \end{equation}
    and for some constant $C$ independent of $v$. First, we have for any $v\in H^1_0(\Omega)$ that
    \begin{equation}
        \label{eq:subtract-wv}
        -\Delta (w-v) + \1_{ \left \{ y_1 >0 \right\} }(w-v) = \left[ \1_{ \left \{ y_2 >0 \right\} } - \1_{ \left \{ y_1 >0 \right\} }\right] v \quad \text {in } \Omega.
    \end{equation}
    It follows that
    \begin{equation*}
        \norm{w-v}_{H^1_0(\Omega)} \leq C\norm{v}_{L^2(\Omega)}
    \end{equation*}
    for some constant $C>0$. This and the continuous embedding $H^1_0(\Omega) \hookrightarrow L^2(\Omega)$ give
    \begin{equation*}
        \norm{w-v}_{L^2(\Omega)} \leq C\norm{v}_{L^2(\Omega)},
    \end{equation*}
    which along with the triangle inequality yields \eqref{eq:L2-bounded}.
    From the estimate \eqref{eq:L2-bounded} and the density of $H^1_0(\Omega)$ in $L^2(\Omega)$, the operator $Q(u_1, u_2)$ has a unique continuous extension, also denoted by $Q(u_1, u_2)$, from $L^2(\Omega)$ to $L^2(\Omega)$.

    It remains to show \eqref{eq:Q-trans-pde} and \eqref{eq:Q-identity-pde}. It is easy to obtain the identity \eqref{eq:Q-trans-pde} from the definition of $Q(u_1, u_2)$ and the uniqueness of solutions to \eqref{eq:bouligand_pde}. By density, to prove \eqref{eq:Q-identity-pde} we only need to show that 
    \begin{equation}
        \label{eq:Q-identity-pde-aux}
        \norm{v - Q(u_1, u_2)v}_{L^2(\Omega)} \leq C(u_1, u_2) \norm{v}_{L^2(\Omega)}\quad\text{for all }v \in H^1_0(\Omega).
    \end{equation}
    Since $\Omega$ is bounded in $\R^d$ with $d \in \{2,3\}$, one has $H^1_0(\Omega) \hookrightarrow L^6(\Omega)$. 
    Testing \eqref{eq:subtract-wv} by $w-v$ and exploiting the Hölder inequality yield
    \begin{equation*}
        \begin{aligned}
            \norm{\nabla (w-v)}_{L^2(\Omega)}^2 & \leq \int_\Omega \left[ \1_{ \left \{ y_2 >0 \right\} } - \1_{ \left \{ y_1 >0 \right\} }\right] v (w-v)dx \\
            & \leq \norm{\1_{ \left \{ y_2 >0 \right\} } - \1_{ \left \{ y_1 >0 \right\} }}_{L^3(\Omega)} \norm{v}_{L^2(\Omega)} \norm{w-v}_{L^6(\Omega)}.
        \end{aligned}
    \end{equation*}
    From this and the continuous embedding $H^1_0(\Omega) \hookrightarrow L^6(\Omega)$, we obtain
    \begin{equation*}
        \norm{\nabla (w-v)}_{L^2(\Omega)} \leq C \norm{\1_{ \left \{ y_2 >0 \right\} } - \1_{ \left \{ y_1 >0 \right\} }}_{L^3(\Omega)} \norm{v}_{L^2(\Omega)}
    \end{equation*}
    for some constant $C$ independent of $u_1$ and $u_2$. The Poincaré inequality thus implies that
    \begin{equation*}
        \norm{w-v}_{L^2(\Omega)} \leq C_* \norm{\1_{ \left \{ y_2 >0 \right\} } - \1_{ \left \{ y_1 >0 \right\} }}_{L^{3}(\Omega)} \norm{v}_{L^2(\Omega)},
    \end{equation*} 
    which is identical to \eqref{eq:Q-identity-pde-aux}.
\end{proof}

\begin{proposition}
    \label{prop:gtcc-extension}
    Let $Q: L^2(\Omega)^2 \to \Linop(L^2(\Omega))$ be the mapping defined as in \cref{lem:gtcc-extend}, let $\bar u \in L^2(\Omega)$ be arbitrary, and let $\rho$ be a positive number. Then, for any $u_1, u_2 \in \overline{B}_{L^2(\Omega)}(\bar u, \rho)$, there holds
    \begin{equation*}
        \norm{I - Q(u_1, u_2)}_{\Linop(L^2(\Omega))} \leq \kappa(\rho),
    \end{equation*}
    where
    \begin{equation}
        \label{eq:kappa_func}
        \kappa(\rho) := C_*|\{ |\bar y| \leq C_F \rho \}|^{1/3}
    \end{equation}
    with $\bar y := F(\bar u)$ and $C_*$ defined as in \cref{lem:gtcc-extend}.
\end{proposition}
\begin{proof}
    Set $\bar y = F(\bar u)$. According to \eqref{eq:F-Lip}, we thus have that
    \begin{equation*}
        \| \bar y - y_u \|_{C(\overline{\Omega})} \leq C_F \|\bar u - u \|_{L^2(\Omega)} \leq C_F\rho =:\epsilon
    \end{equation*}
    for all $u \in \overline B_{L^2(\Omega)}(\bar u, \rho)$ and $y_u := F(u)$. Similar to \eqref{eq:different_state}, it holds that
    \begin{equation*}
        \left| \1_{ \left\{ y_1 >0 \right\} } - \1_{ \left\{ y_2 >0 \right\} } \right|  \leq  \1_{ \left\{ -\epsilon \leq \bar y \leq \epsilon \right\} },
    \end{equation*}
    which together with the definition of $C(u_1, u_2)$ yields
    \begin{equation*}
        C(u_1, u_2) \leq C_*|\{ |\bar y| \leq \epsilon\}|^{1/3}.
    \end{equation*}
    This and \cref{lem:gtcc-extend} give the desired conclusion.
\end{proof}

From \eqref{eq:eta_func} and \eqref{eq:kappa_func}, we immediately obtain that $\kappa(\rho)$ and $\eta(\rho)$ can be made arbitrarily small provided that $|\{\bar y = 0 \}|$ is small enough. In particular, we deduce that \eqref{eq:GTCC} holds with the required bound on the constant $\eta(\rho)$.
\begin{corollary} \label{lem:smallness_cond}
    Let functions $\kappa$ and $\eta$ be, respectively, defined by \eqref{eq:kappa_func} and \eqref{eq:eta_func}.
    Let $\bar u \in L^2(\Omega)$ be such that $|\{F(\bar u) = 0 \}|$ is sufficiently small. Then \eqref{eq:upper_bounded_GTCC} holds.
\end{corollary}

\subsection{Bouligand--Levenberg--Marquardt iteration}

The results obtained so far indicate that the solution mapping $F$ of \eqref{eq:maxpde} and the mapping $u \mapsto G_u$ with $G_u$ the Bouligand subderivative defined as in \cref{prop:Gu} satisfy \cref{ass:gtcc}, provided that $\left| \{ F(u^\dag) = 0 \} \right|$ is small enough. We note that in this case $F$ is injective, i.e., $u^\dag$ is the unique solution to \eqref{eq:inv-pro}.
We can therefore exploit $G_u$ in the Levenberg--Marquardt method \eqref{eq:LM-iteration}--\eqref{eq:disc2} to obtain a convergent \emph{Bouligand--Levenberg--Marquardt iteration} for the iterative regularization of the non-smooth ill-posed problem $F(u)=y$. 
\begin{corollary}\label{cor:bouligandLM}
    Let $u^\dag \in L^2(\Omega)$ be such that $\left| \{ y^\dag = 0 \} \right|$ is small enough with $y^\dag := F(u^\dag)$. Let $\{\alpha_n \}$ be defined by \eqref{eq:Lag-para} with $\alpha_0^{1/2} \geq \norm{G_{u^\dag}}_{\Linop(L^2(\Omega))}$. Then there exists $\rho^* >0$ such that for all starting points $u_0\in \overline B_{L^2(\Omega)}(u^\dag,\rho^*)$, the Bouligand--Levenberg--Marquardt iteration \eqref{eq:LM-iteration}
    stopped according to the discrepancy principle \eqref{eq:disc2} is a well-posed and strongly convergent regularization method.
\end{corollary}
\begin{proof}
    Take $U = Y = L^2(\Omega)$ and note that $\mathcal{N}(G_{u^\dag}) = \{0\}$ and so $\mathcal{N}(G_{u^\dag})^\bot = L^2(\Omega)$. 
    Then, \cref{ass:gtcc} is satisfied according to \cref{prop:gtcc,prop:gtcc-extension,lem:gtcc-extend,lem:smallness_cond}. 
    \Cref{ass:adj_compact} follows directly from \cref{prop:Gu} together with the compactness of the embedding $H^1_0(\Omega)\hookrightarrow L^2(\Omega)$.
    Finally, the various requirements on the smallness of constants involving $\eta(\rho)$ and $\kappa(\rho)$ are satisfied due to \cref{prop:gtcc-extension}.
    The claim now follows from \cref{thm:REG}.
\end{proof}
We point out that the assumption on the support of $F(u^\dag)$ does \emph{not} entail a similar requirement on $F(u_n^\delta)$, and that this non-differentiability of $F$ at the iterates is the primary source of difficulty in showing convergence.

\bigskip

To close this section, we comment on the practical implementation of the Bouligand--Levenberg--Marquardt iteration \eqref{eq:LM-iteration} for the non-smooth PDE \eqref{eq:maxpde}.
Let $y^\delta \in L^2(\Omega)$. For any $u_n^\delta \in L^2(\Omega)$, we set $y_n^\delta := F(u_n^\delta)$ and define the \emph{correction step}
\begin{equation}
    \label{eq:correction}
    s_n^\delta := \left(\alpha_n I + (G_{n}^{\delta})^*G_{n}^\delta \right)^{-1} (G_{n}^{\delta})^*\left(y^\delta - y_n^\delta \right).
\end{equation}
From this, \eqref{eq:LM-iteration} can be rewritten as $u_{n+1}^\delta = u_n^\delta + s_n^\delta$ with
\begin{equation*}
    \alpha_ns_n^\delta = (G_n^\delta)^*\left(- G_n^\delta s_n^\delta + y^\delta - y_n^\delta\right).
\end{equation*}
By introducing $z_n^\delta := G_{n}^\delta s_n^\delta$ and $b_n^\delta := y^\delta - y_n^\delta$, we deduce that $s_n^\delta$ and $z_n^\delta$ satisfy
\begin{equation}
    \label{eq:correction_ex}
    \left\{
        \begin{aligned}
            - \Delta z_n^\delta + \1_{\{y_n^\delta > 0\}} z_n^\delta &= s_n^\delta &&\ \text {in } \Omega, \quad z_n^\delta = 0 \ \text {on } \partial \Omega,\\
            - \Delta s_n^\delta + \1_{\{y_n^\delta > 0\}}s_n^\delta &= \frac{1}{\alpha_n} \left(- z_n^\delta +b_n^\delta \right) &&\ \text {in } \Omega, \quad s_n^\delta = 0 \ \text {on } \partial \Omega.
        \end{aligned}
    \right.
\end{equation}
A Bouligand--Levenberg--Marquardt step can thus be performed by solving a coupled system of two elliptic equations.

\section{Numerical experiments} \label{sec:num}

This section provides numerical results that illustrate the performance of the Bouligand--Levenberg--Marquardt iteration. In the first subsection, we give a short description of our discretization scheme and the solution of the non-smooth PDE using a semismooth Newton (SSN) method. The second subsection reports the results of numerical examples.

\subsection{Discretization}

In the following, we restrict ourselves to the case where $\Omega$ is an open bounded convex polygonal domain in $\R^2$. We shall use the standard continuous piecewise linear finite elements (FE), see, e.g., \cite{Knabner-Angermann,Glowinski1984}, to discretize the non-smooth semilinear elliptic equation \eqref{eq:maxpde} as well as the linear system \eqref{eq:correction_ex}. 
In \cite{Christof2018,ClasonNhu2018}, the discrete version of \eqref{eq:maxpde} as well as its equivalent nonlinear algebraic system were obtained by employing a mass lumping scheme for the non-smooth nonlinearity. We shall use the same technique to discretize the system \eqref{eq:correction_ex}.
Let $\mathcal{T}_h$ stand for the triangulation of $\Omega$ corresponding to parameter $h$, where $h$ denotes the maximum length of the edges of all the triangles of $\mathcal{T}_h$. 
For each triangulation $\mathcal{T}_h$, let $V_h \subset H^1_0(\Omega)$ be the space of piecewise linear finite elements on $\Omega$. 
We denote by $d_h$ and $\{\varphi_j\}_{j=1}^{d_h}$, respectively, the dimension and the basis of $V_h$ corresponding to the set of nodes $\mathcal{N}_h := \{x_1,\dots,x_{d_h} \}$. For each $T \in \mathcal{T}_h$, we write $\overline{T}$ for the closure of $T$ (i.e., the inner sum is over all vertices of the triangle $T$).

We first consider the nonlinear equation \eqref{eq:maxpde}. Let $y_h$ and $u_h \in V_h$ be the FE approximations of $y$ and $u$, respectively, with $y$ and $u$ satisfying \eqref{eq:maxpde}. As shown in \cite{ClasonNhu2018,Christof2018}, the discrete equation of \eqref{eq:maxpde} is given by
\begin{equation}
    \label{eq:dis-maxpde} 
    \int_{\Omega} \nabla y_h \cdot \nabla v_h \,dx + \frac{1}{3}\sum_{T \in \mathcal{T}_h} |T| \sum_{x_i \in \overline T \cap \mathcal{N}_h} \max(0,y_h(x_i))v_h(x_i) = \int_{\Omega} u_hv_h\,dx, \quad v_h \in V_h,
\end{equation} 
and its equivalent nonlinear algebraic system is defined as
\begin{equation}
    \mathbf{A} y + \mathbf{D} \max(y,0) = \mathbf{M}u, \label{eq:coor-maxpde}
\end{equation} 
where $\mathbf{A} := ((\nabla\varphi_j, \nabla \varphi_i)_{L^2(\Omega)})_{i,j=1}^{d_h}$ is the stiffness matrix, $\mathbf{M} := ((\varphi_j, \varphi_i)_{L^2(\Omega)})_{i,j=1}^{d_h}$ is the mass matrix, $\mathbf D := \frac{1}{3}\diag(\omega_1,\dots, \omega_{d_h})$ with $\omega_i :=|\{\varphi_i\neq 0\}|$ is the lumped mass matrix, and $\max(\cdot, 0): \R^{d_h} \to \R^{d_h}$ is the componentwise max-function. According to \cite{ClasonNhu2018}, the equation \eqref{eq:coor-maxpde} is semismooth  in $\R^{d_h}$ and can be solved via a SSN method. 
Here, with a slight abuse of notation, we write $y\in\R^{d_h}$ and $u\in\R^{d_h}$, respectively, instead of $(y_h(x_i))_{i=1}^{d_h}$ and $(u_h(x_i))_{i=1}^{d_h}$. 

We now turn to the system \eqref{eq:correction_ex}. According to \cite[Sec.~2.5]{Glowinski1984} (see also \cite[Sec.~9.1.3]{Ulbrich2011}), for a fixed $\delta>0$, the discrete linear system of \eqref{eq:correction_ex} is given by
\begin{equation*}
    \label{eq:dis-optimality}
    \left\{
        \begin{aligned}
            \int_{\Omega} \nabla z_h \cdot \nabla v_h \,dx + \frac{1}{3}\sum_{T \in \mathcal{T}_h} |T| \sum_{x_i \in \overline T \cap \mathcal{N}_h} \1_{\{y_n^\delta>0\}}(x_i) z_h(x_i) v_h(x_i) & = \int_{\Omega} s_h v_h\,dx,\\
            \int_{\Omega} \nabla s_h \cdot \nabla w_h \,dx + \frac{1}{3}\sum_{T \in \mathcal{T}_h} |T| \sum_{x_i \in \overline T \cap \mathcal{N}_h} \1_{\{y_n^\delta>0\}}(x_i) s_h(x_i) w_h(x_i) & = -\frac{1}{\alpha_n} \int_{\Omega} \left(z_h - b_h \right)w_h\,dx 
        \end{aligned}
    \right. 
\end{equation*}
for all $v_h,w_h \in V_h$, where $z_h, s_h$, and $b_h$ stand for the FE approximations of $z_n^\delta, s_n^\delta$, and $b_n^\delta$, respectively.
By standard computations, the above variational system can be reformulated as
\begin{equation}
    \label{eq:coor-optimality}
    \left\{
        \begin{aligned}
            \mathbf{A}z + \mathbf{K_y}z &= \mathbf{M}s,\\
            \mathbf{A}s + \mathbf{K_y}s &= - \frac{1}{\alpha_n} \mathbf{M}(z-b)
        \end{aligned}
    \right.
\end{equation}
with 
\begin{equation*}
    \mathbf{K_y} = \frac{1}{3} \diag\left(\omega_i \1_{\{\textbf{y}_i >0\}} \right) \in \R^{d_h \times d_h}, \quad \textbf{y}_i := y_n^\delta(x_i) \quad \text {for all } 1 \leq i \leq d_h. 
\end{equation*}
Here, again, we denote the coefficient vectors $(z_h(x_i))_{i=1}^{d_h}$, $(s_h(x_i))_{i=1}^{d_h}$, and $(b_h(x_i))_{i=1}^{d_h}$ by $z\in\R^{d_h}$, $s\in\R^{d_h}$, and $b\in\R^{d_h}$, respectively. 
A standard argument shows that \eqref{eq:coor-optimality} is uniquely solvable.

\subsection{Numerical examples} 

In this subsection, we consider $\Omega :=(0,1) \times (0,1) \subset \R^2$ and employ a uniform triangular Friedrichs--Keller triangulation with $n_h\times n_h$ vertices for $n_h=512$ unless noted otherwise. 
A direct sparse solver is used to solve the SSN system \eqref{eq:coor-maxpde} and the linear system \eqref{eq:coor-optimality}. The SSN iteration for solving \eqref{eq:coor-maxpde} is initiated at $y^0=0$ and terminated if the active sets $AC^k:=\{i:y^k_i> 0\}$ at two consecutive iterates coincide.
The Python implementation used to generate the following results (as well as a Julia implementation) can be downloaded from \url{https://github.com/clason/bouligandlevenbergmarquardt}. The timings reported in the following were obtained using an Intel Core i7-7600U CPU (2.80\,GHz) and 16\,GByte RAM.

As in \cite{ClasonNhu2018}, we choose the exact solution 
\begin{multline*}
    u^\dag(x_1,x_2) := \max(y^\dag(x_1,x_2),0) \\
    + 
    \left[
        4 \pi^2 y^\dag(x_1,x_2) - 2\left((2x_1-1)^2 + 2(x_1-1+\beta)(x_1-\beta) \right)\sin (2\pi x_2)
    \right]
    \1_{[\beta,1-\beta]}(x_1)
\end{multline*}
where
\begin{equation*}
    y^\dag(x_1,x_2) := \left[(x_1-\beta)^2(x_1-1+ \beta)^2 \sin (2\pi x_2)\right]\1_{[\beta,1-\beta]}(x_1)
\end{equation*} 
for some $\beta \in [0,0.5]$ is the corresponding exact state.
Obviously, $y^\dag \in H^2(\Omega) \cap H^1_0(\Omega)$ and satisfies \eqref{eq:maxpde} for the right-hand side $u^\dag$. Moreover, $y^\dag$ vanishes on a set of measure $2\beta$. The forward operator $F: L^2(\Omega) \to L^2(\Omega)$ is therefore not Gâteaux differentiable at $u^\dag$ whenever $\beta \in (0,0.5]$; see, e.g., \cite[Prop.~3.4]{ClasonNhu2018}. Let us denote by $y_{h}^\dag$ the discrete projection of $y^\dag$ to $V_h$.
We now add a random Gaussian noise componentwise to $y^\dag_{h}$ to create noisy data $y^\delta_{ h}$ corresponding to the noise level
\begin{equation*}
    \delta := \| y^\dag_{h} - y^\delta_{h} \|_{L^2(\Omega)}.
\end{equation*}
Here and below, all norms for discrete functions $v_h$ are computed exactly by $\|v_h\|_{L^2(\Omega)}^2 = v_h^T\mathbf{M} v_h$ (identifying again the function $v_h$ with its vector of expansion coefficients). From now on, to simplify the notation, we omit the subscript $h$.
In the following, we consider different choices of the parameter $\beta$ and two different choices of starting points: the trivial point $u_0 \equiv 0$ and the discrete projection of 
\begin{equation} \label{guess-source}
    \bar u := u^\dag - 20 \sin(\pi x_1) \sin(2\pi x_2).
\end{equation}
We point out that for the second starting point, $u^\dag$ satisfies the generalized source condition
\begin{equation}
    \label{source-condition}
    u^\dag - \bar u \in \mathcal{R}\left[\left(G_{u^\dag}^*G_{u^\dag}\right)^{1/2}\right] \subset \mathcal{R}\left[\left(G_{u^\dag}^*G_{u^\dag}\right)^{\nu}\right]
\end{equation} 
for some $\nu \in (0,1/2)$.
Note also that $\bar u$ is far from the exact solution $u^\dag$ and that $u_0 \equiv 0$ is not close to $u^\dag$ when the parameter $\beta$ is far from $0.5$. 
For the case $\beta =0.005$, the exact solution $u^\dag$ and the starting point $\bar u$ are shown in \cref{fig:target_exactsolution} and \cref{fig:guess-source}, respectively. The corresponding noisy data $y^\delta$ and the reconstructions $u_{N_\delta}^\delta$ with respect to the noise level $\delta \in \{ 1.056 \cdot 10^{-2}, 1.058 \cdot 10^{-4}\}$ are presented in \cref{fig:noise-source} for parameters $\alpha_0=1$, $r=0.5$, $\beta = 0.005$, $\tau =1.5$ and for the starting point $u_0 = \bar u$. 
\begin{figure}[tp]
    \centering
    \begin{minipage}[t]{0.495\textwidth}
        \centering
        \includegraphics[width=\linewidth]{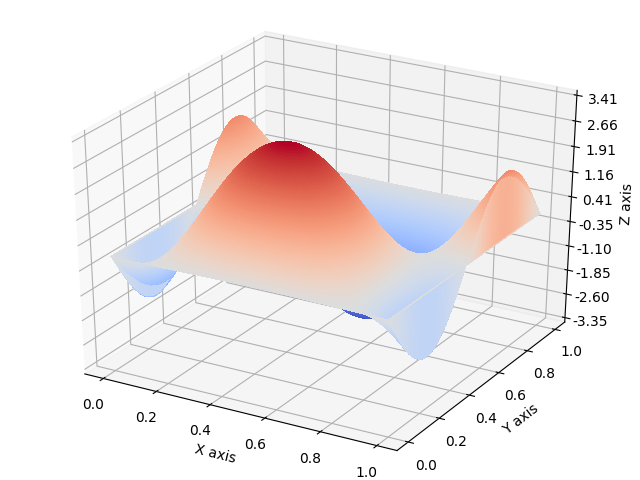}
        \caption{exact solution $u^\dag$ for $\beta = 0.005$}
        \label{fig:target_exactsolution}
    \end{minipage}
    \hfill
    \begin{minipage}[t]{0.495\textwidth}
        \includegraphics[width=\linewidth]{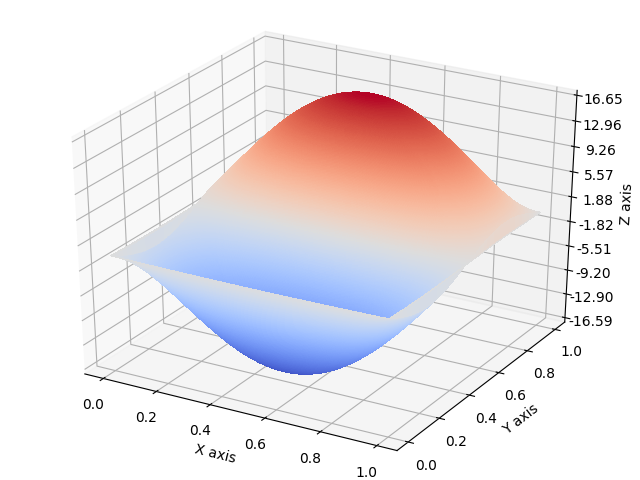}
        \caption{starting point $u_0=\bar u$ for $\beta =0.005$}
        \label{fig:guess-source}
    \end{minipage}
    \begin{subfigure}[t]{0.495\textwidth}
        \centering
        \includegraphics[width=\linewidth]{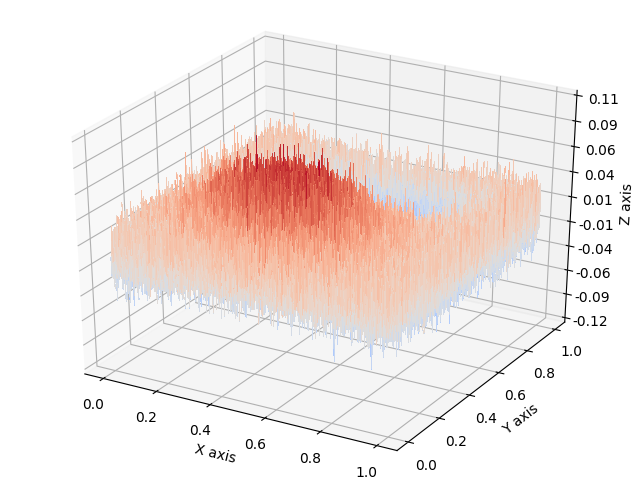}
        \caption{$y^\delta$, $\delta = 1.056 \cdot 10^{-2}$}
    \end{subfigure}
    \hfill
    \begin{subfigure}[t]{0.495\textwidth}
        \centering 
        \includegraphics[width=\linewidth]{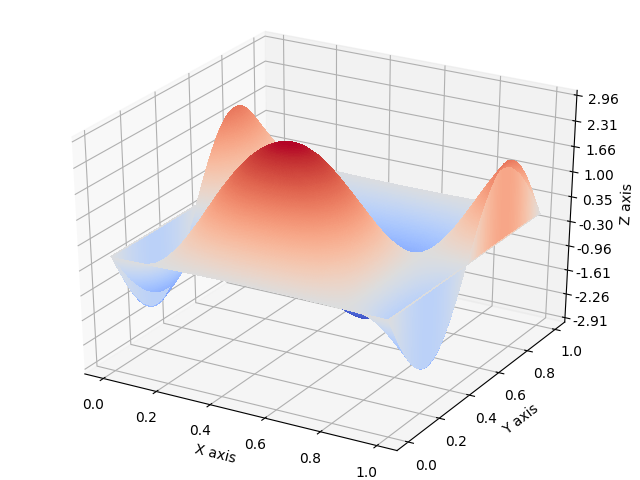}
        \caption{$u^\delta_{N_\delta}$, $N_\delta = 14$} 
    \end{subfigure}
    \\
    \begin{subfigure}[t]{0.495\textwidth}
        \centering
        \includegraphics[width=\linewidth]{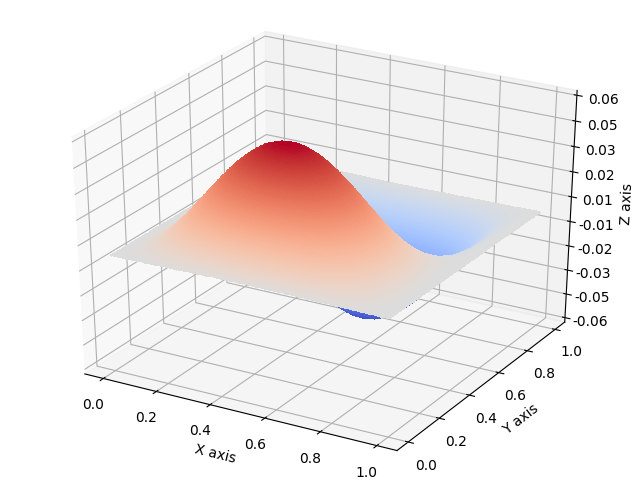}
        \caption{$y^\delta$, $\delta = 1.058 \cdot 10^{-4}$}
    \end{subfigure}
    \hfill
    \begin{subfigure}[t]{0.495\textwidth}
        \centering 
        \includegraphics[width=\linewidth]{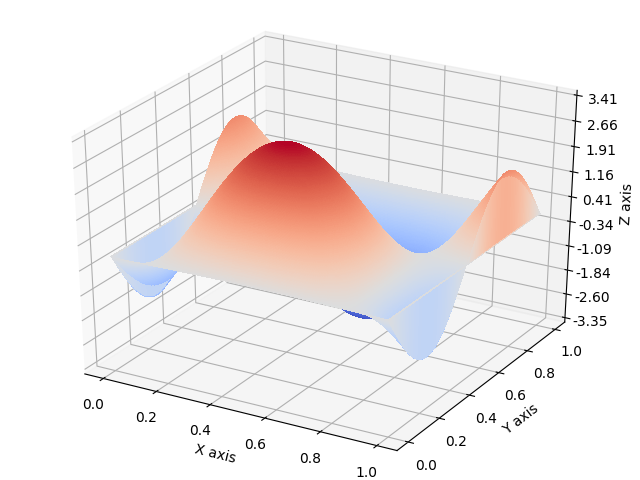}
        \caption{$u^\delta_{N_\delta}$, $N_\delta = 16$}
    \end{subfigure}
    \caption{noisy data $y^\delta$ and reconstructions $u^\delta_{N_\delta}$ for  $u_0 = \bar u$ and $\alpha_0 =1$, $r=0.5$, $\beta = 0.005$, $\tau = 1.5$}
    \label{fig:noise-source}
\end{figure}

We now address the regularization property of the Bouligand--Levenberg--Marquardt iteration from \cref{cor:bouligandLM}. 
We first illustrate the effects of the starting guess on the convergence of the iteration. \cref{tab:noise_diff_guess} displays for the same parameters $\alpha_0 = 1$, $r= 0.5$, $\beta =0.005$, $\tau=1.5$, a decreasing sequence of noise levels, and both starting points (for the same realization of the random data) the stopping index $N_\delta=N(\delta,y^\delta)$, the logarithmic rate of the stopping index
\begin{align}
    \label{eq:log-rate}
    LR_\delta &:= \frac{N_\delta}{1 + |\log(\delta)|},
    \intertext{the relative error}
    \label{rel-error-noise}
    E^\delta&:= \frac{\| u^\dag -u^\delta_{N_\delta} \|_{L^2(\Omega)}}{\| u^\dag\|_{L^2(\Omega)}},
    \intertext{the empirical convergence rate}
    \label{conv-rate}
    R^\delta&:= \frac{\| u^\dag - u^\delta_{N_\delta}\|_{L^2(\Omega)}}{\sqrt{\delta}},
\end{align}
as well as the final Tikhonov parameter $\alpha_{N_\delta}$ from \eqref{eq:Tikhonov-choice}.
This table indicates that the speed of convergence of the iteration for the starting point $u_0 = \bar u$ is faster than that for the trivial starting point $u_0\equiv 0$. While the growth of the stopping index $N_\delta$ for the trivial starting point is slightly faster than that for $\bar u$, the logarithmic rates \eqref{eq:log-rate} for both starting points are stable. This fits \cref{thm:REG}. For the starting guess $u_0 = \bar u$, the empirical convergence rate $R^\delta$ is not greater than $0.4$ as $\delta$ is small enough. This agrees with the convergence rate $\mathcal{O}(\sqrt{\delta})$ expected from the classical source condition $u^\dag-u_0\in \mathcal R\left[(F'(u^\dag)^*F'(u^\dag)^{1/2} \right]$.
\begin{table}[!t]
    \caption{regularization property for $\alpha_0=1$, $r=0.5$, $\beta = 0.005$, $\tau = 1.5$: noise level $\delta$; stopping index $N_\delta$; logarithmic rate {$LR_\delta$} from \eqref{eq:log-rate}; relative error $E^\delta$ from \eqref{rel-error-noise}; empirical convergence rate $R^\delta$ from \eqref{conv-rate}; final Tikhonov parameter $\alpha_{N_\delta}$}
    \label{tab:noise_diff_guess}
    \centering
    \begin{tabular}{%
            S[table-format=1.2e-2,scientific-notation=true,round-mode=places,round-precision=2]
            S[table-format=2]
            S[table-format=1.1,scientific-notation=true,round-mode=places,round-precision=1]
            S[table-format=1.2e-1,scientific-notation=true,round-mode=places,round-precision=2]
            S[table-format=1.1e-1,scientific-notation=true,round-mode=places,round-precision=1]
            S[table-format=2]
            S[table-format=1.1,round-mode=places,round-precision=1]
            S[table-format=1.2e-1,scientific-notation=true,round-mode=places,round-precision=2]
            S[table-format=1.1,round-mode=places,round-precision=1] 
            S[table-format=1.1e-1,scientific-notation=true,round-mode=places,round-precision=1] 
        }
        \toprule
        & \multicolumn{4}{c}{$u_0\equiv 0$} & \multicolumn{5}{c}{$u_0=\bar u$} \\
        \cmidrule(lr){2-5} \cmidrule(lr){6-10}
        {$\delta$} & {$N_\delta$} & {$LR_\delta$}      & {$E^\delta$}         & {$\alpha_{N_\delta}$} & {$N_\delta$} & {$LR_\delta$}      & {$E^\delta$}           & {$R^\delta$}        & {$\alpha_{N_\delta}$} \\
        \midrule                                                                                                                        
        1.056e-02 & 12           & 2.161944383663075  & 0.4696226322801729   & 4.096e-09             & 14           & 2.522268447606921  & 0.15489855288625703    & 2.2599960301113002  & 1.6384e-10            \\
        1.059e-03 & 16           & 2.0381843728706475 & 0.2360039884930222   & 6.5536e-12            & 15           & 1.910797849566232  & 0.020684346280542436   & 0.9528977241146434  & 3.2768e-11            \\
        1.058e-04 & 20           & 1.9697284944764908 & 0.14443145223412981  & 1.048576e-14          & 16           & 1.5757827955811925 & 0.001572329646729303   & 0.22917103233800698 & 6.5536e-12            \\
        1.054e-05 & 25           & 2.0064072742036223 & 0.07326802350440055  & 3.3554432e-18         & 17           & 1.364356946458463  & 0.0003573843573686938  & 0.1650364174000802  & 1.31072e-12           \\
        1.057e-06 & 30           & 2.0325427962934035 & 0.035547482392459896 & 1.073741824e-21       & 18           & 1.2195256777760421 & 0.00018491591808791902 & 0.2696521385334294  & 2.62144e-13           \\
        1.059e-07 & 34           & 1.992900892328594  & 0.02769122585195216  & 1.7179869184e-24      & 21           & 1.2309093746735433 & 6.466176656871649e-05  & 0.29790120197753356 & 2.097152e-15          \\
        \bottomrule
    \end{tabular}
\end{table}

To show the dependence on parameter $\beta$ of the performance of the Bouligand--Levenberg--Marquardt iteration, we summarize in
\cref{tab:noise_diff_beta} the results obtained for $\beta \in \{ 0, 0.15,0.3 \}$, $\alpha_0=1$, $r= 0.5$, $\tau = 1.5$, and $u_0 = \bar u$. 
\Cref{tab:noise_diff_beta} indicates that the stopping index seems not to be significantly influenced by the parameter $\beta$. However, it is not surprising that the relative error $E^\delta$ increases with respect to $\beta$ since $|\{y^\dag = 0\}| \to 0$ as $\beta \to 0^+$. 
\begin{table}[!t]
    \caption{regularization property for $\alpha_0=1$, $r=0.5$, $\tau = 1.5$, $u_0 = \bar u$: noise level $\delta$; stopping index $N_\delta=N(\delta,y^\delta)$; relative error $E^\delta$ from \eqref{rel-error-noise}}
    \label{tab:noise_diff_beta}
    \centering
    \begin{tabular}{%
            S[table-format=1.2e-2,scientific-notation=true,round-mode=places,round-precision=2]
            S[table-format=2]
            S[table-format=1.2e-1,scientific-notation=true,round-mode=places,round-precision=2]
            S[table-format=2]
            S[table-format=1.2e-1,scientific-notation=true,round-mode=places,round-precision=2]
            S[table-format=2]
            S[table-format=1.2e-1,scientific-notation=true,round-mode=places,round-precision=2]
        }
        \toprule
        & \multicolumn{2}{c}{$\beta = 0$} & \multicolumn{2}{c}{$\beta = 0.15$} & \multicolumn{2}{c}{$\beta = 0.3$} \\
        \cmidrule(lr){2-3} \cmidrule(lr){4-5} \cmidrule(lr){6-7}
        {$\delta$} & {$N_\delta$} & {$E^\delta$}          & {$N_\delta$} & {$E^\delta$}          & {$N_\delta$} & {$E^\delta$}         \\
        \midrule
        1.059e-01  & 11           & 3.096603267432972     & 11           & 11.322106625468484    & 11           & 64.0850013272603     \\
        1.058e-02  & 14           & 0.14974025449860318   & 14           & 0.5467991241900994    & 14           & 3.101289351522812    \\
        1.057e-03  & 15           & 0.01962664791836133   & 15           & 0.07108813191039302   & 15           & 0.4105428851642272   \\
        1.059e-04  & 16           & 0.0015019252427895206 & 16           & 0.005749644212339104  & 16           & 0.03521304813871039  \\
        1.057e-05  & 17           & 0.0003431288402027048 & 17           & 0.00355545603312123   & 17           & 0.01024127740176486  \\
        1.059e-06  & 18           & 0.000181345040429809  & 19           & 0.0037424266649199052 & 18           & 0.005294979116644965 \\
        \bottomrule
    \end{tabular}
\end{table}

Finally, the stopping index as well as the total CPU time (in seconds) of the proposed Bouligand--Levenberg--Marquardt (BLM) iteration and of the Bouligand--Landweber (BL) iteration from \cite{ClasonNhu2018} are compared in \cref{fig:LM_BL_comparision}. Recall that the BL iteration is defined by
\begin{equation}
    \label{eq:modified-landweber} 
    u_{n+1}^\delta = u_n^\delta + w_n G_{u_n^\delta}^* \left(y^\delta - F(u_n^\delta) \right), \quad n \geq 0
\end{equation}
with parameter $w_n >0$ 
and is terminated via the discrepancy principle \eqref{eq:disc2}. To compare the numerical results, we set $\alpha_0 = 1$, $r= 0.5$, $\beta = 0.005$, $\tau = 1.5$, $w_n = (2-2\mu)/\bar L^2$ for all $n \geq 0$ with $\mu = 0.1$ and $\bar L = 0.05$. \cref{fig:LM_BL_comparision}
shows the stopping index of the two iterative methods versus the noise level $\delta$ for both $u_0 = \bar u$ and $u_0 \equiv 0$. \Cref{fig:stop-index-source,fig:stop-index-zero} indicate that for the BLM iteration, in both cases $N_\delta = O(1 + |\log(\delta)|)$ as $\delta \to 0$, as expected from \cref{thm:REG}.
On the other hand, \cref{fig:stop-index-source-landweber,fig:stop-index-zero-landweber} show that for the BL iteration, $N_\delta = O(\delta^{-1})$ for $u_0 = \bar u$ and $N_\delta = O(\delta^{-2})$ for $u_0 \equiv 0$ as $\delta\to 0$. 
As also shown in these figures, the total CPU time to run each method is almost directly proportional to their stopping indices (approximately $52$ seconds per step for the BLM iteration and $16$ seconds per for the BL iteration, corresponding to the size of \eqref{eq:coor-optimality} compared to that of the discretization of \eqref{eq:modified-landweber}). For $u_0 \equiv 0$ and $\delta \approx 5 \cdot 10^{-5}$, the total CPU time of the BLM iteration is only $1136$ seconds while that of the BL iteration is nearly $41337$ seconds. Similarly, for $u_0 = \bar u$ and $\delta \approx 10^{-7}$, it takes $1087$ seconds for the BLM iteration and approximately $18054$ seconds for the BL iteration to terminate.
Hence even though the cost of each step of the two iterations is different, the BLM iteration is significantly faster also in terms of CPU time for small values of $\delta$. 
\begin{figure}[t]
    \centering
    \begin{subfigure}[b]{0.495\textwidth}
        \centering
        \input{stop-index-zero.tikz} 
        \caption{$u_0=0$, BLM}
        \label{fig:stop-index-zero}
    \end{subfigure}
    \begin{subfigure}[b]{0.495\textwidth}
        \centering
        \input{stop-index-zero-landweber.tikz} 
        \caption{$u_0=0$, BL}
        \label{fig:stop-index-zero-landweber}
    \end{subfigure}
    \\[1ex]
    \begin{subfigure}[b]{0.495\textwidth}
        \centering
        \input{stop-index-source.tikz} 
        \caption{$u_0=\bar u$, BLM}
        \label{fig:stop-index-source}
    \end{subfigure}
    \begin{subfigure}[b]{0.495\textwidth}
        \centering
        \input{stop-index-source-landweber.tikz} 
        \caption{$u_0=\bar u$, BL}
        \label{fig:stop-index-source-landweber}
    \end{subfigure}
    \caption{comparison of stopping index $N_\delta$ and total CPU time (in seconds) for Bouligand--Levenberg--Marquardt (BLM) and Bouligand--Landweber (BL) iterations}
    \label{fig:LM_BL_comparision} 
\end{figure}
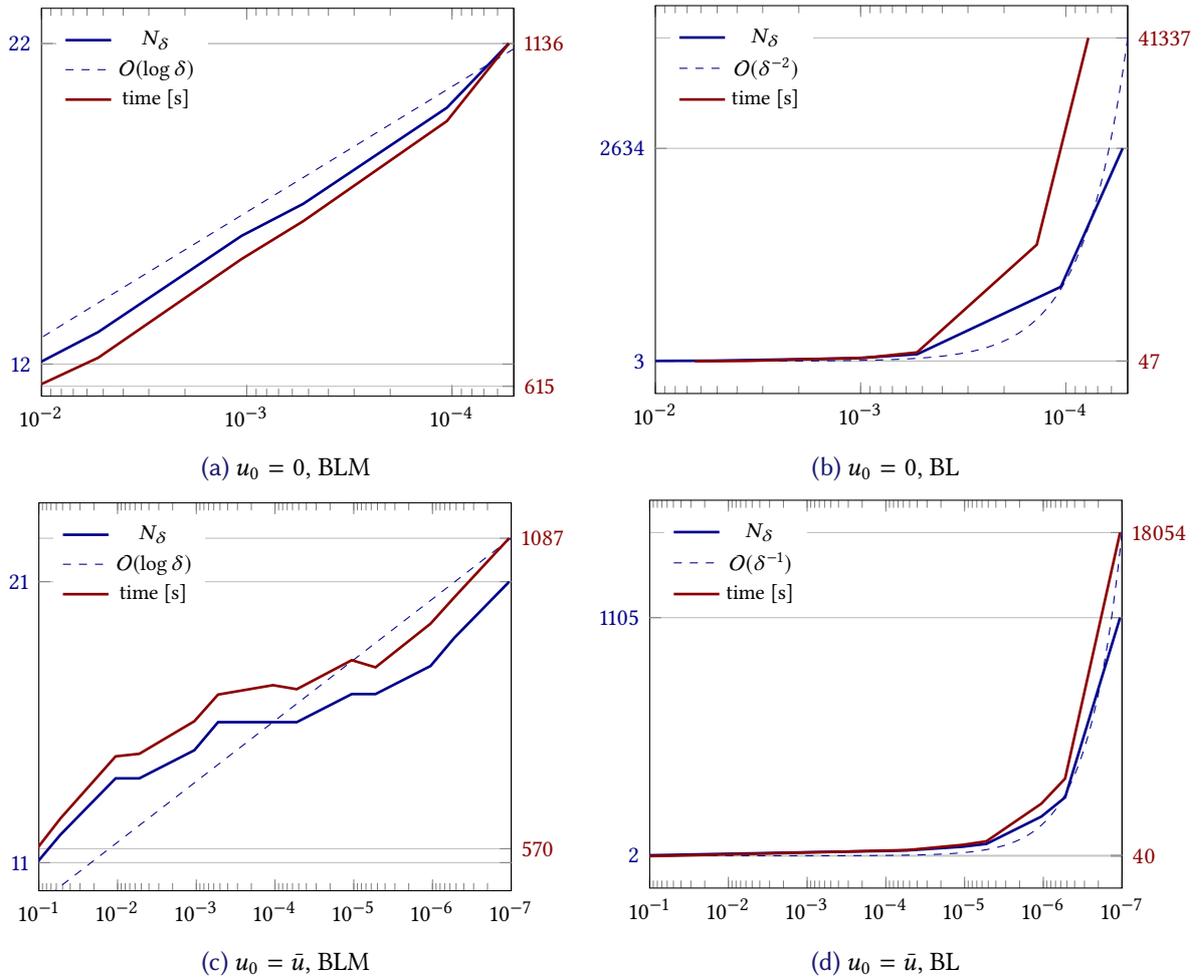

\section{Conclusion} \label{sec:Con}

We have proposed a novel Newton-type regularization method for non-smooth ill-posed inverse problems that extends the classical Levenberg--Marquardt iteration. Using a family of bounded operators $\{G_u\}$ to replace the Fréchet derivative in the classical Levenberg--Marquardt iteration, we proved under a generalized tangential cone condition the asymptotic stability of the iterates and from this derived the regularization property of the iteration. In particular, when considering ill-posed inverse problem where the forward operator corresponds to the solution of a non-smooth semilinear elliptic PDE, we can take $G_u$ from the Bouligand subdifferential of the forward operator. If the non-differentiability of the forward mapping is sufficiently ``weak'' at the exact solution, these operators satisfy the required assumptions, and the resulting Bouligand--Levenberg--Marquardt iteration thus provides a convergent regularization method. As the numerical example illustrates, this iteration requires significantly less iterations and can be much faster than first-order methods such as the Bouligand--Landweber iteration from \cite{ClasonNhu2018}.

This work can be extended in several directions. First, it would be interesting to derive convergence rates under the generalized source condition \eqref{eq:source-cond}. Of particular interest would be the extension of the proposed iteration for non-smooth ill-posed inverse problems with additional constraints such as non-negativity of the unknown parameter.
Finally, similar non-smooth extensions of other Newton-type methods such as the iteratively regularized Gauss-Newton method could be derived.

\appendix

\section{Auxiliary lemmas} \label{sec:app}

This section provides some estimates on the sequence of parameters defined by \eqref{eq:Lag-para} and on bounded linear operators between Hilbert spaces.
\begin{lemma}[\protect{\cite{Jin2010}}] \label{lem:sum-esti}
    Let $\{\alpha_n \}$ be defined via \eqref{eq:Lag-para}. Then there hold for all $k \geq 0$
    \begin{align*}
        \left(\sum_{j = 0}^k \alpha_j^{-1} \right)^{-1} &\leq c_0^2 \alpha_{k+1},&
        \sum_{m=0}^k \alpha_m^{-1/2} \left(\sum_{j = m}^k \alpha_j^{-1} \right)^{-1/2} &\leq c_1, &
        \sum_{m=0}^k \alpha_m^{-1/2} \left(\sum_{j = m}^k \alpha_j^{-1} \right)^{-1} &\leq c_2\alpha_{k+1}^{1/2} ,\\
        && \sum_{m=0}^k \alpha_m^{-1} \left(\sum_{j = m}^k \alpha_j^{-1} \right)^{-1/2} &\leq c_3\alpha_{k+1}^{-1/2},&
        \sum_{m=0}^k \alpha_m^{-1} \left(\sum_{j = m}^k \alpha_j^{-1} \right)^{-1} &\leq c_4,
    \end{align*}
    where the constants $c_i$, $i = 0,\ldots,4$, are defined in \eqref{eq:constants-c}.
\end{lemma}

The next lemma provides some more estimates on sequence $\{\alpha_n \}$ defined via \eqref{eq:Lag-para} with exponent $\nu \in [0,1/2)$. Its proof is standard and thus is omitted.
\begin{lemma} \label{lem:sum-nu}
    Let $\{\alpha_n \}$ be defined via \eqref{eq:Lag-para} and let $0 \leq \nu < \frac{1}{2}$. Then there hold for all $k \geq 0$
    \begin{equation} \label{eq:K0-esti}
        \sum_{m=0}^k \alpha_m^{\nu -1/2} \left(\sum_{j = m}^k \alpha_j^{-1} \right)^{-1/2} \leq K_0(r,\nu)\alpha_{k+1}^{\nu}
    \end{equation} 
    and
    \begin{equation} \label{eq:K1-esti}
        \sum_{m=0}^k \alpha_m^{\nu-1/2} \left(\sum_{j = m}^k \alpha_j^{-1} \right)^{-1} \leq K_1(r,\nu)\alpha_{k+1}^{\nu +1/2}
    \end{equation}
    with $K_0(r,\nu)$ and $K_1(r,\nu)$ defined in \eqref{eq:constants-K}.
\end{lemma}

The next lemmas give useful estimates of a bounded linear operator between Hilbert spaces and generalize the corresponding results in \cite{HankeGroetsch1998}. Their proofs are based on the spectral theory and functional calculus of self-adjoint operators; see, e.g. \cite{Engl1996,Schmudgen2012}.
\begin{lemma}[{\cite[Lem.~2]{Jin2010}}] \label{lem:spectral-nu}
    Let $\{ \alpha_k \}$ be a sequence of positive numbers and let $T: H_1 \to H_2$ be a bounded linear operator between Hilbert spaces. Then, for any $0 \leq \nu \leq 1$ and any integers $0 \leq m \leq l$, there holds
    \begin{equation*}
        \left \| \prod_{j=m}^l \alpha_j \left(\alpha_j I + T^*T \right)^{-1} (T^*T)^{\nu} \right\|_{\Linop(H_1)} \leq \left(\sum_{j=m}^{l} \alpha_j^{-1} \right)^{-\nu} .
    \end{equation*} 
\end{lemma} 
This result can be improved for the specific case $\nu = \frac12$ to be sharp as shown by the choice $m=l$.
\begin{lemma} \label{lem:spectral-half}
    Let $\{ \alpha_k \}$ be a sequence of positive numbers and let $T: H_1 \to H_2$ be a bounded linear operator between Hilbert spaces. Then, for any integers $0 \leq m \leq l$, there holds
    \begin{equation*}
        \left \| \prod_{j=m}^l \alpha_j \left(\alpha_j I + T^*T \right)^{-1} T^* \right\|_{\Linop(H_2, H_1)} \leq \frac{1}{2} \left(\sum_{j=m}^{l} \alpha_j^{-1} \right)^{-1/2} .
    \end{equation*} 
\end{lemma} 
\begin{proof}
    Set $t_0 := \norm{T}_{\Linop(H_1,H_2)}^2$ and define the continuous function 
    \begin{equation*}
        g(t) := \prod_{j=m}^l \alpha_j \left(\alpha_j + t \right)^{-1}, \quad t \geq 0.
    \end{equation*} 
    From the spectral theory of self-adjoint operators, see, e.g. \cite[Chap.~2]{Engl1996}, we have that
    \begin{equation}
        \label{eq:spec-est1}
        \left\|g(T^*T)T^* \right\|_{\Linop(H_2, H_1)} \leq \sup \left\{ \sqrt{t}|g(t)| : 0 \leq t \leq t_0 \right\}.
    \end{equation} 
    On the other hand, a simple computation and the Cauchy–Schwarz inequality give
    \begin{equation*}
        \prod_{j=m}^l \left(\alpha_j + t \right) \geq \prod_{j=m}^l \alpha_j \left(1 + t \sum_{j=m}^l \alpha_j^{-1} \right) \geq 2 \sqrt{t} \prod_{j=m}^l \alpha_j\left(\sum_{j=m}^{l} \alpha_j^{-1} \right)^{1/2},
    \end{equation*}
    which leads to
    \begin{equation*}
        \sqrt{t}|g(t)| \leq \frac{1}{2} \left(\sum_{j=m}^{l} \alpha_j^{-1} \right)^{-1/2}
    \end{equation*}
    for all $t \geq 0$.
    Combining this with \eqref{eq:spec-est1} yields the desired estimate.
\end{proof}
Finally, we have the following direct consequence of \cref{lem:spectral-nu,lem:spectral-half}.
\begin{corollary} \label{lem:spectral-halfnu}
    Let $\{ \alpha_k \}$ be a sequence of positive numbers and let $T: H_1 \to H_2$ be a bounded linear operator between Hilbert spaces. Then, for any $0 \leq \nu \leq \frac{1}{2}$ and any integers $0 \leq m \leq l$, there holds
    \begin{equation*}
        \left \| \prod_{j=m}^l \alpha_j \left(\alpha_j I + T^*T \right)^{-1} (T^*T)^{\nu}T^* \right\|_{\Linop(H_2,H_1)} \leq \left(\sum_{j=m}^{l} \alpha_j^{-1} \right)^{-\nu-1/2} .
    \end{equation*} 
\end{corollary}

\section*{Acknowledgments} 
The authors would like to thank the two anonymous reviewers for their constructive comments which led to notable improvements of the paper.

\bibliographystyle{jnsao}
\bibliography{BouligandLevenbergMarquardt}

\end{document}

%% file: stop-index-zero.tikz
\begin{tikzpicture}[baseline]
    \begin{axis}[
        width=\textwidth,
        xmin=5.e-05, xmax=1.e-02,
        xmode=log,
        x dir=reverse,
        ytick=\empty,		
        ymin=11,
        extra y tick style={grid=major},
        extra y ticks={12,22},
        extra y tick labels={$12$,$22$},
        extra y tick style={DarkBlue},
        legend style={legend pos=north west,draw=none,font=\scriptsize}
        ]
        \addplot [color= DarkBlue,solid,line width=1pt]
            table[row sep=crcr]{%
                1.05642e-02		12\\
                5.29113e-03		13\\
                1.06026e-03		16\\
                5.30585e-04		17\\
                1.05788e-04		20\\
                5.29167e-05 	22\\
            };
        \addlegendentry{$N_\delta$};
        \addplot[%
            domain=5.e-05:1.e-02,
            samples=100,
            color= DarkBlue,dashed
            ]
            {
                (25-8.5*ln(x))/5
            };
        \addlegendentry{$\mathcal{O}(\log\delta)$};
        \addlegendimage{/pgfplots/refstyle=plot_cputime}\addlegendentry{time [s]};
    \end{axis} 
    \begin{axis}[
    	width=\textwidth,
    	axis y line*=right,
    	xmin=5.e-05, xmax=1.e-02,
        ymin=600,
    	axis x line=none,
    	xmode=log,
    	x dir=reverse,
    	ytick=\empty,
    	extra y tick style={grid=major},
    	extra y ticks={615,1136},
    	extra y tick labels={$615$,$1136$},
        extra y tick style={DarkRed},
        set layers,axis background
    	]
    	\addplot[DarkRed, line width=1pt]
	    	table[row sep=crcr]{%
	    		1.05642e-02		614.8189404010773\\
	    		5.29113e-03		658.4350681304932\\
	    		1.06026e-03		808.4647862911224\\
	    		5.30585e-04		866.128849029541\\
	    		1.05788e-04		1018.3742916584015\\
	    		5.29167e-05 	1136.1994893550873\\
	    	}; \label{plot_cputime}
    \end{axis}
\end{tikzpicture}%

%% file: stop-index-zero-landweber.tikz
\begin{tikzpicture}[baseline]
    \begin{axis}[
        width=\textwidth,
        xmin=5.e-05, xmax=1.e-02,
        xmode=log,
        x dir=reverse,
        ytick=\empty,
        extra y tick style={grid=major},
        extra y ticks={3,2634},
        extra y tick labels={$3$,$2634$},
        extra y tick style={DarkBlue},
        legend style={legend pos=north west,draw=none,font=\scriptsize}
        ]
        \addplot [color= DarkBlue,solid,line width=1pt]
            table[row sep=crcr]{%
                1.05642e-02		3\\
                5.29113e-03		6\\
                1.06026e-03		36\\
                5.30585e-04		86\\
                1.05788e-04		921\\
                5.29167e-05 	2634\\					
            };
        \addlegendentry{$N_\delta$};
        \addplot[%
            domain=5.e-05:1.e-02,
            samples=100,
            color=DarkBlue,dashed
            ]
            {
                0.0000100/(x*x)
            };
            \addlegendentry{$\mathcal{O}(\delta^{-2})$};
            \addlegendimage{/pgfplots/refstyle=plot_cputime_BL}\addlegendentry{time [s]};
    \end{axis}
    \begin{axis}[
    	width=\textwidth,
    	axis y line*=right,
    	axis x line=none,
    	xmode=log,
    	x dir=reverse,
    	ytick=\empty,
    	extra y tick style={grid=major},
    	extra y ticks={47,41337},
    	extra y tick labels={$47$,$41337$},
        extra y tick style={DarkRed},
        set layers,axis background
    	]
    	\addplot[DarkRed, line width=1pt]
    		table[row sep=crcr]{%
    			1.05642e-02		46.7875554561615\\
    			5.29113e-03		82.87854647636414\\
    			1.06026e-03		480.6225845813751\\
    			5.30585e-04		1164.9991130828857\\
    			1.05788e-04		14930.66363978386\\
    			5.29167e-05 	41336.55525422096\\
     	}; \label{plot_cputime_BL}
    \end{axis}
\end{tikzpicture}

%% file: stop-index-source.tikz
\begin{tikzpicture}[baseline]
    \begin{axis}[
        width=\textwidth,
        xmin=1.e-07, xmax=1.e-1,
        xmode=log,
        x dir=reverse,
        ymin=10,
        ytick=\empty,		
        extra y tick style={grid=major},
        extra y ticks={11,21},
        extra y tick labels={$11$,$21$},
        extra y tick style={DarkBlue},
        legend style={legend pos=north west,draw=none,font=\scriptsize}
        ]
        \addplot [color=DarkBlue,solid,line width=1pt]
            table[row sep=crcr]{%
                1.05829e-01		11\\
                5.29420e-02		12\\
                1.05976e-02		14\\
                5.29637e-03 	14\\
                1.05863e-03 	15\\
                5.29610e-04 	16\\
                1.06223e-04		16\\
                5.29698e-05		16\\
                1.05916e-05		17\\
                5.29109e-06		17\\
                1.05700e-06		18\\
                5.28392e-07		19\\
                1.05870e-07		21\\
            };
        \addlegendentry{$N_\delta$};
        \addplot[%
            domain=1.e-07:1.e-01,
            samples=100,
            color=DarkBlue,dashed
            ]
            {
                (37-4.7*ln(x))/5
            };
        \addlegendentry{$\mathcal{O}(\log\delta)$};
        \addlegendimage{/pgfplots/refstyle=plot_cputime_source}\addlegendentry{time [s]};
    \end{axis}
    \begin{axis}[
    	width=\textwidth,
    	axis y line*=right,
    	xmin=1.e-07, xmax=1.e-1,
        ymin=500,
    	axis x line=none,
    	xmode=log,
    	x dir=reverse,
    	ytick=\empty,
    	extra y tick style={grid=major},
    	extra y ticks={570,1087},
    	extra y tick labels={$570$,$1087$},
        extra y tick style={DarkRed},
        set layers,axis background
    	]
    	\addplot[DarkRed,line width=1pt]
    		table[row sep=crcr]{%
	    		1.05829e-01		569.8918704986572\\
	    		5.29420e-02		620.6904604434967\\
	    		1.05976e-02		723.7786290645599\\
	    		5.29637e-03 	728.0395710468292\\
	    		1.05863e-03 	782.2865355014801\\
	    		5.29610e-04 	826.6376779079437\\
	    		1.06223e-04		842.3206787109375\\
	    		5.29698e-05		835.6096479892731\\
	    		1.05916e-05		883.9892916679382\\
	    		5.29109e-06		872.0093109607697\\
	    		1.05700e-06		944.7616450786591\\
	    		5.28392e-07		988.5718698501587\\
	    		1.05870e-07		1087.2021219730377\\
    	}; \label{plot_cputime_source}
	\end{axis}
\end{tikzpicture}%

%% file: stop-index-source-landweber.tikz
\begin{tikzpicture}[baseline]
    \begin{axis}[
        width=\textwidth,
        xmin=1.e-07, xmax=1.e-1,
        xmode=log,
        x dir=reverse,
        ytick=\empty,
        extra y tick style={grid=major},
        extra y ticks={2,1105},
        extra y tick labels={$2$,$1105$},
        extra y tick style={DarkBlue},
        legend style={legend pos=north west,draw=none,font=\scriptsize}
        ]
        \addplot [color= DarkBlue,solid,line width=1pt]
            table[row sep=crcr]{%
                1.05829e-01		2\\
                5.29420e-02		4\\
                1.05976e-02		9\\
                5.29637e-03 	11\\
                1.05863e-03 	16\\
                5.29610e-04 	18\\
                1.06223e-04		22\\
                5.29698e-05		25\\
                1.05916e-05		42\\
                5.29109e-06		55\\
                1.05700e-06		182\\
                5.28392e-07		271\\
                1.05870e-07		1105\\
            };
        \addlegendentry{$N_\delta$};
        \addplot[%
            domain=1.e-07:1.e-01,
            samples=100,
            color=DarkBlue,dashed
            ]
            {
                0.00015*1/x
            };
        \addlegendentry{$\mathcal{O}(\delta^{-1})$};
        \addlegendimage{/pgfplots/refstyle=plot_cputime_source_BL}\addlegendentry{time [s]};
    \end{axis}
    \begin{axis}[
    	width=\textwidth,
    	axis y line*=right,
    	xmin=1.e-07, xmax=1.e-1,
    	axis x line=none,
    	xmode=log,
    	x dir=reverse,
    	ytick=\empty,
    	extra y tick style={grid=major},
    	extra y ticks={40,18054},
    	extra y tick labels={$40$,$18054$},
        extra y tick style={DarkRed},
        set layers,axis background
    	]
    	\addplot[DarkRed, line width=1pt]
    		table[row sep=crcr]{%
    			1.05829e-01		40.15858340263367\\
    			5.29420e-02		62.981141328811646\\
    			1.05976e-02		138.713928937912 \\
    			5.29637e-03 	165.66338896751404\\
    			1.05863e-03 	241.356351852417\\
    			5.29610e-04 	269.1167061328888\\
    			1.06223e-04		351.0504319667816\\
    			5.29698e-05		379.83857440948486\\
    			1.05916e-05		670.0487039089203\\
    			5.29109e-06		866.593493938446\\
    			1.05700e-06		2965.012938976288\\
    			5.28392e-07		4360.244261264801\\
    			1.05870e-07		18054.36757493019\\
    	}; \label{plot_cputime_source_BL}
    \end{axis}
\end{tikzpicture}